\documentclass[twoside,a4paper,reqno,11pt]{amsart} 
\usepackage[top=28mm,right=28mm,bottom=28mm,left=28mm]{geometry}

\makeatletter
\def\@tocline#1#2#3#4#5#6#7{\relax
  \ifnum #1>\c@tocdepth 
  \else
    \par \addpenalty\@secpenalty\addvspace{#2}%
    \begingroup \hyphenpenalty\@M
    \@ifempty{#4}{%
      \@tempdima\csname r@tocindent\number#1\endcsname\relax
    }{%
      \@tempdima#4\relax
    }%
    \parindent\z@ \leftskip#3\relax \advance\leftskip\@tempdima\relax
    \rightskip\@pnumwidth plus4em \parfillskip-\@pnumwidth
    #5\leavevmode\hskip-\@tempdima
      \ifcase #1
       \or\or \hskip 1em \or \hskip 2em \else \hskip 3em \fi%
      #6\nobreak\relax
    \hfill\hbox to\@pnumwidth{\@tocpagenum{#7}}\par
    \nobreak
    \endgroup
  \fi}
\makeatother

\usepackage{amsfonts, amsmath, amssymb, mathrsfs, bm, latexsym, stmaryrd, array, hyperref, mathtools, bold-extra,colortbl}

\renewcommand{\a}{\alpha}
\renewcommand{\b}{\beta}

\newcommand{\s}{\sigma}

\renewcommand{\O}{\Omega}

\newcommand{\normeq}{\trianglelefteqslant}

\newcommand{\<}{\langle}
\renewcommand{\>}{\rangle}
\newcommand{\la}{\langle}
\newcommand{\ra}{\rangle}

\renewcommand{\to}{\rightarrow}

\newcommand{\leqs}{\leqslant}
\newcommand{\geqs}{\geqslant}

\newcommand{\vs}{\vspace{2mm}}
\newcommand{\fpr}{\mbox{{\rm fpr}}}

\newcommand{\what}{\widehat}

\makeatletter
\newcommand{\imod}[1]{\allowbreak\mkern4mu({\operator@font mod}\,\,#1)}
\makeatother

\newtheorem{theorem}{Theorem} 
\newtheorem{conje}{Conjecture}

\newtheorem{thm}{Theorem}[section] 
\newtheorem{prop}[thm]{Proposition} 
\newtheorem{lem}[thm]{Lemma}
\newtheorem{cor}[thm]{Corollary}

\theoremstyle{definition}
\newtheorem{rem}[thm]{Remark}
\newtheorem{remk}{Remark}

\newtheorem*{deff}{Definition}
\newtheorem{defn}[thm]{Definition}

\begin{document}

\title[The regularity number of a finite group]{On the regularity number of a finite group \\ and other base-related invariants}

\author{Marina Anagnostopoulou-Merkouri}
\address{M. Anagnostopoulou-Merkouri, School of Mathematics, University of Bristol, Bristol BS8 1UG, UK}
\email{marina.anagnostopoulou-merkouri@bristol.ac.uk}

\author{Timothy C. Burness}
\address{T.C. Burness, School of Mathematics, University of Bristol, Bristol BS8 1UG, UK}
\email{t.burness@bristol.ac.uk}

\begin{abstract}
A $k$-tuple $(H_1, \ldots, H_k)$ of core-free subgroups of a finite group $G$ is said to be 
regular if $G$ has a regular orbit on the Cartesian product $G/H_1 \times \cdots \times G/H_k$. The regularity number of $G$, denoted $R(G)$, is the smallest positive integer $k$ with the property that every such $k$-tuple is regular. In this paper, we develop some general methods for studying the regularity of subgroup tuples in arbitrary finite groups, and we determine the precise regularity number of all almost simple groups with an alternating or sporadic socle. For example, we prove that $R(S_n) = n-1$ and $R(A_n) = n-2$. We also formulate and investigate natural generalisations of several well-studied problems on base sizes for finite permutation groups, including conjectures due to Cameron, Pyber and Vdovin. For instance, we extend earlier work of Burness, O'Brien and Wilson by proving that $R(G) \leqs 7$ for every almost simple sporadic group, with equality if and only if $G$ is the Mathieu group ${\rm M}_{24}$. We also show that every triple of soluble subgroups in an almost simple sporadic group is regular, which generalises recent work of Burness on base sizes for transitive actions of sporadic groups with soluble point stabilisers.
\end{abstract}

\date{\today}

\maketitle

\setcounter{tocdepth}{2}
\tableofcontents

\section{Introduction}\label{s:intro}

Let $G$ be a finite group and let $H_1, \ldots, H_k$ be a collection of core-free subgroups of $G$, allowing repetitions. Consider the natural action of $G$ on the Cartesian product 
\[
X = G/H_1 \times \cdots \times G/H_k
\]
and observe that $G$ has a regular orbit on $X$ if and only if  
\[
\bigcap_{i=1}^k H_i^{g_i} = 1
\]
for some elements $g_i \in G$. We introduce the following definition.

\begin{deff}
A $k$-tuple $(H_1, \ldots, H_k)$ of core-free subgroups of $G$ is \emph{regular} if $G$ has a regular orbit on $X$, and \emph{non-regular} otherwise. 
\end{deff}

For $k = 2$ or $3$, we will refer to regular and non-regular pairs or triples, respectively. Clearly, the regularity of a given tuple is independent of the ordering of the subgroups in the tuple. In addition, $(H_1, \ldots, H_k)$ is regular if and only if $(H_1^{g_1}, \ldots, H_k^{g_k})$ is regular for some $g_i \in G$, so we are free to replace each $H_i$ by any conjugate. Given some group-theoretic property $\mathcal{P}$ of a subgroup of $G$, such as maximal, soluble or nilpotent, we will say that $(H_1, \ldots, H_k)$ has property $\mathcal{P}$ if every component subgroup in the tuple has this property (we will always assume that the subgroups in any given tuple are core-free). It will also be convenient to say that such a $k$-tuple is \emph{conjugate} if each $H_i$ is conjugate to $H_1$. 

Recall that if $G \leqs {\rm Sym}(\O)$ is a permutation group on a set $\O$, then a subset of $\O$ is a \emph{base} if its pointwise stabiliser in $G$ is trivial. So a conjugate tuple $(H_1, \ldots, H_k)$ is regular if and only if $G$ has a base of size $k$ with respect to its natural action on $G/H_1$. Of course, the latter is equivalent to the bound $k \geqs b(G,H_1)$, where $b(G,H_1)$ is the \emph{base size} of $G$, which is the minimal size of a base for the action of $G$ on $G/H_1$. There is a very substantial literature on bases for finite permutation groups stretching all the way back to the nineteenth century, finding an extensive range of applications and connections to other areas of group theory and combinatorics (we refer the reader to the survey articles \cite{BC,LSh3,Maroti} and \cite[Section 5]{Bur181} for further details). In view of this connection, we introduce the following invariants.

\begin{deff}
Let $G$ be a finite group.

\vspace{1mm}

\begin{itemize}\addtolength{\itemsep}{0.3\baselineskip}
\item[{\rm (i)}] The \emph{regularity number} of $G$, denoted $R(G)$, is the minimal positive integer $k$ such that every $k$-tuple of core-free subgroups of $G$ is regular.
\item[{\rm (ii)}] The \emph{base number} of $G$, denoted $B(G)$, is the minimal positive integer $k$ such that $b(G,H) \leqs k$ for every core-free subgroup $H$ of $G$.
\end{itemize}
\end{deff}

\begin{remk}\label{r:main1}
Let us record some immediate remarks on these definitions:

\vspace{1mm}

\begin{itemize}\addtolength{\itemsep}{0.3\baselineskip}
\item[{\rm (a)}] Clearly, we have $B(G) \leqs R(G)$. It turns out that there are examples where this  inequality is strict. For instance, if $G$ is the sporadic simple group ${\rm M}_{11}$, then the main theorem of \cite{BOW} gives $B(G) = 4$, but we find that $R(G)=5$. Indeed, $G$ has maximal subgroups $H = {\rm M}_{10} = A_6.2$ and $K = {\rm L}_{2}(11)$, and one can check that the $4$-tuple $(H,H,K,K)$ is non-regular. In fact, we will show that there are infinitely many finite simple groups with $R(G) \geqs B(G)+1$ (see Proposition \ref{p:simple}). Moreover, in Proposition \ref{p:simple2} we prove that if $G = {\rm GL}_n(2)$ with $n \geqs 5$, then $B(G) = n$ and $R(G) \geqs 2(n-1)$, so the difference $R(G) - B(G)$ can be arbitrarily large.

\item[{\rm (b)}] Notice that if $H$ and $K$ are subgroups of $G$, with $H$ nontrivial and $K$ core-free, then $H \cap K^g < H$ for some $g \in G$. As a consequence, it follows that $R(G) \leqs \ell(G)$, where $\ell(G)$ is the \emph{length} of $G$ (recall that this is the length of the longest chain $G=G_0 > G_1 > \cdots > G_{\ell} = 1$ of subgroups of $G$). In particular, $R(G) \leqs \log_2|G|$.

\item[{\rm (c)}] In \cite{Cameron2014}, Cameron introduces several base-related invariants of a finite group $G$. Let us define $b_2(G)$ to be the maximum, over all faithful permutation representations of $G$, of the maximal size of a minimal base (recall that a base is \emph{minimal} if no proper subset is a base). Since $B(G)$ is the minimal size of a minimal base over all faithful \emph{transitive} permutation representations of $G$, we have $B(G) \leqs b_2(G)$. In 
\cite[Corollary 3.3]{Cameron2014}, Cameron proves that $b_2(G) \leqs \sigma(G)$ (also see Proposition \ref{p:indep}), where $\sigma(G)$ is the maximal size of an \emph{independent subset} of $G$ (that is, $\s(G)$ is the maximal size of a subset $S$ of $G$ such that $x \not\in \la S \setminus \{x\}\ra$ for all $x \in S$). By combining this result with a theorem of Whiston \cite{Whiston} on independent subsets of symmetric groups, it follows that  
\[
B(G) = b_2(G) = \sigma(G) = n-1
\]
for $G = S_n$ (see Corollary \ref{c:sn}).

\item[{\rm (d)}] Two related invariants of a finite group were studied in \cite{BGL1,BGL2}. The \emph{intersection number} of $G$ is defined to be the minimal number of maximal subgroups of $G$ whose intersection coincides with the Frattini subgroup of $G$. And for a group with trivial Frattini subgroup, the \emph{base number} was defined in \cite{BGL2} to be the minimal base size over all faithful primitive permutation representations of $G$. The main results in \cite{BGL1,BGL2} include sharp upper bounds on these invariants for almost simple groups.
\end{itemize}
\end{remk}

A wide range of related invariants can be defined in this setting. Indeed, if $\mathcal{P}$ is a group-theoretic property satisfied by certain core-free subgroups of $G$ (such as maximality, solubility, nilpotency, etc.) then we can define $R_{\mathcal{P}}(G)$ and $B_{\mathcal{P}}(G)$ by restricting to tuples comprising subgroups with property $\mathcal{P}$. The flexibility of this set-up allows us to succinctly describe a number of well-studied problems concerning bases for transitive groups, leading to natural extensions and open problems in the more general regularity setting. Let us highlight three main examples:

\vs

\noindent \emph{Vdovin's conjecture.} 
Let $G$ be a finite group with trivial soluble radical and let $B_{{\rm sol}}(G)$ be the maximal base size $b(G,H)$ over all core-free soluble subgroups $H$ of $G$. Then a conjecture of Vdovin (see \cite[Problem 17.41(b)]{Kou}) asserts that $B_{{\rm sol}}(G) \leqs 5$, which is a strengthening of an earlier conjecture of Babai, Goodman and Pyber (see Conjecture 6.6 in \cite{BGP}). The example $G = S_8$ with $H = S_4 \wr S_2$ shows that the bound in Vdovin's conjecture would be best possible (in fact, there are infinitely many examples with $b(G,H) = 5$). Although Vdovin's conjecture remains open, there has been some recent progress towards a positive solution. For example, the main theorem of \cite{B21} establishes the bound $B_{{\rm sol \, max}}(G) \leqs 5$ with respect to soluble maximal subgroups, and Vdovin \cite{Vdovin} has reduced the conjecture to almost simple groups (recall that a finite group $G$ is \emph{almost simple} if there exists a non-abelian simple group $T$ such that $T \normeq G \leqs {\rm Aut}(T)$; here $T$ is the \emph{socle} of $G$). In the latter setting, the conjecture for groups with an alternating or sporadic socle is resolved in \cite{Bay2, B23}, but the general problem for groups of Lie type remains open. 

\vs

\noindent \emph{Cameron's conjecture.}
In a different direction, a highly influential conjecture of Cameron and Kantor \cite{CK} from 1993 asserts that there exists an absolute constant $c$ such that $B_{{\rm ns}}(G) \leqs c$ for every almost simple group $G$. Here $B_{{\rm ns}}(G)$ denotes the maximal base size $b(G,H)$ over all \emph{non-standard} maximal subgroups $H$ of $G$ (roughly speaking, a maximal subgroup $H$ is \emph{standard} if $G$ is either a classical group and $H$ acts reducibly on the natural module, or $G = S_n$ or $A_n$ and $H$ acts intransitively or imprimitively on $\{1, \ldots, n\}$, otherwise $H$ is \emph{non-standard}). This conjecture was proved by Liebeck and Shalev \cite{LSh2} (with an undetermined constant) using a powerful probabilistic approach for studying bases. In response, Cameron \cite[p.122]{CamPG} conjectured that $7$ is the optimal bound, with $B_{{\rm ns}}(G)= 7$ if and only if $G = {\rm M}_{24}$. Cameron's conjecture was proved in a series of papers by Burness et al. \cite{B07,BGS,BLS,BOW}.

\vs

\noindent \emph{Pyber's conjecture.} 
Another intensively studied conjecture on base sizes was proposed by Pyber in the early 1990s \cite[p.207]{Pyber}. Before stating the conjecture, first observe that if $G$ is a finite transitive permutation group of degree $n$ with point stabiliser $H$, then 
\[
b(G,H) \geqs \log_{n}|G|.
\]
Pyber's conjecture asserts that all primitive groups admit small bases in the sense that there exists an absolute constant $c$ such that $b(G,H) \leqs c\log_{n}|G|$ for every finite primitive group $G$ of degree $n$. Building on earlier work by several authors (see \cite{BS,LSh14,Seress}, for example), the proof of Pyber's conjecture was completed by Duyan et al. in \cite{DHM}. This was extended in \cite{HaLM}, where the main theorem gives the explicit bound
\[
b(G,H) \leqs 2\log_{n}|G| + 24.
\]

Turning to the base number, suppose $G$ is a finite group with a core-free maximal subgroup and let $m(G)$ be the minimal index of such a subgroup (in other words, $m(G)$ is the minimal degree of a faithful primitive permutation representation of $G$). Then 
\[
B_{{\rm max}}(G) \geqs \log_{m(G)}|G|,
\]
where $B_{{\rm max}}(G)$ is the maximal base size of $G$ over all primitive faithful permutation representations (and we can define $R_{{\rm max}}(G)$ with respect to tuples of core-free maximal subgroups of $G$). And the positive solution to Pyber's conjecture implies that $B_{{\rm max}}(G) \leqs c\log_{m(G)}|G|$ for some absolute constant $c$.  

Here we propose natural extensions of the above base size conjectures of Pyber, Cameron and Vdovin in the more general regularity setting.

\begin{conje}\label{con:main}
\mbox{ }\vspace{1mm}
\begin{itemize}\addtolength{\itemsep}{0.3\baselineskip}
\item[{\rm (i)}] There exists an absolute constant $c$ such that 
\[
R_{{\rm max}}(G) \leqs c\log_{m(G)}|G|
\]
for every finite group $G$ with a core-free maximal subgroup.
\item[{\rm (ii)}] We have $R_{\rm ns}(G) \leqs 7$ for every finite almost simple group $G$, with equality if and only if $G = {\rm M}_{24}$. 
\item[{\rm (iii)}] We have $R_{\rm sol}(G) \leqs 5$ for every finite group $G$ with trivial soluble radical.
\end{itemize}
\end{conje}

In this paper, we develop some general techniques for studying the regularity number of a finite group and we take the first steps towards establishing the above conjectures for almost simple groups with socle an alternating or sporadic group. These results will be extended to groups of Lie type in a follow up paper by Anagnostopoulou-Merkouri. We now state our main results, first focussing on symmetric and alternating groups. 

Let $G$ be an almost simple group with socle $T = A_n$. We say that a core-free subgroup $H$ of $G$ is \emph{primitive} if $H \cap T$ acts primitively on the set $\{1, \ldots, n\}$, and we define $B_{{\rm prim}}(G)$ and $R_{{\rm prim}}(G)$ accordingly. Similarly, we define $R_{{\rm intrans}}(G)$ and $R_{{\rm sol \, max}}(G)$ in terms of tuples of intransitive and soluble maximal subgroups of $G$, respectively. (Of course, if $G$ does not have a core-free primitive subgroup, then  $B_{{\rm prim}}(G)$ and $R_{{\rm prim}}(G)$ are undefined, and similarly for $R_{{\rm sol \, max}}(G)$. For example, no proper subgroup of $A_{34}$ acts primitively on $\{1, \ldots, 34\}$.)

\begin{theorem}\label{t:main1}
Let $G$ be an almost simple group with socle $T = A_n$.

\vspace{1mm}

\begin{itemize}\addtolength{\itemsep}{0.3\baselineskip}
\item[{\rm (i)}] For $G \in \{S_n,A_n\}$ we have $R(G) = R_{{\rm intrans}}(G) = B(G) = n-|S_n:G|$.
\item[{\rm (ii)}] Suppose $G$ has a core-free primitive subgroup. Then $R_{\rm prim}(G) \leqs 6$, with equality if and only if $G = A_8$. Moreover, if $n\geqs 13$, then $R_{\rm prim}(G) = 2$.
\item[{\rm (iii)}] Suppose $G$ has a soluble maximal subgroup. Then $R_{{\rm sol\, max}}(G) \leqs 5$, with equality if and only if $G = S_8$. Moreover, if $n \geqs 17$, then $R_{{\rm sol\, max}}(G) = 2$.
\end{itemize}
\end{theorem}

\begin{remk}\label{r:main2}
Let us record some comments on the statement of Theorem \ref{t:main1}.

\vspace{1mm}

\begin{itemize}\addtolength{\itemsep}{0.3\baselineskip}
\item[{\rm (a)}] For $T = A_6$ we note that $R(G) = B(G) = 4$ if $G = {\rm PGL}_2(9)$, ${\rm M}_{10}$ or $A_6.2^2$.

\item[{\rm (b)}] Suppose $G = S_n$. Clearly, if $H = S_{n-1}$ is intransitive then the $(n-2)$-tuple $(H, \ldots, H)$ is non-regular since the intersection of any $n-2$ conjugates of $H$ will contain a transposition. More generally, we can determine all the maximal intransitive non-regular $(n-2)$-tuples for $S_n$ (see Proposition \ref{p:snmax}); up to ordering and conjugacy, they are of the form $(H, \ldots, H, K)$, where $H = S_{n-1}$ and $K = S_k \times S_{n-k}$ for some $1 \leqs k \leqs n/2$. We refer the reader to Remark \ref{r:snmax} for further comments.

\item[{\rm (c)}] Notice that the natural actions of $S_n$ and $A_n$ on $\{1, \ldots, n\}$ immediately give the lower bounds $B(S_n) \geqs n-1$ and $B(A_n) \geqs n-2$. So as a corollary to Theorem \ref{t:main1}, we deduce that $B(S_n) =n-1$ and $B(A_n) = n-2$. Following Cameron \cite{Cameron2014}, we can also show that $B(S_n) =n-1$ by combining the bound $B(S_n) \leqs \sigma(S_n)$ with Whiston's theorem on independent subsets in \cite{Whiston}, as discussed in Remark \ref{r:main1}(c). Similarly, we get $B(A_n) = n-2$ (see Corollary \ref{c:sn}).

\item[{\rm (d)}] We have $m(G) = n$ and thus Theorem \ref{t:main1}(i) yields
\[
R(G) = n-|S_n:G| < 2\log_{m(G)}|G|
\]
for all $G \in \{S_n,A_n\}$ with $n \geqs 5$. In particular, this establishes part (i) of Conjecture \ref{con:main} for all symmetric and alternating groups. 

\item[{\rm (e)}] The bound in (ii) extends the main theorem of \cite{BGS} by establishing a stronger form of Cameron's conjecture (see Conjecture \ref{con:main}(ii)) for all almost simple groups with socle an alternating group. We refer the reader to Table \ref{tab:prim} in Section \ref{ss:prim} for a complete list of the groups $G$ with $R_{{\rm prim}}(G) = k > 2$, together with an example of a non-regular primitive $(k-1)$-tuple. The exact value of $B_{{\rm prim}}(G)$ is determined in \cite{BGS} and we deduce that $R_{{\rm prim}}(G) \ne B_{{\rm prim}}(G)$ if and only if $G = A_8$.

\item[{\rm (f)}] Part (iii) extends the main theorem of \cite{B21} for groups with socle $A_n$ and it establishes a special case of Conjecture \ref{con:main}(iii). We have computed the exact value of $R_{{\rm sol\, max}}(G)$ in all cases and the groups with $R_{{\rm sol\, max}}(G) = k > 2$ are recorded in Table \ref{tab:maxsol} (see Section \ref{ss:prim}) along with an example of a non-regular $(k-1)$-tuple of the form $(H, \ldots, H)$. In particular, we have $R_{{\rm sol \, max}}(G) = B_{{\rm sol \, max}}(G)$ for all $n \geqs 5$.
\end{itemize}
\end{remk}

\vs

In \cite{Bay2}, Baykalov proves that $B_{{\rm sol}}(G) \leqs 5$ for every almost simple group $G$ with socle $A_n$ and a special case of Conjecture \ref{con:main}(iii) asserts that $R_{{\rm sol}}(G) \leqs 5$. With the aid of {\sc Magma} \cite{magma}, we have verified the latter bound for the groups with $n \leqs 11$ and we find that $R_{{\rm sol}}(G) = 5$ if and only if $G = S_8$ (and up to conjugacy, the only non-regular soluble $4$-tuple is $(H,H,H,H)$ with $H = S_4 \wr S_2$). We speculate that the following much stronger conjecture holds in this setting.

\begin{conje}
We have $R_{{\rm sol}}(S_n) = R_{{\rm sol}}(A_n) = 2$ for all sufficiently large $n$. 
\end{conje}

As recorded in Table \ref{tab:maxsol}, we have $R_{{\rm sol \, max}}(G) = 3$ for $G = S_{16}$ since $b(G,H) = 3$ when $H = S_4 \wr S_4$. More generally, if $G = S_{16+m}$ and $H = (S_4 \wr S_4) \times S_m < G$ with $m \leqs 4$, then with the aid of {\sc Magma} \cite{magma} we can show that $b(G,H) = 3$. We also note that $b(G,H) = 2$ when $G = S_{21}$ and $H = (S_4 \wr S_4) \times S_4$, which leads us to speculate that $R_{{\rm sol}}(S_n) = R_{{\rm sol}}(A_n) = 2$ for all $n \geqs 21$. In the context of the above conjecture, it is also worth recalling a theorem of Dixon \cite{Dixon67}, which states that $|H| \leqs 24^{(n-1)/3}$ for every soluble subgroup $H$ of $G = S_n$ (with equality if $n = 4^d$ and $H = S_4 \wr S_4 \wr \cdots \wr S_4$ is the iterated wreath product of $d$ copies of $S_4$). In particular, Dixon's bound implies that $|H|^3 < |G|$ for all $n \geqs 59$.

\vs

For nilpotent subgroups, Zenkov \cite{Zen21} uses the Classification of Finite Simple Groups to prove that $R_{{\rm nilp}}(G) \leqs 3$ for every finite group $G$ with trivial Fitting subgroup. This bound is best possible (for instance, if $G = S_8$ and $H$ is a Sylow $2$-subgroup, then $b(G,H) = 3$) and interest in this problem can be traced all the way back to work of Passman in the 1960s. For example, in \cite{Pass}, Passman proves that if $G$ is a finite $p$-soluble group and $P$ is a Sylow $p$-subgroup, then there exist $x,y \in G$ such that $P \cap P^x \cap P^y = O_p(G)$ is the $p$-core of $G$. 

In \cite{Zenkov14}, Zenkov proves that every nilpotent pair of subgroups of $S_n$ or $A_n$ (for $n \geqs 5$) is regular, with the single exception of the symmetric group $S_8$. Computations in the low degree groups lead us to propose the following conjecture.

\begin{conje}
If $G = S_n$ or $A_n$ with $n \geqs 13$, then every pair of subgroups $(H,K)$ is regular when $H$ is nilpotent and $K$ is soluble.
\end{conje}

Note that $G = S_{12}$ has a non-regular pair $(H,K)$, where $H$ is a Sylow $2$-subgroup and $K = S_4 \wr S_3$, so the conjectured bound $n \geqs 13$ would be best possible.

\vs

Now let us state our main result for almost simple sporadic groups, which establishes a strong form of Conjecture \ref{con:main} in this setting (in part (ii), $c = 1$ or $2$).

\begin{theorem}\label{t:main2}
Let $G$ be an almost simple sporadic group with socle $T$.

\vspace{1mm}

\begin{itemize}\addtolength{\itemsep}{0.3\baselineskip}
\item[{\rm (i)}] The exact values of $B(G)$ and $R(G)$ are recorded in Table \ref{tab:RG_sporadic}. In particular, $R(G) \leqs 7$, with equality if and only if $G = {\rm M}_{24}$.
\item[{\rm (ii)}] We have $R_{{\rm sol}}(G) \leqs 3$, with equality if 
\[
G = {\rm M}_{11}, \, {\rm M}_{12}.c, \, {\rm M}_{22}.c, \, {\rm M}_{23}, \, {\rm M}_{24}, \, {\rm J}_2.c, \, {\rm HS}.2, \, {\rm Co}_2, \, {\rm Fi}_{22}.c \mbox{ or } {\rm Fi}_{23}.
\]
\item[{\rm (iii)}] Suppose $G$ has a soluble maximal subgroup. Then $R_{{\rm sol \, max}}(G) \leqs 3$, with equality if and only if $T = {\rm M}_{11}$, ${\rm M}_{12}$, ${\rm J}_2$, ${\rm Fi}_{22}$ or ${\rm Fi}_{23}$.
\end{itemize}
\end{theorem}

\begin{remk}\label{r:main3}
Some comments on the statement of Theorem \ref{t:main2} are in order.

\vspace{1mm}

\begin{itemize}\addtolength{\itemsep}{0.3\baselineskip}
\item[{\rm (a)}] In part (i), the base number of every almost sporadic group $G$ can be read off from the main theorem of \cite{BOW}, which gives $b(G,H)$ for every core-free maximal subgroup $H$. As an immediate consequence of part (i), we observe that $R(G) \leqs B(G)+1$, with equality if and only if $G = {\rm M}_{11} $ or ${\rm M}_{12}$. 

\item[{\rm (b)}] In Table \ref{tab:RG_sporadic} (see Section \ref{ss:spor_reg}) we list all the \emph{large} non-regular $k$-tuples with $k = R(G)-1$, up to conjugacy and ordering (large tuples are defined in Definition \ref{d:large}), with the exception of the Baby Monster and the Monster (see Remark \ref{r:RG_spor} for comments on the latter two cases). It follows that  the only non-regular $6$-tuple arises when $G={\rm M}_{24}$ and every component is conjugate to $H = {\rm M}_{23}$.

\item[{\rm (c)}] A more detailed result on $R_{\rm sol}(G)$ is given in Proposition \ref{p:nonregsol}. In particular, we have computed $R_{\rm sol}(G)$ precisely unless $T$ is one of the following groups:
\[
{\rm Co}_{1}, \, {\rm HN}, \, {\rm J}_{4}, \, {\rm Ly}, \, {\rm Th}, \, {\rm Fi}_{24}', \, \mathbb{B}, \, \mathbb{M}.
\]
In addition, for some of the groups with $R_{\rm sol}(G) = 3$ we present a non-regular soluble pair in Table \ref{tab:spor_sol_pairs}. The complete list of non-regular soluble maximal pairs is recorded in Table \ref{table-sporadics} (see Proposition \ref{p:spor_maxsol}).

\item[{\rm (d)}] In part (iii), note that every almost simple sporadic group $G$ has a soluble maximal subgroup unless $G = {\rm M}_{12}$, ${\rm M}_{12}.2$, ${\rm M}_{24}$ or ${\rm HS}$. For the latter groups, $R_{{\rm sol\, max}}(G)$ is not defined.

\item[{\rm (e)}]  We refer the reader to Remark \ref{r:sol_nip} for some comments on the existence of non-regular pairs $(H,K)$ with $H$ nilpotent and $K$ soluble. It is worth noting that we are only aware of examples for the groups ${\rm M}_{22}.2$ and ${\rm J}_2.2$. 
\end{itemize}
\end{remk}

\vs

We conclude this introduction by briefly commenting on the proofs of Theorems \ref{t:main1} and \ref{t:main2}, which involve a combination of probabilistic, combinatorial and computational methods. 

Let $G$ be a finite group, let $\tau = (H_1, \ldots, H_k)$ be a $k$-tuple of core-free subgroups and let  
\[
Q(G,\tau) = \frac{|\{(\alpha_1, \ldots, \alpha_k) \in \Gamma \,:\, \bigcap_{i = 1}^{k}G_{\alpha_i}\neq 1\}|}{|\Gamma|}
\]
be the probability that a randomly chosen tuple in $\Gamma = G/H_1 \times \cdots \times G/H_k$ is not contained in a regular orbit of $G$, with respect to the uniform distribution on $\Gamma$. Notice that $\tau$ is regular if and only if $Q(G, \tau) < 1$. 

By adapting earlier work of Liebeck and Shalev \cite{LSh2} on base sizes, we can estimate $Q(G,\tau)$ in terms of fixed point ratios (see Lemma \ref{l:fpr}) and this provides a powerful tool for establishing the regularity of $\tau$. Indeed, this approach is at the heart of our proof of parts (ii) and (iii) in Theorem \ref{t:main1}, where we can appeal to earlier work of Guralnick and Magaard \cite{GM} and Mar\'{o}ti \cite{Maroti02} in order to derive  appropriate bounds on the relevant fixed point ratios. It is also a key ingredient in our proof of Theorem \ref{t:main2}, especially in our analysis of some of the larger sporadic groups. For example, if $G = \mathbb{M}$ is the Monster then we can use the \textsf{GAP} Character Table Library \cite{GAPCTL} to show that $Q(G,\tau)<1$ for every maximal triple $\tau$, and by combining this with the main theorem of \cite{BOW} we deduce that $R(\mathbb{M}) = 3$ and $R_{\rm sol}(\mathbb{M}) \leqs 3$.

In stark contrast, our proof of part (i) of Theorem \ref{t:main1} for $G = S_n$ or $A_n$ is essentially constructive, and it relies on a highly combinatorial analysis of the stabilisers of subsets and partitions of $\{1, \ldots, n\}$. Here one of the key results is Lemma \ref{lem:2sets}, which shows that if $H_1$ and $H_2$ are the stabilisers in $G$ of uniform partitions, then we can typically find an element $g \in G$ such that $H_1 \cap H_2^g$ fixes any given $2$-element subset of $\{1, \ldots, n\}$.

Our proof of Theorem \ref{t:main2} on sporadic groups is entirely computational, working with {\sc Magma} \cite{magma} and \textsf{GAP} \cite{GAP}. For example, if $\tau = (H_1, \ldots, H_k)$ is a maximal $k$-tuple then we can often use {\sc Magma} to construct $G$ and each $H_i$ as subgroups of some symmetric group $S_n$ (typically where $n = m(G)$ is the minimal degree of a faithful primitive permutation representation of $G$) and then use random search to identify elements $g_i \in G$ with $\bigcap_i H_i^{g_i} = 1$. On the other hand, if we are seeking to show that $\tau$ is non-regular, then in many cases we can directly calculate the orbits of $G$ on $G/H_1 \times \cdots \times G/H_k$ in order to confirm the non-existence of a regular orbit. Of course, these computations are not feasible for some of the larger sporadic groups and they require a different approach. For instance, the probabilistic method outlined above is often a very useful tool in this setting, working closely with the detailed information on sporadic groups available in the \textsf{GAP} Character Table Library \cite{GAPCTL} and the Web Atlas \cite{WebAt}.

\vs

\noindent \textbf{Notation.} Our notation is all fairly standard. Let $G$ be a finite group and let $n$ be a positive integer. We will write $C_n$, or just $n$, for a cyclic group of order $n$ and $G^n$ will denote the direct product of $n$ copies of $G$. An unspecified extension of $G$ by a group $H$ will be denoted by $G.H$; if the extension splits then we may write $G{:}H$. We adopt the standard notation for simple groups of Lie type from \cite{KL}.

\vs

\noindent \textbf{Organisation.} We begin in Section \ref{s:prel} by presenting  a number of preliminary results, which we will need for the proofs of our main theorems. For example, the probabilistic technique highlighted above is introduced in Section \ref{ss:prob}, and we provide a brief discussion of some of our main computational methods in Section \ref{ss:comp} (referring the reader to the supplementary file \cite{AB_comp} for more details). Our proof of Theorem \ref{t:main1} for $S_n$ and $A_n$ is given in Section \ref{s:sym}, where the details are divided into two subsections. In Section \ref{ss:prim} we focus on the tuples $(H_1, \ldots, H_k)$, where each $H_i$ acts primitively on $\{1, \ldots, n\}$, using fixed point ratio estimates to prove parts (ii) and (iii) of Theorem \ref{t:main1}. These results feed in to our proof of Theorem \ref{t:main1}(i) in Section \ref{ss:regnumber}, which requires an in-depth analysis of the tuples involving intransitive and imprimitive subgroups. Finally, we present our proof of Theorem \ref{t:main2} in Section \ref{s:spor}, handling parts (i) and (ii) in Sections \ref{ss:spor_reg} and \ref{ss:spor_sol}, respectively.

\vs

\noindent \textbf{Acknowledgements.} First of all, we thank an anonymous referee for their careful reading of the paper and several helpful comments. We thank Thomas Breuer, Derek Holt, J\"{u}rgen M\"{u}ller and Eamonn O'Brien for their generous assistance with some of the computations involving sporadic groups. We also thank Peter Cameron for helpful comments and for drawing our attention to his paper \cite{Cameron2014}. The first author thanks the Heilbronn Institute for Mathematical Research for providing financial support during her doctoral studies at the University of Bristol. 

\section{Preliminaries}\label{s:prel}

In this section, we record several preliminary results that we will need in the proofs of our main theorems. More precisely, in Section \ref{ss:prob} we adapt a probabilistic approach due to Liebeck and Shalev \cite{LSh2} for bounding the base size of a finite permutation group, which gives a useful technique for determining if a given subgroup tuple is regular. In Section \ref{ss:indep}, we highlight a connection between the base and independence numbers of a finite group, which was first observed by Cameron \cite{Cameron2014}, and we use this to compute the base numbers $B(S_n)$ and $B(A_n)$. Next, in Section \ref{ss:almost} we record several results that apply in the special case where $G$ is almost simple. This includes Proposition \ref{p:simple}, which shows that there are infinitely many finite simple groups with $B(G) < R(G)$. Finally, in Section \ref{ss:comp} we briefly describe some of the main computational methods we employ in this paper, with further details provided in the supplementary file \cite{AB_comp}.

\subsection{Probabilistic methods}\label{ss:prob}

Recall that if $G \leqs {\rm Sym}(\O)$ is a transitive permutation group with point stabiliser $H$, then
\[
{\rm fpr}(x,\O) = \frac{|C_{\O}(x)|}{|\O|} = \frac{|x^G \cap H|}{|x^G|}
\]
is the \emph{fixed point ratio} of $x \in G$, where $C_{\O}(x) = \{ \a \in \O \,:\, \a^x = \a\}$ is the set of fixed points of $x$ and $x^G$ is the conjugacy class of $x$ in $G$. Note that $\fpr(x, \O)$ is the probability that $x$ fixes a randomly chosen element of $\O$ (with respect to the uniform distribution on $\O$).

In \cite{LSh2}, Liebeck and Shalev introduced a powerful probabilistic method, based on fixed point ratio estimates, for bounding the base size of a finite permutation group. They used this approach to resolve the Cameron-Kantor base size conjecture from \cite{CK}, which asserts that there is an absolute constant $c$ such that $B_{{\rm ns}}(G) \leqs c$ for every almost simple group $G$ (here $B_{{\rm ns}}(G)$, as defined in Section \ref{s:intro},  is the maximal base size $b(G,H)$ over all \emph{non-standard} maximal subgroups $H$ of $G$). The same approach was at the heart of the proof of Cameron's conjecture in the sequence of papers \cite{B07,BGS,BLS,BOW}, which established the optimal bound $B_{{\rm ns}}(G) \leqs 7$ in this setting. 

The following key lemma is a natural generalisation and it encapsulates our probabilistic approach for studying the regularity of subgroup tuples.

\begin{lem}\label{l:fpr}
Let $G$ be a finite group and let $\tau = (H_1, \ldots, H_k)$ be a core-free tuple of subgroups of $G$. Then $\tau$ is regular if 
\[
\widehat{Q}(G,\tau) := \sum_{i=1}^t|x_i^G| \cdot \left(\prod_{j=1}^k {\rm fpr}(x_i, G/H_j)\right) < 1,
\]
where $x_1, \ldots, x_t$ is a set of representatives of the conjugacy classes in $G$ of elements of prime order.
\end{lem}

\begin{proof}
Set $\Gamma_i = G/H_i$ and $\Gamma = \Gamma_1\times \cdots \times \Gamma_k$. Let 
\[
Q(G, \tau) = \frac{|\{(\alpha_1, \ldots, \alpha_k) \in \Gamma \,:\, \bigcap_{i = 1}^{k}G_{\alpha_i}\neq 1\}|}{|\Gamma|}
\]
be the probability that a randomly chosen tuple in $\Gamma$ is not contained in a regular orbit. Then $\tau$ is regular if and only if $Q(G, \tau) < 1$, and thus it suffices to show that $Q(G, \tau) \leqs \widehat{Q}(G, \tau)$. 

A $k$-tuple $(\alpha_1, \ldots, \alpha_k) \in \Gamma$ is not in a regular $G$-orbit if and only if there exists an element $x \in G$ of prime order such that $\alpha_i^x = \alpha_i$ for all $i$. Now the probability that an element $x \in G$ fixes a random element of $\Gamma$ is equal to 
\[
\frac{|\{(\alpha_1, \ldots, \alpha_k) \in \Gamma \,:\, \alpha_i^x = \alpha_i \mbox{ for all $i$} \}|}{|\Gamma|} = |\Gamma|^{-1}\prod_{j=1}^k |C_{\Gamma_j}(x)| = \prod_{j = 1}^k \fpr(x, \Gamma_j)
\]
and thus
\[
Q(G, \tau) \leqs \sum_{x \in \mathcal{P}} \left( \prod_{j = 1}^k {\rm fpr}(x, \Gamma_j)\right) = \sum_{i = 1}^t |x_i^G| \cdot \left( \prod_{j = 1}^k {\rm fpr}(x_i, \Gamma_j)\right) = \widehat{Q}(G,\tau),
\]
where $\mathcal{P} = \bigsqcup_{i=1}^t x_i^G$ is the set of prime order elements of $G$. The result follows.
\end{proof}

The following elementary result generalises \cite[Lemma 2.1]{B07}. It provides a useful way to produce an upper bound on the expression $\what{Q}(G,\tau)$ in Lemma \ref{l:fpr}.

\begin{lem}\label{l:favbound}
Let $G$ be a finite group and let $\tau = (H_1, \ldots, H_k)$ be a core-free tuple of subgroups of $G$. Suppose $x_1, \ldots, x_m$ are elements in $G$ such that $|x_i^G| \geqs B$ for all $i$ and $\sum_{i = 1}^m |x_i^G\cap H_j| \leqs A_j$ for all $j$. Then 
\[
\sum_{i = 1}^m |x_i^G| \cdot \left( \prod_{j = 1}^k {\rm fpr}(x_i, G/H_j)\right) \leqs B^{1 - k}\cdot \prod_{j = 1}^kA_j
\]
\end{lem}

\begin{proof}
We have    
\[
\sum_{i = 1}^m |x_i^G| \cdot \left( \prod_{j = 1}^k {\rm fpr}(x_i, G/H_j) \right) =  \sum_{i = 1}^m |x_i^G|^{1 - k} \cdot \left( \prod_{j = 1}^k |x_i^G \cap H_j|\right)\leqs B^{1 - k} \cdot \left( \sum_{i = 1}^m \prod_{j = 1}^k |x_i^G \cap H_j| \right)
\]
and thus
\[
\sum_{i = 1}^m |x_i^G| \cdot \left( \prod_{j = 1}^k {\rm fpr}(x_i, G/H_j) \right) \leqs B^{1 - k} \cdot \left( \prod_{j = 1}^k \sum_{i = 1}^m  |x_i^G \cap H_j| \right) \leqs B^{1 - k} \cdot \prod_{j = 1}^{k} A_j,
\]
as claimed.
\end{proof}

\subsection{Independent sets}\label{ss:indep}

A subset $S$ of a finite group $G$ is \emph{independent} if none of its elements lie in the subgroup generated by the others (that is to say, $x \not\in \la S \setminus \{x\}\ra$ for all $x \in S$). For example, $\{(1,2), (2,3), \ldots, (n-1,n)\}$ is an independent generating set for $G = S_n$. 

We define the \emph{independence number} of $G$, denoted $\s(G)$, to be the maximal size of an independent subset of $G$.

\begin{prop}\label{p:indep}
We have $B(G) \leqs \sigma(G)$ for every finite group $G$.
\end{prop}

\begin{proof}
Suppose $G$ acts faithfully on a set $\O$ and let $\mathcal{B} = \{\a_1, \ldots, \a_b\}$ be a minimal base for $G$ (in the sense that no proper subset of $\mathcal{B}$ is a base). In \cite{Cameron2014}, Cameron reveals a beautiful connection with lattices in order  to establish the bound $b \leqs \sigma(G)$ (see \cite[Proposition 3.2]{Cameron2014}). We can also argue directly as follows. 

For all $i \in \{1, \ldots, b\} = [b]$, let 
\[
H_i = \bigcap_{j \in [b] \setminus \{i\}} G_{\a_j}
\]
and note that the minimality of $\mathcal{B}$ implies that each $H_i$ is nontrivial. Fix a nontrivial element $g_i \in H_i$ and observe that $g_i \not\in G_{\a_i}$. Then the $g_i$ are distinct and we claim that $\{g_1, \ldots, g_b\}$ is an independent subset of $G$. To see this, set $K_i = \la g_1, \ldots, g_{i - 1}, g_{i + 1}, \ldots, g_b \ra$ for all $i \in [b]$. Since $g_j \in G_{\a_i}$ for all $j \ne i$, we have $K_i \leqs G_{\a_i}$ for all $i$. But $g_i \not\in G_{\a_i}$ and thus $g_i \not\in K_i$. This justifies the claim and we conclude that $b \leqs \sigma(G)$. 

In particular, $B(G) \leqs \sigma(G)$, as required. 
\end{proof}

We can now combine this observation with a theorem of Whiston \cite{Whiston} to compute the base numbers of $S_n$ and $A_n$ (also see \cite[Corollary 3.4]{Cameron2014}).

\begin{cor}\label{c:sn}
For $n \geqs 5$ we have $B(S_n) = n-1$ and $B(A_n) = n-2$.
\end{cor}

\begin{proof}
By considering the natural action of $S_n$ and $A_n$ of degree $n$, we see that $B(S_n) \geqs n-1$ and $B(A_n) \geqs n-2$. We now combine this trivial observation with the main theorem of \cite{Whiston}, which states that every independent subset in $S_n$ has at most $n-1$ elements, with equality only if the subset generates $S_n$ (for example, $\{(1,2), (2,3), \ldots, (n-1,n)\}$ is an independent subset of maximal size). It follows that $\sigma(S_n) = n-1$ and $\sigma(A_n) \leqs n-2$, and we conclude by applying Proposition \ref{p:indep}.  
\end{proof}

\begin{rem}
Let $G = S_n$ or $A_n$, and let $H$ be a core-free maximal subgroup of $G$. Through the combined efforts of several authors, the base size $b(G,H)$ is now known in all cases:

\vspace{1mm}

\begin{itemize}\addtolength{\itemsep}{0.3\baselineskip}
\item[{\rm (a)}] $H$ primitive on $\{1, \ldots, n\}$: See \cite{BGS}
\item[{\rm (b)}] $H$ imprimitive: See \cite{BGL2,JJ,MS}
\item[{\rm (c)}] $H$ intransitive: See \cite{dVRD,H12,MS0}
\end{itemize}
\end{rem}

\subsection{Almost simple groups}\label{ss:almost}

Let $G$ be a finite almost simple group with socle $T$, which means that  
\[
T \normeq G \leqs {\rm Aut}(T),
\]
where we identify the non-abelian simple group $T$ with its group of inner automorphisms. Let us define the regularity and base number of $G$, denoted $R(G)$ and $B(G)$ respectively, as in Section \ref{s:intro}, and recall that $R(G) \geqs B(G)$. We begin by recording the following trivial observation, which applies in the general setting.

\begin{lem}\label{l:triv}
Let $G$ be a finite group and let $\tau = (H_1, \ldots, H_k)$ be a $k$-tuple of core-free subgroups of $G$. Then $\tau$ is non-regular if 
\[
\prod_{i=1}^k |H_i| > |G|^{k-1}.
\]
\end{lem}

\begin{proof}
Clearly, $G$ has a regular orbit on $G/H_1 \times \cdots \times G/H_k$ only if 
\[
|G| \leqs \prod_{i=1}^k |G:H_i|
\]
and the result follows.
\end{proof}

\begin{prop}\label{p:low}
If $G$ is a finite almost simple group, then $B(G) \geqs 3$.
\end{prop}

\begin{proof}
By \cite[Proposition 2.5]{BG23}, $G$ has a core-free subgroup $H$ with $|H|^2 > |G|$ and thus $(H,H)$ is a non-regular pair by Lemma \ref{l:triv}.
\end{proof}

It will be useful to introduce the following definition.

\begin{defn}\label{d:mG}
Let $G$ be an almost simple group with socle $T$. We define $\mathcal{M}(G)$ to be the set of core-free subgroups of $G$ that are maximal in some overgroup of $T$. 
\end{defn}

For example, $H = {\rm AGL}_3(2) \in \mathcal{M}(S_8)$ since it is a maximal subgroup of $A_8$ (and note that $H$ is not contained in a core-free maximal subgroup of $S_8$). This definition allows us to present the following elementary lemma, which tells us that in the almost simple setting, we only need to consider tuples of subgroups in $\mathcal{M}(G)$ in order to bound $R(G)$ from above.

\begin{lem}\label{lem:maxreduction}
Let $G$ be an almost simple group with socle $T$. Then $R(G) \leqs k$ if and only if every $k$-tuple of subgroups in $\mathcal{M}(G)$ is regular.
\end{lem}

\begin{proof}
The forward implication is clear, so let us assume every $k$-tuple of subgroups in $\mathcal{M}(G)$ is regular and let $\tau = (H_1, \ldots, H_k)$ be an arbitrary core-free tuple. Then each $H_i$ is contained in a subgroup $M_i \in \mathcal{M}(G)$ and the regularity of $(M_1, \ldots, M_k)$ immediately implies that $\tau$ is also regular. 
\end{proof}

\begin{rem}\label{r:large}
Notice that $|G:T| \leqs 2$ for every almost simple group $G$ we are considering in Theorems \ref{t:main1} and \ref{t:main2}, with the single exception of the special case $G = A_6.2^2$. If $G = T.2$ and $H<T$ is maximal, then $H$ is contained in a core-free maximal subgroup of $G$ if $N_G(H) \ne H$. So it follows that $R(G) \leqs k$ if and only if every $k$-tuple of subgroups in $\mathcal{M}'(G)$ is regular, where the latter comprises the core-free maximal subgroups of $G$, together with any maximal subgroups $H$ of $T$ with $N_G(H) = H$. This minor refinement of Lemma \ref{lem:maxreduction} will be useful in the proof of Theorem \ref{t:main2} (see the discussion at the start of Section \ref{ss:spor_reg}).
\end{rem}

The following easy observation will be useful when we study soluble tuples for sporadic groups in Section \ref{ss:spor_sol}.

\begin{lem}\label{l:easy}
Let $G$ be a finite group and let $H$ be a subgroup of $G$ with trivial soluble radical and $R_{\rm sol}(H) \leqs k$. Then every core-free soluble $k$-tuple of subgroups of $G$ of the form $\tau = (H_1, \ldots, H_k)$ with $H_1 \leqs H$ is regular.
\end{lem}

\begin{proof}
Consider the $k$-tuple $\tau' = (H_1, H_2 \cap H, \ldots, H_k \cap H)$ of subgroups of $H$. Since $R_{\rm sol}(H) \leqs k$ and each component of $\tau'$ is a core-free soluble subgroup of $H$, it follows that $\tau'$ is regular. Therefore, there exist elements $x_i \in H$ such that 
\[
1 = H_1^{x_1} \cap \bigcap_{i=2}^k (H_i \cap H)^{x_i} = \bigcap_{i=1}^k H_i^{x_i}
\]
and the result follows.
\end{proof}

As noted in Remark \ref{r:main1}(a), there are examples where the trivial bound $B(G) \leqs R(G)$ is a strict inequality. In fact, we can show that there are infinitely many simple groups with this property. 

\begin{prop}\label{p:simple}
There are infinitely many finite simple groups $G$ with $B(G) < R(G)$.
\end{prop}

\begin{proof}
Let $G = {\rm SL}_3(q)$, where $q \geqs 3$ and $q \not\equiv 1 \imod{3}$. Let $V$ be the natural module and let $H = P_1$ be the stabiliser in $G$ of a $1$-dimensional subspace of $V$. Similarly, let $K$ be the stabiliser of a $2$-space. Then $b(G,H) = b(G,K) = 4$ and by applying the main theorem of \cite{B07} we deduce that $B(G) = 4$. So in order to prove the proposition, it suffices to establish the following claim:

\vs

\noindent \textbf{Claim.} \emph{$(H,H,K,K)$ is a non-regular $4$-tuple.}

\vs

To see this, let $\la v_1 \ra$ and $\la v_2 \ra$ be distinct $1$-spaces and choose $v_3$ so that $\b = \{v_1,v_2,v_3\}$ is a basis for $V$. Let $H_i$ be the stabiliser of $\la v_i \ra$ and let $L = H_1 \cap H_2$, so $|L| = q^2(q-1)^2$. Set $\O = G/K$ and identify $\O$ with the set of $2$-dimensional subspaces of $V$. It suffices to show that $b(L,\O) \geqs 3$. Equivalently, we need to show that no point stabiliser $L_{\a}$ (with $\a \in \O$) has a regular orbit on $\O$. 

Set $U_1 = \la v_1,v_2\ra$, $U_2 = \la v_1,v_3 \ra$ and $U_3 = \la v_2,v_3 \ra$. Then these $2$-spaces represent three distinct $L$-orbits, with respective lengths $1,q$ and $q$. In addition, the stabiliser in $L$ of $U_4 = \la v_1+v_2,v_3 \ra$ is the subgroup
\[
L_4 = \left\{\left(\begin{array}{lll}
a & 0 & b \\
0 & a & b \\
0 & 0 & a^{-2}
\end{array}\right) \,:\, a\in \mathbb{F}_{q}^{\times},\, b \in \mathbb{F}_q\right\}
\]
with respect to the basis $\b$, whence the $L$-orbit of $U_4$ has size $q(q-1)$. Since $|\O| = q^2+q+1$, we deduce that $L$ has exactly $4$ orbits on $\O$, with respective stabilisers $L_i = L_{U_i}$ for $i = 1,2,3,4$. 

So to prove the claim, it just remains to show that none of these point stabilisers has a regular orbit on $\O$. This is clear for $i=1,2,3$ since $|L_i| > |\O|$. So let us consider $L_4$. Clearly, the $L_4$-orbits of $U_1$ and $U_4$ both have length $1$, and it is straightforward to show that $U_2$ and $U_3$ lie in distinct $L_4$-orbits of length $q$. So these four orbits have already covered $2q+2$ points in $\O$ and we note that $|\O| - (2q+2) < |L_4|$. Therefore, $L_4$ does not have a regular orbit on $\O$ and the proof of the claim (and proposition) is complete.
\end{proof}

The previous proposition shows that there are infinitely many finite simple groups with $B(G) < R(G)$. This observation is extended in our next result, which demonstrates the existence of simple groups $G$ for which $R(G)-B(G)$ can be arbitrarily large. 

\begin{prop}\label{p:simple2}
Let $G = {\rm L}_n(2)$ with $n \geqs 5$. Then $B(G) = n$ and $R(G) \geqs 2(n-1)$.
\end{prop}

\begin{proof}
First observe that $G = {\rm GL}(V)$, where $V$ is an $n$-dimensional vector space over $\mathbb{F}_2$. 

Let $H$ be a maximal subgroup of $G$. If $H$ acts irreducibly on $V$, then the main theorem of \cite{B07} implies that $b(G,H) \leqs 4$. Now assume $H$ is reducible, which means that we can identify $G/H$ with the set $\O$ of $k$-dimensional subspaces of $V$ for some positive integer $k \leqs n/2$. Clearly, if $k=1$ then $\O$ coincides with the set of nonzero vectors in $V$, in which case a subset of $\O$ is a base if and only if it is a basis and thus $b(G,H) = n$. So in order to prove that $B(G) = n$, it suffices to show that $b(G,H) \leqs n$ for all $2 \leqs k \leqs n/2$. To see this, fix a basis $\{e_1, \ldots, e_n\}$ for $V$ and define the following $k$-dimensional subspaces of $V$:
\begin{align*}
V_1 & = \la e_1,e_2,  \ldots, e_k \ra \\
V_2 & = \la e_2,e_3, \ldots, e_{k+1} \ra \\
\vdots & \\
V_n & = \la e_n, e_1, \ldots, e_{k-1} \ra
\end{align*}
Let $L_i$ be the setwise stabiliser of $V_i$ in $G$ and set $L = \bigcap_i L_i$. Now each  basis vector $e_i$ is contained in exactly $k$ of these subspaces, say $V_{a_1}, \ldots, V_{a_k}$, and one checks that $\bigcap_j V_{a_j} = \la e_i \ra$. Since we are working over $\mathbb{F}_2$, it follows that $L$ fixes each $e_i$ and thus $L = 1$. In other words, $\{V_1, \ldots, V_n\}$ is a base for $G$ and we conclude that $b(G,H) \leqs n$, as required.

Let $H$ and $K$ be the stabilisers of a $1$-space and an $(n-1)$-space, respectively. In order to show that $R(G) \geqs 2(n-1)$, it suffices to prove the following claim.

\vs

\noindent \textbf{Claim.} \emph{The $(2n-3)$-tuple $\tau = (H, \ldots, H, K, \ldots, K)$ comprising $n-1$ copies of $H$ and $n-2$ copies of $K$ is non-regular.}

\vs

To see this, let $g_1, \ldots, g_{n-1}$ be arbitrary elements in $G$ and set $L = \bigcap_{i}H^{g_i}$. Here $H^{g_i}$ is the stabiliser of a nonzero vector $v_i \in V$ and $L$ contains the following elementary abelian $2$-subgroup (with respect to an appropriate basis for $V$)
\[
J = \left\{ \left(\begin{array}{c|c}
I_{n-1} & A \\ \hline
0 & I_1 
\end{array}\right) \,:\, A \in \mathbb{F}_{2}^{n-1} \right\}
\]
Note that $L = J$ if the $v_i$ are linearly independent and we work with a basis of the form $\{v_1, \ldots, v_{n-1},v_{n}\}$ for $V$. 

So to prove the claim, it suffices to show that $b(J,\O) \geqs n-1$, where $J$ is defined as above (with respect to a basis $\{v_1, \ldots, v_{n-1},v_{n}\}$ for $V$) and $\O$ is the set of $(n-1)$-dimensional subspaces of $V$. Notice that if $\tau$ denotes the inverse-transpose graph automorphism of $G$, then we have $b(J,\O) = b(J^{\tau},\Gamma)$, where $\Gamma$ is the set of nonzero vectors in $V$, so we need to show that $b(J^{\tau},\Gamma) \geqs n-1$. 

Let $w_1, \ldots, w_{n-2}$ be arbitrary nonzero vectors in $V$ and write
\[
w_i = b_{i,1}v_1 + \cdots + b_{i,n-1}v_{n-1} + b_{i,n}v_n
\]
for all $i$. Let $B = (b_{i,j})$ be the $(n-2) \times (n-1)$ matrix corresponding to the coefficients of $v_j$ for $j < n$. Then the pointwise stabiliser of the $w_i$ in $J^{\tau}$ is the set of matrices of the form 
\[
\left(\begin{array}{c|c}
I_{n-1} & 0 \\ \hline
A & I_1 
\end{array}\right) \in J^{\tau}
\]
where $A = (a_1, \ldots, a_{n-1}) \in \mathbb{F}_{2}^{n-1}$ and $BA^T = 0$. Since the kernel of the linear map $\mathbb{F}_{2}^{n-1} \to \mathbb{F}_{2}^{n-2}$ corresponding to $B$ is obviously nonzero, it follows that the pointwise stabiliser of the $w_i$ in $J^\tau$ is nontrivial. In turn, this implies that $b(J^{\tau},\Gamma)  = b(J,\O) \geqs n-1$ and thus $\tau$ is non-regular as claimed.
\end{proof}

\begin{rem}
\mbox{ }
\begin{itemize}\addtolength{\itemsep}{0.3\baselineskip}
\item[{\rm (a)}] Notice that in Proposition \ref{p:simple2} we assume $G = {\rm L}_n(2)$ with $n \geqs 5$. For $n=3$ we get $B(G) = 3$ and $R(G) = 4$. However, the case $n=4$ is different since ${\rm L}_4(2) \cong A_8$ and thus $B(G) = R(G) = 6$ (as a special case of Theorem \ref{t:main1}(i)).

\item[{\rm (b)}] We expect $R(G) = 2(n-1)$ when $G = {\rm L}_n(2)$ and $n \geqs 3$, which would imply that $R(G) < 2B(G)$. In particular, this would be consistent with Conjecture \ref{con:main}(i), which asserts (as a special case) that there is an absolute constant $c$ such that $R(G) < c\cdot B(G)$ for every finite simple group $G$.
\end{itemize}
\end{rem}

\subsection{Computational methods}\label{ss:comp}

To conclude this preliminary section, we briefly discuss some of the main computational methods we apply in this paper. All of our computations are performed using {\sc Magma} \cite{magma} (version 2.28-4) and \textsf{GAP} \cite{GAP} (version 4.11.1). We refer the reader to the supplementary file \cite{AB_comp} for further details, including sample code.

Let $G$ be a finite group and let $\tau = (H_1, \ldots, H_k)$ be a $k$-tuple of core-free subgroups with $k \geqs 2$. As noted in Lemma \ref{l:triv}, $\tau$ is non-regular if $\prod_i |H_i|>|G|^{k-1}$.

\subsubsection{Random search}\label{ss:random}

If a suitable permutation or matrix representation of $G$ is available, then random search often provides an efficient way to show that $\tau$ is regular. Here the aim is to find elements $g_i \in G$, in terms of the given representation, such that
\[
\bigcap_{i=1}^k H_i^{g_i} = 1.
\]
Typically, we are interested in the case where $\tau$ is a maximal tuple, which means that each component $H_i$ is a core-free maximal subgroup of $G$.

Of course, the groups $S_n$ and $A_n$ are naturally defined in terms of a permutation representation. For an almost simple sporadic group $G \not\in \{{\rm J}_4, {\rm Ly}, {\rm Th}, \mathbb{B}, \mathbb{M}\}$, we use the {\sc Magma} function \texttt{AutomorphismGroupSimpleGroup} to construct $G$ as a permutation group on $n$ points, where $n = m(G)$ is the minimal degree of a faithful permutation representation.  Similarly, for $G \in \{{\rm J}_4, {\rm Ly}, {\rm Th}\}$ we can use the function \texttt{MatrixGroup} to construct $G$ as a matrix group. For example, this allows us to view the Lyons  group as a subgroup of ${\rm GL}_{111}(5)$. However, the largest sporadic groups $\mathbb{B}$ and $\mathbb{M}$ do not admit a permutation or matrix representation that is suitable for direct computation, so they will require a different approach (for example, $97239461142009186000 \sim 9.7 \times 10^{19}$ is the minimal degree of a faithful permutation representation of $\mathbb{M}$).

Given a suitable representation of $G$, we can use the {\sc Magma} function \texttt{MaximalSubgroups} to construct a complete set of representatives of the conjugacy classes of maximal subgroups of $G$. For example, this is effective for all $G \in \{S_n,A_n\}$ with $n \leqs 200$. And similarly for all sporadic groups $G$, unless 
$G \in \{ {\rm Fi}_{24}', {\rm Fi}_{24}, {\rm Th}\}$. For the latter groups, we can construct the maximal subgroups we will need to work with via explicit generators presented in the Web Atlas \cite{WebAt}, which are given as words in the standard generators for $G$ (these generators can also be constructed using the \textsf{GAP} package \texttt{AtlasRep} \cite{AR}). We can of course filter the output from \texttt{MaximalSubgroups} if we are just interested in soluble maximal subgroups, or the maximal subgroups of $S_n$ that act primitively on $\{1, \ldots, n\}$, etc. 

For some of the groups we are interested in, we can replace \texttt{MaximalSubgroups} by other {\sc Magma} functions, such as \texttt{SolubleSubgroups} and \texttt{NilpotentSubgroups}, which return a complete set of representatives of the conjugacy classes of soluble (respectively, nilpotent) subgroups of $G$ (not surprisingly, the effectiveness of the latter functions is more limited for groups of large order). We can also iteratively apply \texttt{MaximalSubgroups} in order to descend deeper into the subgroup lattice of $G$. For example, this will be a useful technique in the proof of Theorem \ref{t:main2}(ii), when we are seeking to show that $R_{\rm sol}(G) \leqs 3$ in the setting where $G$ is a sporadic group and the function \texttt{SolubleSubgroups} is unavailable (see the proofs of Lemmas \ref{l:solspor2} and \ref{l:solspor3}, for instance).

\subsubsection{Orbit computations}

Suppose we can construct $G$ and each $H_i$, in terms of a suitable permutation or matrix representation, and let us assume random search is inconclusive. Here it is useful to observe that $\tau$ is regular if and only if $H_1$ has a regular orbit on 
\[
Y = G/H_2 \times \cdots \times G/H_k.
\]
So assuming that the indices $|G:H_i|$ are not prohibitively large, we can use the {\sc Magma} function \texttt{CosetAction} to construct the action on $G$ on each set $G/H_i$ with $i \geqs 2$, which will allow us to calculate all of the orbits of $H_1$ on $Y$ and then  determine whether or not $\tau$ is regular. Note that in order to conclude that $\tau$ is non-regular, it suffices to find a collection of distinct non-regular $H_1$-orbits $Y_1, \ldots, Y_t$ such that $\sum_{i} |Y_i| > |Y| - |H_1|$.

In the special case $k=2$, we can also use \texttt{DoubleCosetRepresentatives} to calculate the size of every $(H_1,H_2)$ double coset in $G$, noting that $\tau$ is regular if and only if there exists a double coset of size $|H_1||H_2|$. 

\begin{rem}
There are a handful of maximal pairs $\tau = (H_1,H_2)$ for $G \in \{ {\rm J}_4, {\rm Ly}, {\rm Th}\}$ where we have $|H_1||H_2|< |G|$, random search is inconclusive and both of the above methods for determining whether or not $\tau$ is regular are ineffective due to the very large size of the indices $|G:H_1|$ and $|G:H_2|$. We thank Derek Holt, J\"{u}rgen M\"{u}ller and Eamonn O'Brien for resolving these special cases, which are presented in Lemma \ref{l:derek}. In particular, M\"{u}ller used the \textsf{GAP} package \textsc{Orb} \cite{Orb} to show that the maximal pair for $G = {\rm J}_4$ with
\[
H_1 = 2^{3+12}.(S_5 \times {\rm L}_3(2)),\;\; H_2 = 2^{1+12}.3.{\rm M}_{22}{:}2
\]
is non-regular (here we have $|G:H_1| = 131358148251$ and $|G:H_2| = 3980549947$).
\end{rem}
 
\subsubsection{Character theory}

If we have access to the character tables of $G$ and each component subgroup $H_i$, together with the fusion maps from $H_i$-classes to $G$-classes, then we can compute
\[
{\rm fpr}(x,G/H_i) = \frac{|x^G \cap H_i|}{|x^G|}
\]
precisely for all $x \in G$ and all $i$. In turn, this allows us to evaluate the expression 
$\what{Q}(G,\tau)$ in Lemma \ref{l:fpr}, recalling that $\tau$ is regular if $\what{Q}(G,\tau)<1$. This can be an effective way to produce a list of candidate non-regular $k$-tuples, which can then be analysed using other methods (such as random search, as described above).

The character table of every almost simple sporadic group $G$ is available in the \textsf{GAP} Character Table Library \cite{GAPCTL}. In addition, if $G \ne \mathbb{M}$, then  the character table of every maximal subgroup of $G$ is also available in \cite{GAPCTL} and we can use the \textsf{GAP} function \texttt{Maxes} to access the associated fusion maps on conjugacy classes. For $G \ne \mathbb{M}$, this allows us to compute $\what{Q}(G,\tau)$ for every maximal tuple $\tau$ (if $G = \mathbb{M}$, we can handle some, but not all, maximal tuples in the same way, using the function \texttt{NamesOfFusionSources} to access the relevant character tables and fusion maps). We will rely heavily on this probabilistic approach to prove Theorem \ref{t:main2} for the large sporadic groups, including $\mathbb{B}$ and $\mathbb{M}$ (see the proofs of Lemmas \ref{l:rG2}, \ref{l:rG3} and \ref{l:rG4}, for example).

\section{Symmetric and alternating groups}\label{s:sym}

In this section we assume $G$ is an almost simple group with socle $A_n$ and our aim is to prove Theorem \ref{t:main1}. We begin in Section \ref{ss:prim} by considering parts (ii) and (iii), focussing initially on primitive tuples (indeed, part (iii) quickly follows from part (ii)  since every soluble maximal subgroup of $G$ is primitive when $n \geqs 17$). And then in Section \ref{ss:regnumber} we turn our attention to the regularity number, proving part (i) of Theorem \ref{t:main1}. To do this, we first handle the tuples where each subgroup acts intransitively on $\{1, \ldots, n\}$ (see Section \ref{ss:intrans}). The imprimitive tuples are then treated in Section \ref{ss:imprim} and we bring everything together in Section \ref{ss:final}, where we complete the proof of Theorem \ref{t:main1}.

\subsection{Primitive subgroups}\label{ss:prim}

Let $G$ be an almost simple group with socle $T = A_n$ and recall that a subgroup $H$ of $G$ is said to be \emph{primitive} if $H \cap T$ acts primitively on the set $[n] = \{1, \ldots, n\}$. We define $R_{{\rm prim}}(G)$ to be the minimal positive integer $k$ such that every core-free primitive $k$-tuple of subgroups of $G$ is regular. Similarly, $B_{{\rm prim}}(G)$ is the maximal base size $b(G,H)$ over all core-free primitive subgroups $H$ of $G$. Of course, if $G$ does not have a core-free primitive subgroup, then $R_{{\rm prim}}(G)$ and $B_{{\rm prim}}(G)$ are undefined.

In this section we will prove parts (ii) and (iii) of Theorem \ref{t:main1}. As a consequence,  we establish a strong form of Conjecture \ref{con:main}(ii) for symmetric and alternating groups, and we take the first steps towards Conjecture \ref{con:main}(iii) in this setting.

Let us briefly outline our approach. We begin by handling the groups of small degree in Proposition \ref{p:n40} below, which are amenable to direct computation in {\sc Magma} \cite{magma}. Our method for dealing with the general case relies on the probabilistic approach described in Section \ref{ss:prob}. To do this, we will apply two key results in order to derive the required fixed point ratio estimates, which we combine in the statement of Theorem \ref{thm:maroti}. The first is a theorem of Mar\'{o}ti \cite{Maroti02} on the orders of core-free primitive subgroups of $S_n$. The second main ingredient is a theorem of Guralnick and Magaard \cite{GM} on the minimal degree of primitive subgroups, which we can use to obtain lower bounds on the conjugacy class sizes of prime order elements lying in a primitive subgroup. Our approach is similar to the one adopted in \cite{BGS}, where the main theorem states that $B_{{\rm prim}}(G) \leqs 5$, with $B_{{\rm prim}}(G) = 2$ if $n \geqs 13$. 

In order to establish the bound $R_{{\rm prim}}(G) \leqs k$, it suffices to show that every $k$-tuple of primitive subgroups in $\mathcal{M}(G)$ is regular, where $\mathcal{M}(G)$ is the set of core-free subgroups that are maximal in some subgroup of $G$ containing $T$ (see Definition \ref{d:mG} and Lemma \ref{lem:maxreduction}). In fact, if $G = S_n$ or $A_n$ then we only need to consider tuples involving subgroups in $\mathcal{M}'(G)$, which is the set of core-free maximal subgroups of $G$, together with any maximal subgroups $H$ of $T$ with $N_G(H) = H$ (see Remark \ref{r:large}).

We begin by establishing part (ii) of Theorem \ref{t:main1} for the groups with $n < 60$.

\begin{prop}\label{p:n40}
Let $G$ be an almost simple group with socle $T = A_n$, where $n < 60$, and assume $G$ has a core-free primitive subgroup.

\vspace{1mm}

\begin{itemize}\addtolength{\itemsep}{0.3\baselineskip}
\item[{\rm (i)}] We have $R_{\rm prim}(G) \leqs 6$, with equality if and only if $G = A_8$. 
\item[{\rm (ii)}] If $n\geqs 13$, then $R_{\rm prim}(G) = 2$. For $n \leqs 12$, the groups with $R_{{\rm prim}}(G) = k > 2$ are recorded in Table \ref{tab:prim}, together with a non-regular primitive tuple $\tau = (H_1, \ldots, H_{k-1})$.
\end{itemize} 
\end{prop}

{\scriptsize
\begin{table}
\[
\begin{array}{c l  l} \hline
R_{{\rm prim}}(G) & G & \tau \\ \hline
6 & A_8 &  ({\rm AGL}_3(2)_1, {\rm AGL}_3(2)_1,  {\rm AGL}_3(2)_1,  {\rm AGL}_3(2)_2,  {\rm AGL}_3(2)_2) \\
5 & S_6 &  ({\rm PGL}_2(5), {\rm PGL}_2(5), {\rm PGL}_2(5), {\rm PGL}_2(5)) \\
4 & A_6 &  ({\rm L}_2(5), {\rm L}_2(5), {\rm L}_2(5)) \\
& A_6.2 \ne S_6 &  ({\rm L}_2(5), {\rm L}_2(5), {\rm L}_2(5)) \\
& A_6.2^2 &  ({\rm L}_2(5), {\rm L}_2(5), {\rm L}_2(5)) \\ 
& A_7 & ({\rm L}_2(7)_1, {\rm L}_2(7)_1, {\rm L}_2(7)_2) \\
& S_8 &  ({\rm PGL}_2(7), {\rm PGL}_2(7), {\rm PGL}_2(7))\\
3 & A_5 &  ({\rm D}_{10}, {\rm D}_{10})\\
& S_5 &  ({\rm AGL}_1(5), {\rm AGL}_1(5))\\
& S_7 &  ({\rm L}_2(7), {\rm L}_2(7))\\
& A_9 &  ({\rm P\Gamma L}_2(8), {\rm P\Gamma L}_2(8))\\
& S_9 &  ({\rm AGL}_2(3), {\rm AGL}_2(3))\\
& S_{10} &  ({\rm P\Gamma L}_2(9), {\rm P\Gamma L}_2(9))\\
& A_{11} &  ({\rm M}_{11}, {\rm M}_{11} )\\
& S_{11} &  ({\rm M}_{11}, {\rm M}_{11})\\
& A_{12} &  ({\rm M}_{12}, {\rm M}_{12})\\
& S_{12} &  ({\rm M}_{12}, {\rm M}_{12})
\\ \hline
\end{array}
\]
\caption{Almost simple groups $G$ with socle $A_n$ and $R_{{\rm prim}}(G)  > 2$}
\label{tab:prim}
\end{table}
}

\begin{proof}
This is a routine {\sc Magma} \cite{magma} computation. For $n \geqs 13$, we can use random search (see Section \ref{ss:random}) to show that every primitive pair $(H_1,H_2)$ with $H_i \in \mathcal{M}'(G)$ is regular. Similarly, for $n \leqs 12$ it is straightforward to identify all the non-regular pairs $(H_1,H_2)$ with $H_i \in \mathcal{M}(G)$, which we can then use to determine any non-regular triples with components in $\mathcal{M}(G)$. If required, we can then extend the analysis to $4$-tuples and $5$-tuples, which allows us to deduce that every primitive $6$-tuple is regular.
\end{proof}

\begin{rem}\label{r:sn40}
We record some remarks on the statement of Proposition \ref{p:n40}.

\vspace{1mm}

\begin{itemize}\addtolength{\itemsep}{0.3\baselineskip}
\item[\rm (a)] Notice that if $n \leqs 12$, then $G = A_{10}$ is the only group with $R_{{\rm prim}}(G) = 2$. Here $H = {\rm M}_{10} = A_6.2$ represents the unique class of primitive maximal subgroups of $G$ and one checks that $b(G,H) = 2$, so every primitive pair for $G$ is regular.
\item[\rm (b)] By inspecting \cite{BGS}, we deduce that if $n < 60 $ then $R_{{\rm prim}}(G) = B_{{\rm prim}}(G)$ unless $G =A_7$ or $A_8$, where we have 
\begin{align*}
R_{{\rm prim}}(A_7) & = B_{{\rm prim}}(A_7) + 1 = 4 \\
R_{{\rm prim}}(A_8) & = B_{{\rm prim}}(A_8) + 2 = 6
\end{align*}
\item[\rm (c)] The group $G = A_8$ has two conjugacy classes of maximal subgroups isomorphic to ${\rm AGL}_3(2)$, represented by $H$ and $K$. Here $b(G,H) = b(G,K) = 4$, but we find that the $5$-tuple $(H,H,H,K,K)$ is non-regular, as indicated in the first row of Table \ref{tab:prim} (we use subscripts $1$ and $2$ to distinguish representatives of the two classes). 
\item[\rm (d)] Similarly, $G = A_7$ has two classes of maximal subgroups isomorphic to ${\rm L}_2(7)$ and $(H,H,K)$ is a non-regular triple, where $H$ and $K$ represent these two conjugacy classes (note that $b(G,H) = b(G,K) = 3$).
\end{itemize}
\end{rem}

In order to complete the proof of Theorem \ref{t:main1}(ii), we may assume $n \geqs 60$. So  $G = S_n$ or $A_n$, and our goal is to prove that $R_{\rm prim}(G) = 2$.

Recall that the \emph{minimal degree} of a subgroup $H$ of $G$, denoted $\mu(H)$, is the minimal number of points in $[n] = \{1, \ldots, n\}$ moved by a non-identity element of $H$. That is,
\[
\mu(H) = \min\{{\rm supp}(x) \,:\, 1 \ne x \in H\}
\]
where ${\rm supp}(x)$, the \emph{support} of $x$, is the number of points moved by $x$.

As explained above, we will complete the proof of Theorem \ref{t:main1}(iii) by applying the probabilistic approach encapsulated in Lemma \ref{l:fpr}, which relies on 
key theorems of Mar\'{o}ti \cite{Maroti02} and Guralnick and Magaard \cite{GM}. We combine (and slightly simplify) both results in the following statement.

\begin{thm}\label{thm:maroti}
Suppose $n \geqs 60$ and $H \in \mathcal{M}(G)$ is primitive. Then either 
\[
|H| < n^{1 + \lceil \log_2 n \rceil} \mbox{ and } \mu(H) \geqs n/2 - \sqrt{n},
\]
or one of the following holds:

\vspace{1mm}

\begin{itemize}\addtolength{\itemsep}{0.3\baselineskip}
\item[\rm (i)] $H = (S_l \wr S_k)\cap G$ in its product action on $n = l^k$ points, where $l \geqs 5$ and $k \geqs 2$, and we have $\mu(H) = 2l^{k-1}$.
\item[\rm (ii)] $H = S_l \cap G$ in its action on the $k$-element subsets of $\{1, \ldots, l\}$, so $n = \binom{l}{k}$ with $l\geqs 5$ and $2 \leqs k < l/2$, and we have $\mu(H) = 2\binom{l-2}{k-1}$.
\end{itemize}
\end{thm}

\begin{proof}
The bound on $|H|$ follows immediately from \cite[Theorem 1.1]{Maroti02}. Now let us turn to the minimal degree of $H$. By \cite[Theorem 1]{GM}, if we exclude the cases recorded in (i) and (ii), then either $\mu(H) \geqs n/2$, or $H = {\rm O}(V)$ is an orthogonal group over the field $\mathbb{F}_2$ acting on a set of hyperplanes of the natural module $V$. In each of the latter cases, the exact value of $\mu(H)$ is given in the statement of \cite[Theorem 1]{GM} and it is easy to check that the desired bound $\mu(H) \geqs n/2 - \sqrt{n}$ is satisfied.
\end{proof}

Next we seek lower bounds on the size of a conjugacy class $x^G$ defined in terms of the support of $x$. This will allow us to translate the minimal degree bound in Theorem \ref{thm:maroti} into a lower bound on $|x^G|$, which applies whenever $x$ is a prime order element contained in an appropriate primitive subgroup of $G$. The following result holds for all $n \geqs 5$.

\begin{lem}\label{lemma:cclasses1}
Let $G = S_n$ or $A_n$ with $n \geqs 5$ and suppose $x \in G$ has prime order $r$ with ${\rm supp}(x) = m$. Then $|x^G| \geqs f_2(m)$, with $|x^G| \geqs f_3(m)$ if $r$ is odd, where
\[
f_s(m)= \frac{n!}{s^{m/s}\lceil m/s \rceil!(n - m)!}.
\]
\end{lem}

\begin{proof}
First observe that we may assume $G = S_n$. To see this, suppose $G = A_n$ and recall that $|x^G| = |x^{S_n}|$ unless $x$ is an $r$-cycle and $r \in \{n-1,n\}$. If $r =n$ then $m=n$ and the desired conclusion holds since
\[
|x^G| = \frac{1}{2}(n-1)! > \frac{n!}{3^{n/3}\lceil n/3 \rceil!} > \frac{n!}{2^{n/2}\lceil n/2 \rceil!}
\]
And similarly if $r = n-1$. So for the remainder, we will assume $G = S_n$.

Suppose $x$ has cycle-type $(r^a,1^{n-ar})$ for some $a \geqs 1$. Then ${\rm supp}(x) = ar$ and 
\[
|x^G| = \frac{n!}{r^aa!(n - ar)!}.
\]
In particular, if $r=2$ then $|x^G| = f_2(m)$, so we may assume $r$ is odd.

Suppose there exists an element $y \in G$ of prime order $q<r$ with ${\rm supp}(y) = m$. Then it is easy to check that $q^{m/q}(m/q)! > r^{m/r}(m/r)!$ and thus $|x^G| > |y^G|$. In particular, if $m$ is even then we immediately deduce that $|x^G| \geqs f_2(m)$. And the same conclusion holds for $m$ odd since 
\[
2^{m/2}((m+1)/2)! \geqs r^{m/r}(m/r)!
\]
for every odd prime $r$ dividing $m$. Finally, in order to establish the bound $|x^G| \geqs f_3(m)$ it suffices to show that
\[
3^{m/3}\lceil m/3\rceil! \geqs r^{m/r}(m/r)!
\]
and this is straightforward to verify.
\end{proof}

We now bring the minimal degree bound in Theorem \ref{thm:maroti} in to play, which allows us to derive a lower bound on the conjugacy class sizes of elements of prime order contained in a primitive maximal subgroup of $S_n$ or $A_n$. In the statement, we refer to the function $f_s(m)$ defined in Lemma \ref{lemma:cclasses1}.

\begin{lem}\label{lemma:cclasses2}
Let $G = S_n$ or $A_n$ with $n \geqs 60$ and suppose $x \in H$ has prime order $r$, where $H \in \mathcal{M}(G)$ is primitive. 

\vspace{1mm}

\begin{itemize}\addtolength{\itemsep}{0.3\baselineskip}
\item[\rm (i)] We have $|x^G| \geqs f_2(\ell)$, with $|x^G| \geqs f_3(\ell)$ if $r$ is odd, where $\ell = \lceil 2\sqrt{n}\rceil$.
\item[\rm (ii)] If $\mu(H) \geqs n/2 - \sqrt{n}$, then $|x^G| \geqs f_2(\ell')$ with $\ell' = \lceil n/2 - \sqrt{n} \rceil$.
\end{itemize}
\end{lem}

\begin{proof}
Set $m = {\rm supp}(x)$. By applying Theorem \ref{thm:maroti}, we deduce that $\mu(H) \geqs 2\sqrt{n}$, which in turn implies that $\lceil 2\sqrt{n} \rceil \leqs m \leqs n$. By Lemma \ref{lemma:cclasses1}, it follows that $|x^G| \geqs f_2(m)$, with $|x^G| \geqs f_3(m)$ if $r$ is odd, and it is straightforward to show that both lower bounds are minimal when $m = \lceil 2\sqrt{n} \rceil$, which establishes (i). Part (ii) is entirely similar.
\end{proof}

We will also need an upper bound on the number of involutions in a primitive subgroup $H \in \mathcal{M}(G)$, which is denoted by $i_2(H)$.

\begin{lem}\label{lemma:frob-schur}
If $n \geqs 5$, then $i_2(H) \leqs 5^{(n - 1)/6}\sqrt{|H|}$ for every subgroup $H$ of $S_n$.
\end{lem}

\begin{proof}
First recall that $i_2(H) \leqs \sqrt{k|H|}$, where $k$ is the number of conjugacy classes in $H$ (this standard bound is easily obtained by considering the Frobenius-Schur indicators of the complex irreducible characters of $H$). By the main theorem of \cite{GM15} we have $k \leqs 5^{(n - 1)/3}$ and the result follows.
\end{proof}

\begin{cor}\label{lemma:invols}
Let $G = S_n$ or $A_n$ with $n \geqs 60$ and suppose $H \in \mathcal{M}(G)$ is primitive as in part (i) or (ii) of Theorem \ref{thm:maroti}. Then 
\[
i_2(H) \leqs \lceil\sqrt{n}\rceil!(1+5^{(\sqrt{n} - 1)/6})^2.
\]
\end{cor}

\begin{proof}
First assume that $H = (S_l \wr S_k) \cap G$, where $n = l^k$, $l \geqs 5$ and $k \geqs 2$, and note that it suffices to bound $i_2(S_l\wr S_k)$. For $k=2$,  Lemma \ref{lemma:frob-schur} implies that there are fewer than $(1 + 5^{(l - 1)/6}(l!)^{1/2})^2$ involutions in the base group $S_l \times S_l$ of the wreath product, and there are exactly $l!$ additional involutions in $H$. Therefore, 
\[
i_2(H) \leqs (1 + 5^{(l - 1)/6}(l!)^{1/2})^2 + l! < l!(1+5^{(l - 1)/6})^2.
\]
And for $k \geqs 3$ one can check that 
\[
i_2(H) \leqs |S_l \wr S_k| = (\sqrt[k]{n})!^k k! \leqs \lceil \sqrt{n}\rceil!(1+5^{(\sqrt{n} - 1)/6})^2 = f(n).
\]

Now assume $H = S_l \cap G$ and $n = \binom{l}{k}$, where $l \geqs 5$ and $2 \leqs k < l/2$. For $k=2$, Lemma \ref{lemma:frob-schur} implies that 
\[
i_2(H) \leqs i_2(S_l) \leqs 5^{(l - 1)/6}(l!)^{1/2} < 5^{\sqrt{2n}/6}\lceil \sqrt{2n} + 1\rceil!^{1/2} = g(n)
\]
and it is easy to check that $i_2(H) \leqs |H|  \leqs l! < g(n)$ for $k \geqs 3$ and $n\geqs 60$. The result now follows since $g(n) \leqs f(n)$ for all $n \geqs 60$.
\end{proof}

We are now in a position to prove Theorem \ref{t:main1}(ii). 

\begin{proof}[Proof of Theorem \ref{t:main1}(ii)]
In view of Proposition \ref{p:n40}, we may assume $n \geqs 60$, so $G = S_n$ or $A_n$, and our goal is to prove that $R_{\rm prim}(G) = 2$. To do this, it suffices to show that every primitive pair $\tau = (H_1,H_2)$ is regular, where $H_1,H_2 \in \mathcal{M}(G)$. Define 
\[
\what{Q}(G,\tau) = \sum_{i = 1}^t \frac{|x_i^G \cap H_1|\, |x_i^G\cap H_2|}{|x_i^G|}
\]
as in Lemma \ref{l:fpr}, where $x_1, \ldots, x_t$ is a set of representatives of the conjugacy classes in $G$ of elements of prime order, and recall that $\tau$ is regular if $\what{Q}(G,\tau)<1$.

First assume that $\mu(H_1) \geqs n/2-\sqrt{n}$. Then $|x_i^G\cap H_1| = 0$ if ${\rm supp}(x_i) < n/2 - \sqrt{n}$ and thus
\[
\what{Q}(G,\tau) = \sum_{i = 1}^s \frac{|y_i^G \cap H_1|\, |y_i^G\cap H_2|}{|y_i^G|}
\]
where $y_1, \ldots, y_s$ represent the classes of prime order elements with support at least $n/2 - \sqrt{n}$. By combining Lemmas \ref{l:favbound} and \ref{lemma:cclasses2}(ii) with Mar\'{o}ti's bound in Theorem \ref{thm:maroti} on the orders of $H_1$ and $H_2$, we deduce that 
\[
\what{Q}(G,\tau) \leqs n^{2 +2\lceil \log_2 n \rceil}\cdot \frac{2^{\ell/2}\lceil \ell/2 \rceil!(n - \ell)!}{n!},
\]
where $\ell = \lceil n/2 -\sqrt{n}\rceil$. One checks that this upper bound is less than $1$ for all $n \geqs 60$. 

By the same argument, $\tau$ is regular if $\mu(H_2) \geqs n/2-\sqrt{n}$, so to complete the proof we may assume that both $H_1$ and $H_2$ are described as in item (i) or (ii) in Theorem \ref{thm:maroti}. Write
\[
\what{Q}(G,\tau)  =\a+\b,
\]
where $\a$ is the contribution from involutions, and note that $|H_i| \leqs 2\lceil \sqrt{n}\rceil!^2$ for $i=1,2$. Then by combining Lemmas \ref{l:favbound} and \ref{lemma:cclasses2} with Corollary \ref{lemma:invols}, we obtain the bounds
\begin{align*}
\a & \leqs i_2(H_1)i_2(H_2) \cdot \frac{2^{m/2}\lceil m/2 \rceil!(n - m)!}{n!} \\
& \leqs \left(\lceil \sqrt{n}\rceil!(1+5^{(\sqrt{n} - 1)/6})^2\right)^2 \cdot \frac{2^{m/2}\lceil m/2\rceil!(n - m)!}{n!}
\end{align*}
and
\[
\b  \leqs |H_1||H_2| \cdot \frac{3^{m/3}\lceil m/3 \rceil!(n - m)!}{n!}  \leqs 4\lceil \sqrt{n}\rceil!^4 \cdot \frac{3^{m/3}\lceil m/3\rceil!(n - m)!}{n!},
\]
where $m = \lceil 2\sqrt{n}\rceil$. Together, these bounds imply that $\a+\b< 1$ and the result follows.
\end{proof}

Part (iii) of Theorem \ref{t:main1} now follows as an easy corollary. Recall that $R_{{\rm sol\, max}}(G)$ is the minimal integer $k$ such that every $k$-tuple of soluble maximal  subgroups of $G$ is regular.

\begin{proof}[Proof of Theorem \ref{t:main1}(iii)]
For $n \leqs 16$ we can use {\sc Magma} to verify the result (in Table \ref{tab:maxsol} we record all the groups $G$ with $R_{{\rm sol\, max}}(G) = k > 2$, together with a non-regular soluble maximal $(k-1)$-tuple of the form $(H, \ldots,H)$). And if $n \geqs 17$, then every soluble maximal subgroup of $G$ is primitive and therefore $R_{{\rm sol\, max}}(G) =2$ since $R_{{\rm prim}}(G) = 2$ by part (ii) of  Theorem \ref{t:main1}. 
\end{proof}

{\scriptsize
\begin{table}
\[
\begin{array}{c l l} \hline
R_{{\rm sol\, max}}(G) & G & H \\ \hline
5 & S_8 & S_4\wr S_2 \\
4 & S_5 & S_4 \\
& S_6 & S_3\wr S_2 \\
& A_6.2^2 & (S_3\wr S_2).2  \\
& A_8 & (S_4\wr S_2)\cap G\\
3 & A_5 & A_4 \\
& A_6 & (S_3\wr S_2)\cap G\\
& A_6.2 = {\rm PGL}_2(9) & {\rm D}_{20} \\
& A_6.2 = {\rm M}_{10} & 3^2{:}Q_8 \\
& A_7 & (S_4\times S_3)\cap G \\
& S_7& S_4\times S_3\\
& A_9 & (S_3\wr S_3)\cap G\\
& S_9 & S_3\wr S_3\\
& A_{12} & (S_4\wr S_3)\cap G\\
& S_{12} & S_4\wr S_3\\
& A_{16} & (S_4\wr S_4)\cap G\\
& S_{16} & S_4\wr S_4
\\ \hline
\end{array}
\]
\caption{Almost simple groups $G$ with socle $A_n$ and $R_{{\rm sol\, max}}(G)  > 2$}
\label{tab:maxsol}
\end{table}
}

\subsection{The regularity number}\label{ss:regnumber}

Let $G = S_n$ or $A_n$ with $n \geqs 5$. In this section we complete the proof of Theorem \ref{t:main1} by determining the regularity number of $G$. By considering the base size of the natural action of $G$ on $[n] = \{1, \ldots, n\}$, it suffices to show that $R(G) \leqs k$, where $k = n-1$ if $G = S_n$ and $k=n-2$ if $G = A_n$. So in view of Lemma \ref{lem:maxreduction}, we just need to verify that every $k$-tuple of subgroups in $\mathcal{M}(G)$ is regular, where $\mathcal{M}(G)$ is the set of core-free subgroups of $G$ that are maximal in $S_n$ or $A_n$.

The proof divides naturally into three cases since any subgroup in $\mathcal{M}(G)$ acts intransitively, imprimitively or primitively on $[n]$. Recall that we define $R_{{\rm intrans}}(G)$ to be the minimal integer $k$ such that every $k$-tuple of intransitive subgroups $(H_1, \ldots, H_k)$ of $G$ is regular. Similarly, $R_{{\rm imprim}}(G)$ is defined with respect to imprimitive subgroups (note that if $n$ is a prime number, then every transitive subgroup of $G$ is primitive, so $R_{{\rm imprim}}(G)$ is not defined in this case).

In Section \ref{ss:intrans} we begin by handling the intransitive tuples, presenting several lemmas that will allow us to prove that $R_{\rm intrans}(S_n) = n-1$ and $R_{\rm intrans}(A_n) = n-2$ via a direct construction (see Proposition \ref{p:rintrans}). The imprimitive tuples are then studied in Section \ref{ss:imprim}. Here the analysis is more complicated and we rely on two key results (Lemmas \ref{lem:pointwise} and \ref{lem:2sets}), which may be of independent interest. Roughly speaking, our strategy is to reduce the imprimitive case to a situation involving intransitive subgroups, which we have already handled in Section \ref{ss:intrans}. Finally, we complete the proof of Theorem \ref{t:main1} by combining our work on intransitive and imprimitive tuples with our results on primitive subgroups in Section \ref{ss:prim}. The details are presented in Section \ref{ss:final}.

\subsubsection{Intransitive tuples}\label{ss:intrans}

Our main result in this section is Proposition \ref{p:rintrans}, which gives 
\[
R_{\rm intrans}(S_n) = n-1 \; \mbox{ and } \; R_{\rm intrans}(A_n) = n-2.
\] 
Throughout this section we fix an integer $n \geqs 5$ and we define $W_n$ to be the set of subsets of $[n] = \{1, \ldots, n\}$ of size at most $n/2$. Following Halasi \cite{H12}, if $\alpha\in [n]$ and $\mathcal{A}\subseteq W_n$, then we refer to 
\[
N_{\mathcal{A}}(\alpha) = \{X\in \mathcal{A} \,:\, \alpha \in X\}
\]
as the \emph{neighbourhood} of $\alpha$ with respect to $\mathcal{A}$. 

\begin{lem}\label{lem:neighbourhood}
Let $\mathcal{A} = \{X_1, \ldots, X_k\} \subseteq W_n$ and set $H = \bigcap_{i=1}^kH_i$, where $H_i$ is the setwise stabiliser of $X_i$ in $S_n$. 

\vspace{1mm}

\begin{itemize}\addtolength{\itemsep}{0.3\baselineskip}
\item[\rm (i)] For all $\alpha, \beta\in [n]$, we have $N_{\mathcal{A}}(\alpha) \ne N_{\mathcal{A}}(\beta)$ if and only if $\alpha$ and $\beta$ are contained in different $H$-orbits. 
\item[\rm (ii)] In particular, $H=1$ if and only if $N_{\mathcal{A}}(\alpha) \neq N_{\mathcal{A}}(\beta)$ for all distinct $\alpha, \beta \in [n]$.
\end{itemize}
\end{lem}

\begin{proof}
Since (ii) follows immediately from (i), we just need to establish the claim in (i). 
First assume $\alpha$ and $\beta$ lie in distinct orbits of $H$. Then $H$ does not contain the transposition $g = (\alpha, \beta) \in S_n$, which means that $X_i^{g}\neq X_i$ for some $i$. Therefore, either  $X_i\in N_{\mathcal{A}}(\alpha) \setminus N_{\mathcal{A}}(\beta)$ or $X_i\in N_{\mathcal{A}}(\beta) \setminus N_{\mathcal{A}}(\alpha)$, and thus $N_{\mathcal{A}}(\alpha) \neq N_{\mathcal{A}}(\beta)$.

Conversely, suppose $N_{\mathcal{A}}(\alpha) \neq N_{\mathcal{A}}(\beta)$. Then we may assume there exists a set $X_i \in \mathcal{A}$ such that $\alpha \in X_i$ and $\beta\not\in X_i$. Consider an element $g \in S_n$ such that $\alpha^g = \beta$. Then $X_i^g \neq X_i$ and thus $g \not\in H$, which in turn implies that $\a$ and $\b$ are in distinct $H$-orbits.
\end{proof}

We now present a lemma that will play a key role in the proof of Proposition \ref{p:rintrans}. It will also be used repeatedly in the proof of Theorem \ref{t:main1} in Section \ref{ss:final}.

\begin{lem}\label{lem:orbits}
Let $a_1, \ldots, a_k$ be positive integers such that $a_1 \leqs a_2 \leqs \cdots \leqs a_k \leqs n/2$, where $2 \leqs k < n$. For each $i \in \{1, \ldots, k\}$, let $X_i$ be the following $a_i$-element subset of $[n]$ and let $H = \bigcap_{i=1}^k H_i$, where $H_i$ is the stabiliser of $X_i$ in $S_n$:
\begin{align}\label{e:xi}
\begin{split}
X_1 &= \{1, \ldots, a_1\}\\
X_2 &= \{2, \ldots, a_2 + 1\}\\
&\;\, \vdots\\
X_{k - 1} & =\left\{ \begin{array}{ll}
\{k - 1, \ldots, a_{k-1}+k-2\} & \mbox{if $a_{k-1} \leqs n-k+1$} \\
\{k - 1, \ldots, n - 1, 1, 2, \ldots, a_{k - 1} - (n - k + 1)\} & \mbox{otherwise}
\end{array}\right. \\
X_{k} & = \left\{ \begin{array}{ll}
\{k, \ldots, a_k+k-1\} & \mbox{if $a_{k} \leqs n-k$} \\
\{k,\ldots, n-1, 1, 2, \ldots, a_{k} - (n-k)\} & \mbox{otherwise}
\end{array}\right.
\end{split}
\end{align}
Then  $H \leqs {\rm Sym}(\{k, \ldots, m\})\times {\rm Sym}(\{m+1, \ldots, n\})$, where 
$m = \max{\{\a \, :\, \a \in X_k\}}$. 
\end{lem}

\begin{proof}
Let $\mathcal{A} = \{X_1, \ldots, X_k\}$. First we show that $N_{\mathcal{A}}(\alpha) \neq N_{\mathcal{A}}(\beta)$ for all distinct $\alpha, \beta\in \{1, \ldots, k, n\}$. To do this, we may assume $\alpha < \beta$. 

If $\beta = n$, then  $N_{\mathcal{A}}(\b)$ is empty, whereas $X_{\alpha} \in N_{\mathcal{A}}(\alpha)$, whence  $N_{\mathcal{A}}(\alpha) \neq N_{\mathcal{A}}(\beta)$. Now assume $\beta \leqs k$ and suppose that $\alpha \in X_{\beta}$ and $\beta \in X_{\alpha}$. The containment of $\a$ in $X_{\b}$ implies that $(n-1)-\b+1+\a \leqs a_{\b}$ and thus $\b-\a \geqs n/2$ since $a_{\b} \leqs n/2$. Similarly, since $\b \in X_{\a}$ we have $\b-\a +1 \leqs a_{\a}$, so $\b-\a < n/2$ and we have reached a contradiction. It follows that either $\alpha \not\in X_{\beta}$ or $\beta \not\in X_{\alpha}$, whence $N_{\mathcal{A}}(\alpha) \neq N_{\mathcal{A}}(\beta)$ as required.

With a very similar argument, we can show that $N_{\mathcal{A}}(\alpha) \neq N_{\mathcal{A}}(\beta)$ for all $\alpha \in \{1, \ldots, k-1\}$ and $\beta\in \{k, \ldots, m\}$. In addition, we claim that $N_{\mathcal{A}}(\gamma)$ is empty for all $\gamma \geqs m+1$. This is clear if $m = n-1$, so we may assume $m < n-1$ and thus $X_k = \{k, k+1, \ldots, m\}$. For each $i \in \{1, \ldots, k\}$ we have  
\[
|X_i| = a_i \leqs a_k = m - k + 1 < \gamma - k + 1,
\]
so $\gamma \not\in X_i$ and $N_{\mathcal{A}}(\gamma)$ is indeed empty as claimed.
In particular, since $i \in X_i$ and $\delta \in X_k$ for all $\delta \in \{k, \ldots, m\}$, this means that $N_{\mathcal{A}}(\gamma) \ne N_{\mathcal{A}}(\delta)$ for all $\gamma \geqs m+1$ and all $\delta \in \{1, \ldots, m\}$.

Putting all of this together and applying Lemma \ref{lem:neighbourhood}, we deduce the following:

\vspace{1mm}

\begin{itemize}\addtolength{\itemsep}{0.3\baselineskip}
\item[\rm (a)] The points $1, \ldots, k, n$ are contained in distinct $H$-orbits.
\item[\rm (b)] Each $\a \in \{1, \ldots, k-1\}$ and $\beta\in \{k, \ldots, m\}$ are contained in distinct $H$-orbits.
\item[\rm (c)] Each $\a \in \{1, \ldots, m\}$ and $\beta\in \{m+1, \ldots, n\}$ are contained in distinct $H$-orbits.
\end{itemize} 

As a consequence, $H$ acts trivially on $\{1, \ldots, k - 1\}$ and it preserves the subsets $\{k, \ldots, m\}$ and $\{m + 1, \ldots, n\}$. In other words,
\[ 
H \leqs {\rm Sym}(\{k, \ldots, m\})\times{\rm Sym}(\{m+1, \ldots, n\})
\]
as required.
\end{proof}

We are now in a position to establish our main result on intransitive tuples.

\begin{prop}\label{p:rintrans}
For all $n \geqs 5$, we have $R_{\rm intrans}(S_n) = n-1$ and $R_{\rm intrans}(A_n) = n - 2$. 
\end{prop}

\begin{proof}
First observe that the intransitive $(n-2)$-tuple $(S_{n-1}, \ldots, S_{n-1})$ is non-regular, so $R_{\rm intrans}(S_n) \geqs n-1$. Similarly, $R_{\rm intrans}(A_n) \geqs n-2$, so we just need to establish the upper bounds $R_{\rm intrans}(S_n) \leqs n-1$ and $R_{\rm intrans}(A_n) \leqs n-2$.

Suppose $G = S_n$ or $A_n$, and let $\tau = (K_1, \ldots, K_{n - 2})$ be an $(n-2)$-tuple of intransitive subgroups of $G$. We may embed each $K_i$ in the setwise stabiliser $L_i = S_{a_i} \times S_{n-a_i}$ in $S_n$ of some subset of $[n]$ of size $a_i \leqs n/2$, and by reordering we may assume that $a_i\leqs a_{i+1}$ for all $i$. We claim that there exist elements $g_i \in G$ such that 
\begin{equation}\label{e:intrans}
\bigcap_{i = 1}^{n - 2}L_i^{g_i} \leqs \<(\alpha, \beta)\>,
\end{equation}
where $\a = n-1$ and $\beta\in \{n-2, n\}$. Since $G$ acts transitively on the set of $k$-element subsets of $[n]$ for any fixed $k \leqs n/2$, we just need to define a collection of subsets $X_1, \ldots, X_{n-2}$ of $[n]$ such that $|X_i| = a_i$ and $H = \bigcap_{i=1}^{n-2}H_i \leqs \< (\a,\b) \>$, where $H_i$ is the setwise stabiliser of $X_i$ in $S_n$. To do this, we can define the sets $X_1, \ldots, X_{n-2}$ as in \eqref{e:xi} (with $k=n-2$), and then by applying Lemma \ref{lem:orbits} we deduce that $H \leqs \<(\alpha, \beta)\>$ as required. 

For $G = A_n$ we immediately deduce that $\tau$ is regular and thus $R_{\rm intrans}(A_n) \leqs n-2$. So for the remainder, let us assume $G = S_n$ and let $\tau' = (K_1, \ldots, K_{n - 1})$ be an intransitive $(n-1)$-tuple. By the above argument, we have $\bigcap_{i =1}^{n - 2}K_i^{g_i} \leqs \<(\alpha, \beta)\>$ for some $g_i \in G$. Now $K_{n - 1}$ is contained in the setwise stabiliser in $G$ of some $k$-element subset of $[n]$ with $k \leqs n/2$. Of course, we can choose a $k$-set $X_{n-1}$ containing $\a$ but not $\b$, in which case the stabiliser of $X_{n-1}$ in $G$ does not contain the transposition $(\a,\b)$. In other words, there exists an element $g_{n-1} \in G$ such that 
\[
\bigcap_{i =1}^{n - 1}K_i^{g_i} \leqs \<(\alpha, \beta)\> \cap K_{n-1}^{g_{n-1}} = 1
\]
and we conclude that $\tau'$ is regular. This completes the proof of the proposition.
\end{proof}

\subsubsection{Imprimitive tuples}\label{ss:imprim}

In this section we focus on imprimitive tuples and our main result is Proposition \ref{prop:rimprim}, which states that 
\[
R_{{\rm imprim}}(S_n) \leqs n-1\; \mbox{ and }\; R_{{\rm imprim}}(A_n) \leqs n-2
\]
for all composite integers $n \geqs 6$ (recall that if $n$ is a prime, then every transitive subgroup of $S_n$ is primitive). To establish this bound for $G = S_n$, we need to prove that every $(n-1)$-tuple of the form $(S_{a_1} \wr S_{b_1}, \ldots, S_{a_{n-1}} \wr S_{b_{n-1}})$ is regular, where $n = a_ib_i$ and $a_i \geqs 2$ for all $i$. And similarly for $G = A_n$ with respect to $(n-2)$-tuples with components of the form $(S_{a_i} \wr S_{b_i}) \cap G$. Throughout this section we assume $n \geqs 6$ is composite and we write $[n] = \{1, \ldots, n\}$.

\begin{rem}\label{r:imprim}
Let $G = S_n$ or $A_n$ with $n \geqs 6$ composite. We have not attempted to seek sharper estimates on $R_{{\rm imprim}}(G)$ because the above bounds (together with some of the auxiliary results presented below) will be sufficient for our proof of Theorem \ref{t:main1}(i). However, it seems reasonable to expect that the bounds in Proposition \ref{prop:rimprim} are far from best possible. Indeed, the main theorem of \cite{MS} yields  
\[
B_{{\rm imprim}}(G) \leqs \log_2n+2,
\]
which leads us to speculate that $R_{{\rm imprim}}(G) = O(\log_2 n)$. With the aid of 
{\sc Magma}, we have checked that $R_{{\rm imprim}}(G) \leqs 6$ for all $n < 32$, but here we do not pursue this any further. 
\end{rem}

We begin by defining some useful terminology that arises naturally in this setting. We will say that a partition $X$ of $[n]$ is \emph{uniform} if every part has size $\ell$, where $1 < \ell < n$, and we will refer to $X$ as an \emph{$\ell$-partition}. If $X$ is an $\ell$-partition, then the stabiliser of $X$ in $S_n$ is defined to be the largest subgroup $H \leqs S_n$ that preserves the partition, so we have $H = S_\ell\wr S_{n/\ell}$. Similarly, the stabiliser in $A_n$ is the subgroup $H\cap A_n$.

Let $X_1$ and $X_2$ be uniform partitions of $[n]$ with parts of size $a$ and $b$, respectively. Choose an ordering of the parts $P_1, \ldots, P_{k}$ comprising $X_1$, where $k = n/a$, and assign the colour $c_i$ to every point in $P_i$, where $c_1, \ldots, c_{k}$ are distinct. Then the corresponding \emph{colour function} is the map $\chi: [n] \to \{c_1, \ldots, c_{k}\}$. We say that a part $Q$ of $X_2$ has \emph{colour sequence} $(c_{q_1}, \ldots, c_{q_b})$ if there is a bijection from $Q$ to the multiset of colours $\{c_{q_1}, \ldots, c_{q_b}\}$ such that each point in $Q$ is mapped to its assigned colour. And if $R$ is another part of $X_2$ with colour sequence $(c_{r_1}, \ldots, c_{r_b})$, we will say that $Q$ and $R$ are \emph{colour-equivalent} if there exists a permutation of the colour set $\{c_1, \ldots, c_{k}\}$ mapping the multiset $\{c_{q_1}, \ldots, c_{q_b}
\}$ to the multiset $\{c_{r_1}, \ldots, c_{r_b}\}$. Finally, we say that a part of $X_2$ is \emph{monochrome} if all of its points have the same colour. 

We begin with an easy lemma.

\begin{lem}\label{lem:colour-equivalence}
Let $X_1, X_2$ be defined as above and let $H_i$ be the stabiliser of $X_i$ in $S_n$. If $g\in H_1\cap H_2$ and $Q$ is a part of $X_2$, then $Q$ and $Q^g$ are colour-equivalent.
\end{lem}

\begin{proof}
Since $g \in H_1$, it follows that any two elements in $[n]$ with the same colour are mapped under $g$ to elements of the same colour. In particular, $Q$ and $Q^g$ are colour-equivalent.
\end{proof}

\begin{lem}\label{lem:swap}
Let $X$ be an $\ell$-partition of $[n]$ with parts $P_1, \ldots, P_k$, where $\ell \geqs 3$. Fix $\alpha \in P_i$ and $\beta \in P_j$, where $i \ne j$, and let $Y = X^t$ be the $\ell$-partition with parts $Q_i = P_i^t$, where $t = (\a,\b) \in S_n$. If $H_X$ and $H_Y$ denote the stabilisers of $X$ and $Y$ in $S_n$, then $H_X\cap H_Y$ fixes $\{\alpha, \beta\}$ setwise.
\end{lem}

\begin{proof}
For each $s \in \{1, \ldots,k\}$, assign the colour $c_s$ to the points in $P_s$, where $c_1, \ldots, c_k$ are distinct. Note that 
\[
Q_s = 
\left\{\begin{array}{ll}
P_s & \mbox{if $s\neq i, j$} \\
(P_{i}\setminus \{\alpha\})\cup \{\beta\} & \mbox{if $s = i$} \\
(P_{j}\setminus \{\beta\})\cup \{\alpha\} & \mbox{if $s = j$}
\end{array}\right.
\] 
for all $s \in \{1, \ldots, k\}$. Fix an element $g \in H_X\cap H_Y$.

Suppose $Q_i^g \neq Q_i$. Then since $Q_s = P_s$ is monochrome for all $s \neq i, j$, Lemma \ref{lem:colour-equivalence} implies that $Q_i^g = Q_j$ is the only option. Moreover, since any two points in $[n]$ of the same colour must be mapped to points of the same colour by any element in $H_X$, it follows that $\beta^g = \alpha$. Similarly, if $Q_i^g = Q_i$ then $\beta^g = \beta$.

If we now repeat the above argument, with $Q_i$ replaced by $Q_j$, we deduce that either $g$ interchanges $Q_i$ and $Q_j$, in which case $\beta^g = \alpha$ and $\alpha^g = \beta$, or $g$ fixes both $Q_i$ and $Q_j$, which means that $\a^g=\a$ and $\b^g = \b$. The result follows.
\end{proof}

As an application of Lemma \ref{lem:swap} we can now establish the following technical result, which will be an important ingredient in our proof of Proposition \ref{prop:rimprim}.

\begin{lem}\label{lem:pointwise}
Suppose $n = 2\ell$ is even and $A$ is an $m$-element subset of $[n]$ with $3\leqs m < n$. Then there exist $\ell$-partitions $X_1, \ldots, X_m$ of $[n]$ such that $\bigcap_{i=1}^m H_i$ fixes $A$ pointwise, where $H_i$ is the stabiliser of $X_i$ in $S_n$.
\end{lem}

\begin{proof}
It suffices to show that there exists a collection of $\ell$-partitions $Y_1, \ldots, Y_m$ with respective stabilisers $L_1, \ldots, L_m$ in $S_n$ such that $L = \bigcap_{i = 1}^m L_i$ fixes some $m$-set $B$ pointwise. Indeed, the partitions $X_i = Y_i^g$ will then have the required property, where $g \in S_n$ is chosen such that $B^g = A$.

We will define appropriate partitions $Y_i$, with parts labelled $P_{i,1}$ and $P_{i,2}$, and stabiliser $L_i$ in $S_n$. For $Y_1$ and $Y_2$ we set 
\[
P_{1, 1} = \{1, 2, \ldots, \ell\}, \;  P_{1, 2} = \{\ell+1, \ell + 2, \ldots, n\}
\]
and
\[
P_{2, 1} = \{2, 3, \ldots, \ell + 1\}, \; P_{2, 2} = \{1, \ell + 2, \ldots, n\}.
\]
We then define $Y_3, \ldots, Y_m$ iteratively, treating the cases $m \leqs \ell + 1$ and $m > \ell + 1$ separately.

First assume $m \leqs \ell + 1$. Here we define 
\begin{equation}\label{e:pi}
P_{i, 1} = \{1, \ldots, \ell + 1\}\setminus \{i - 1\},\; 
P_{i, 2} = \{\ell + 2, \ldots, n\} \cup \{i - 1\}
\end{equation}
for $i \geqs 3$, which means that  
\[
P_{2, 1} = (P_{1, 1}\setminus \{1\})\cup \{\ell + 1\}, \; P_{2, 2} = (P_{1, 2} \setminus \{\ell + 1\})\cup \{1\}
\]
and 
\[
P_{i , 1} = (P_{i - 1, 1} \setminus \{i - 1\})\cup \{i - 2\},\; P_{i , 2} = (P_{i - 1, 2} \setminus \{i - 2\})\cup \{i - 1\}
\]
for all $i \geqs 3$. Therefore, Lemma \ref{lem:swap} implies that $L_1 \cap L_2$ fixes $\{1, \ell + 1\}$ setwise, and similarly $L_{i - 1}\cap L_i$ fixes $\{i - 2, i - 1\}$ setwise for each $i \geqs 3$. In particular, $L = \bigcap_{i = 1}^m L_i$ fixes each of the following $2$-sets 
\[
\{1, \ell + 1\}, \{1, 2\}, \{2, 3\}, \ldots, \{m - 2, m - 1\},
\]
and hence $L$ fixes the $m$-set $\{1, 2, \ldots, m - 1, \ell + 1\}$ pointwise.

Finally, suppose $m > \ell + 1$. For $3\leqs i \leqs \ell + 1$ we define the two parts of $Y_i$ as in \eqref{e:pi}, and for $i> \ell + 1$ we set
\[
P_{i, 1} = \{2, \ldots, \ell\} \cup \{i\},\; 
P_{i, 2} = \{1, \ell + 1, \ldots, n\} \setminus \{i\}.
\]
As above, we deduce that $\bigcap_{i = 1}^{\ell+1} L_i$ fixes $\{1, \ldots, \ell+1\}$ pointwise. For $i = \ell + 2$, note that
\[
P_{i , 1} = (P_{2, 1} \setminus \{i - 1\})\cup \{i\},\; P_{i , 2} = (P_{2, 2} \setminus \{i\})\cup \{i - 1\}
\] 
and for $i > \ell + 2$ we have 
\[
P_{i , 1} = (P_{i - 1, 1} \setminus \{i - 1\})\cup \{i\},\; P_{i , 2} = (P_{i - 1, 2} \setminus \{i\})\cup \{i - 1\}.
\]
Therefore, by applying Lemma \ref{lem:swap}, we deduce that $\bigcap_{i = \ell+2}^{m} L_i$ fixes each of the $2$-sets 
\[
\{\ell + 1, \ell+2\}, \{\ell+2,\ell+3\}, \ldots, \{m - 1, m\}
\]
setwise and thus $L = \bigcap_{i = 1}^m L_i$ fixes $\{1, \ldots, m\}$ pointwise.
\end{proof}

We are now ready to present the following lemma, which will play an essential role in our proof of Proposition \ref{prop:rimprim}. It will also be used in Section \ref{ss:final}, where we complete the proof of Theorem \ref{t:main1}. The proof is long and technical, so we partition the argument into a number of separate and clearly defined cases.

\begin{lem}\label{lem:2sets}
Let $G = S_n$ or $A_n$, and let $X_i$ be an $a_i$-partition of $[n]$ for $i = 1,2$, where $a_1 \geqs a_2$. Let $\a,\b \in [n]$ be distinct and assume that either 

\vspace{1mm}

\begin{itemize}\addtolength{\itemsep}{0.3\baselineskip}
\item [{\rm (i)}] $n/2> a_1 \geqs 3$; or 
\item [{\rm (ii)}] $a_1 = n/2$, $a_2 \geqs 3$ and $\alpha$ and $\beta$ are contained in different parts of $X_1$.
\end{itemize}
Then there exists an element $g \in G$ such that $H_1\cap H_2^g$ fixes $\{\alpha, \beta\}$ setwise, where $H_i$ is the stabiliser of $X_i$ in $G$.
\end{lem}

\begin{proof}
Write $k = n/a_1$, $m = n/a_2$ and let $P_1, \ldots, P_k$ be the parts of $X_1$. We assign the colour $c_i$ to the points in $P_i$, where $c_1, \ldots, c_k$ are distinct, and we let $\chi$ be the corresponding colour function on $[n]$. In order to prove the lemma, we need to construct an $a_2$-partition $X_3$ of $[n]$ with stabiliser $H_3$ such that $H_1 \cap H_3$ fixes $\{\a,\b\}$ setwise. Indeed, the transitivity of $G$ on the set of $a_2$-partitions of $[n]$ implies that $H_3 = H_2^g$ for some $g \in G$ and the desired result follows.

First assume (i) holds, so $n/2> a_1 \geqs 3$ and $k = n/a_1 \geqs 3$. There are two different cases to consider, according to whether or not $a_2 = 2$.

\vs

\noindent \emph{Case 1.} $n/2 > a_1 \geqs a_2\geqs 3$

\vs

We will write $P_i = \{p_{i1}, \ldots, p_{ia_1}\}$ for each $i$ and we consider two subcases, according to the position of $\a$ and $\b$ in the parts of $X_1$.

\vs

\noindent \emph{Case 1(a).} \emph{$\alpha$ and $\beta$ are in different parts of $X_1$}

\vs

Without loss of generality, we may assume that $\alpha = p_{11} \in P_1$ and $\beta = p_{21} \in P_2$. We first construct an $a_2$-partition $Y$ of $[n]$ with parts $Q_1, \ldots, Q_m$ and stabiliser $H_Y$ in $G$ by setting 
\[
Q_1 = \{\alpha, p_{22}, \ldots, p_{2a_2}\},\;\; Q_2 = \{\beta, p_{12}, \ldots, p_{1a_2}\}
\]
and we define $Q_3, \ldots, Q_m$ by partitioning the remaining points in $[n]$ into sets of size $a_2$ in the following order:
\[
p_{1(a_2 + 1)}, \ldots, p_{1a_1}, p_{31}, \ldots, p_{3a_1}, p_{2(a_2 + 1)}, \ldots, p_{2a_1}, p_{41}, \ldots, p_{4a_1}, \ldots, p_{k1}, \ldots, p_{ka_1}.
\]
Notice that $Q_1$ and $Q_2$ are the only parts of $Y$ with respective colour sequences $(c_1, c_2, \ldots, c_2)$ and $(c_2, c_1, \ldots, c_1)$. Moreover, if $i \ne j$ then our construction implies that $Y$ has parts with colour sequences $(c_i, c_j, \ldots, c_j)$ and $(c_j, c_i, \ldots, c_i)$ if and only if $\{i, j\} = \{1, 2\}$.

Let $x \in H_1\cap H_Y$. Then Lemma \ref{lem:colour-equivalence} implies that  $Q_1^x$ has colour sequence $(c_i, c_j, \ldots, c_j)$ for some $i\neq j$. Since $a_2 \geqs 3$, we deduce that $P_1^x = P_i$ and $P_2^x = P_j$, which in turn implies that $Q_2^x$ has colour sequence $(c_j, c_i, \ldots, c_i)$. So as noted above, we have $(i,j) = (1,2)$ or $(2,1)$. If $(i,j) = (1,2)$ then $Q_1^x = Q_1$ and we deduce that $\a^x = \a$. On the other hand, if $(i,j) = (2,1)$ then $Q_1^x = Q_2$ and $\a^x = \b$. 

Similarly, by repeating the argument with $Q_1$ replaced by $Q_2$, we deduce that $\b^x \in \{\a,\b\}$. Therefore, $H_1 \cap H_Y$ fixes $\{\a,\b\}$ setwise, as required.

\vs 

\noindent \emph{Case 1(b).} \emph{$\alpha$ and $\beta$ are contained in the same part of $X_1$}

\vs

By relabelling, we may assume that $\a,\b \in P_1$, say $\alpha = p_{11}$ and $\beta = p_{12}$. We now define a new $a_2$-partition $Y$ of $[n]$ with parts $Q_1, \ldots, Q_m$ and stabiliser $H_Y$ in $G$. To construct $Y$, we first set  
\[
Q_1 = \{\alpha, p_{21}, \ldots, p_{2(a_2-1)}\},\;\; Q_2 = \{\beta, p_{31}, \ldots, p_{3(a_2 - 1)}\},
\]
so  $Q_1$ and $Q_2$ have respective colour sequences $(c_1,c_2, \ldots, c_2)$ and $(c_1, c_3, \ldots, c_3)$. We then divide the remaining points of $[n]$ into parts $Q_3, \ldots, Q_m$ of size $a_2$ in such a way that none of these parts has colour sequence $(c_1, c_i, \ldots, c_i)$ with $i \geqs 2$. 

Observe that we can always construct such a partition. To justify this claim, first assume $a_1 = 3$, so $a_2 = 3$ and $k =m \geqs 3$. Here we can set $Q_3 = \{p_{13}, p_{23}, p_{33}\}$ and then define $Q_4, \ldots, Q_m$ arbitrarily, noting that none of the latter parts will contain any points with colour $c_1$. Now assume $a_1 \geqs 4$ and write $a_1-2 = sa_2+r$, where $s \geqs 0$ and $0 \leqs r < a_2$. Here we can define $Q_3, \ldots, Q_{s+2}$ so that they all have colour sequence $(c_1, \ldots, c_1)$. If $r\ne 1$ then we can define $Q_{s+3}$ so that it contains exactly $r$ points with colour $c_1$, while the remaining parts $Q_{s+4}, \ldots, Q_m$ contain no points from $P_1$. And if $r=1$ then we can choose $Q_{s+3}$ so that it contains points from $P_1$, $P_2$ and $P_3$, while none of the remaining parts in $Y$ contain any points with colour $c_1$. This justifies the claim. 

Before continuing the analysis of Case 1(b), let us pause to provide a brief outline of the argument. We are seeking to construct an $a_2$-partition $W$ of $[n]$ with stabiliser $H_W$, containing the parts $Q_1$ and $Q_2$ defined above, with the property that $\{Q_1,Q_2\}^x = \{Q_1,Q_2\}$ for all $x \in H_1 \cap H_W$. Indeed, if we have such a partition, then $H_1 \cap H_W$ fixes $\{\a,\b\}$ and the result follows. Set $Y_0 = Y$, which is the $a_2$-partition defined above. If $Y_0$ does not have the desired property, then we will construct a new partition $Y_1 = Y_0^g$ with parts $Q_i^g$ for some carefully chosen transposition $g \in S_n$. We can then repeat this process and we will prove that it produces a partition $W = Y_t$ with the desired property after finitely many iterations.

We start by inspecting the partition $Y$ constructed above. Fix an element $x \in H_1\cap H_Y$ and consider the image of $Q_1$ under $x$. 

First assume $Q_1^x = Q_1$. Here $P_1^x = P_1$ and $P_2^x = P_2$ since $a_2 \geqs 3$, whence $\a^x =\a$ and $Q_2^x$ has colour sequence $(c_1, c_i, \ldots, c_i)$ for some $i \geqs 2$. Then by the construction of $Y$, it follows that $Q_2^x = Q_2$ and thus $\b^x = \b$. Similarly, if $Q_1^x = Q_2$ then $P_1^x = P_1$ and $\alpha^x = \beta$, which in turn implies that $Q_2^x = Q_1$ and $\b^x = \a$. 

So if $Q_1^x \in \{Q_1,Q_2\}$ for all $x \in H_1 \cap H_Y$, then $H_1 \cap H_Y$ fixes $\{\a,\b\}$ and we are done. So let us assume there exists an element $x \in H_1 \cap H_Y$ such that $Q_1^x = Q_a$ for some $a \geqs 3$. Now $Q_1^x$ has colour sequence $(c_i, c_j, \ldots, c_j)$, where $i \ne j$ and $i \geqs 2$. 
More precisely, $\chi(\alpha^x) = c_i$ and therefore $\chi(\beta^x) = c_i$ since $\a$ and $\b$ are in the same part of $X_1$. It follows that $Q_2^x$ has colour sequence $(c_i, c_\ell, \ldots, c_\ell)$ for some $\ell \ne i$, and we note that $\ell \ne j$ since $P_2^x \ne P_3^x$. Fix $\gamma_1\in Q_1^x$ and $\delta_1\in Q_2^x$ such that $\chi(\gamma_1) = c_j$ and $\chi(\delta_1) = c_\ell$, and let $g_1 \in S_n$ be the transposition $(\gamma_1, \delta_1)$. Note that the parts of $Y$ containing $\gamma_1$ and $\delta_1$ are both colour-equivalent to $Q_1$.

At this point, we now switch to the $a_2$-partition $Z = Y_1 = Y^{g_1}$, with parts $R_\ell = Q_\ell^{g_1}$ for $1 \leqs \ell \leqs m$. Let $H_{Z}$ be the stabiliser of $Z$ in $G$ and let $x \in H_1\cap H_Z$. Notice that $R_1 = Q_1$ and $R_2 = Q_2$ since $\gamma_1, \delta_1 \not\in Q_1\cup Q_2$. 

By arguing as above, we deduce that either $R_1^x = R_1$ and $\a^x = \a$, or $R_1^x = R_2$ and $\a^x = \b$, or $R_1^x = R_a$ for some $a \geqs 3$. Let us assume that we are in the latter situation. Then as above, we note that $R_1^x$ has colour sequence $(c_i, c_j, \ldots, c_j)$ for some $i\geqs 2$ with $i \ne j$, and it follows that $R_2^x$ has colour sequence $(c_i, c_\ell, \ldots, c_\ell)$ for some $\ell \ne i,j$. We now choose $\gamma_2\in R_1^x$ and $\delta_2\in R_2^x$ such that $\chi(\gamma_2) = c_j$ and $\chi(\delta_2) = c_\ell$, and we construct the $a_2$-partition $Y_2 = Y_1^{g_2} = Y^{g_1g_2}$, where $g_2 = (\gamma_2, \delta_2) \in S_n$.

We can now repeat this process until after $t$ steps we obtain an $a_2$-partition $W = Y_t = Y^{g}$ of $[n]$ containing the parts $Q_1$ and $Q_2$, where $g = g_1\cdots g_t$ is a product of transpositions and $Q_1^x \in \{Q_1,Q_2\}$ for all $x \in H_1\cap H_W$. We claim that this process does indeed produce such a partition after finitely many iterations. To see this, first observe that if  $a \in \{2, \ldots, t\}$, then the parts comprising the partition $Y_a = Y^{g_1\cdots g_a}$ are as follows:
\begin{equation}\label{eq:parts}
Q_j^{g_1\cdots g_{a}} = 
\begin{cases}
Q_j^{g_1\cdots g_{a - 1}} & \mbox{ if $\{\gamma_{a}, \delta_{a}\}\not\in Q_j^{g_1\cdots g_{a - 1}}$} \\
Q' = (Q_j^{g_1\cdots g_{a - 1}}\setminus \{\gamma_{a}\})\cup \{\delta_{a }\} &\mbox{ if $\gamma_{a }\in Q_j^{g_1\cdots g_{a - 1}}$} \\
Q'' = (Q_j^{g_1\cdots g_{a - 1}}\setminus \{\delta_{a }\})\cup \{\gamma_{a }\} & \mbox{ if $\delta_{a }\in Q_j^{g_1\cdots g_{a - 1}}$}
\end{cases}
\end{equation}
In particular, notice that the parts of $Y_{a-1} = Y^{g_1\cdots g_{a - 1}}$ containing $\gamma_{a}$ and $\delta_{a}$ are both colour-equivalent to $Q_1 = Q_1^{g_1\cdots g_{a - 1}}$, whereas $Q'$ and $Q''$ are not. Therefore, the number of parts that are colour-equivalent to $Q_1$ strictly decreases with each iteration, and this justifies the claim since our initial partition $Y$ has only finitely many parts.

So we now have an $a_2$-partition $W$ containing the parts $Q_1$ and $Q_2$ such that $Q_1^x \in \{Q_1,Q_2\}$ for all $x \in H_1 \cap H_W$. In particular, this implies that $P_1^x = P_1$ and $\a^x \in \{\a,\b\}$ for all $x \in H_1 \cap H_W$. So to complete the analysis of Case 1(b), it suffices to show that $Q_1$ and $Q_2$ are the only parts in $W$ with colour sequence $(c_1, c_i, \ldots, c_i)$ for some $i\geqs 2$. Indeed, if this property is satisfied then $Q_2^x \in \{Q_1, Q_2\}$ and $\beta^x \in \{\alpha, \beta\}$ for all $x \in H_1\cap H_W$, which means that every element in $H_1 \cap H_W$ fixes $\{\a,\b\}$ setwise, as required.

So it just remains to justify the colour sequence claim, which we will do via induction on $t$, recalling that $W = Y_t = Y^g$ and $g = g_1 \cdots g_t$ as defined above. For the base case $t=0$ we have $W = Y$ and we know that $Y$ has the desired property by construction. Now assume $t \geqs 1$ and suppose that $Q_1$ and $Q_2$ are the only parts in $Y_{t-1}$ with colour sequence $(c_1, c_i, \ldots, c_i)$ for some $i \geqs 2$. From \eqref{eq:parts}, we see that every part in $Y_t$, other than $Q'$ and $Q''$, is also a part of $Y_{t-1}$, so $Q_1$ and $Q_2$ are the only parts in $Y_t \setminus \{Q', Q''\}$ with colour sequence $(c_1, c_i, \ldots, c_i)$ for some $i \geqs 2$. The claim now follows since it is clear to see that neither $Q'$ nor $Q''$ is colour-equivalent to $Q_1$ or $Q_2$.

This completes the proof of the lemma in Case 1.

\vs 

\noindent \emph{Case 2.} $n/2 > a_1 \geqs 3$, $a_2 = 2$

\vs 

In order to complete the proof of the lemma in case (i), we may assume $a_2 = 2$. As above, set $k = n/a_1$ and let $P_1, \ldots, P_k$ be the parts of $X_1$. For each $i$, it will be convenient to write $p_{(i - 1)a_1 + 1}, \ldots, p_{ia_1}$ for the points in $P_i$, to which we assign the colour $c_i$ as before. Let $\chi:[n] \to \{c_1, \ldots, c_k\}$ be the corresponding colour function.

\vs 

\noindent \emph{Case 2(a).} \emph{$\alpha$ and $\beta$ are contained in the same part of $X_1$}

\vs

We may assume $\a,\b \in P_1$, say $\a = p_1$ and $\b = p_2$. Consider the $2$-partition $Y$ of $[n]$ with parts
\[
\{p_1, p_2\}, \{p_3, p_{n/2 + 2}\}, \{p_4, p_{n/2 + 3}\}, \ldots, \{p_{n/2 + 1}, p_n\}
\]
and let $H_Y$ be the stabiliser of $Y$ in $G$. Note that every part of $Y$, other than $\{p_1, p_2\}$, is of the form $\{p_i, p_{n/2 - 1 + i}\}$ for some $i \geqs 3$. Since the parts in $X_1$ have size $a_1 < n/2$, it follows that $p_i$ and $p_{n/2 -1 + i}$ cannot both be contained in the same part of $X_1$. Therefore, no part of $Y$, other than $\{p_1, p_2\}$, is colour-equivalent to $\{\alpha, \beta\}$, and by applying Lemma \ref{lem:colour-equivalence} we deduce that $\{\alpha, \beta\}$ is fixed by $H_1\cap H_Y$. 

\vs 

\noindent \emph{Case 2(b).} \emph{$\alpha$ and $\beta$ are contained in different parts of $X_1$}

\vs

We may assume $\a \in P_1$ and $\b \in P_2$, say $\a = p_1$ and $\b = p_{a_1 + 1}$. We now consider two subcases, according to the parity of $a_1$.

First assume $a_1 \geqs 3$ is odd, which means that $k = n/a_1$ is even. Define a $2$-partition $Y$ of $[n]$ with the following parts 
\[
\begin{array}{l}
\{p_1, p_{a_1 + 1}\}, \{p_2, p_3\}, \ldots, \{p_{a_1-1}, p_{a_1}\},\{p_{a_1+2}, p_{a_1+3}\},\ldots, \{p_{2a_1 - 1},p_{2a_1}\},\\ 
\{p_{2a_1 + 1}, p_{3a_1 + 1}\}, \ldots, \{p_{3a_1}, p_{4a_1}\}, \ldots,  
\{p_{n - 2a_1+1}, p_{n-a_1+1}\},\ldots, \{p_{n - a_1}, p_n\}
\end{array}
\]
Note that every part in the second row is of the form $\{p_i, p_{a_1 + i}\}$ for some $i\geqs 2a_1 + 1$. We claim that each $x \in H_1 \cap H_Y$ fixes $\{\a,\b\}$ setwise.

First note that $\{\a,\b\}$ cannot be mapped by $x$ to any of the following parts
\[
\{p_2, p_3\}, \ldots, \{p_{a_1 - 1}, p_{a_1}\}, \{p_{a_1 + 2}, p_{a_1 + 3}\}, \ldots, \{p_{2a_1 - 1}, p_{2a_1}\}
\]
since the latter are all monochrome (see Lemma \ref{lem:colour-equivalence}).  Now assume $\{\alpha, \beta\}^x = \{p_i, p_{a_1 + i}\}$ for some $i\geqs 2a_1 + 1$. Then since $\alpha\in P_1$ and $\chi(\alpha^x)\neq c_1, c_2$, it follows that $P_1^x = P_j$ for some $j \geqs 3$ and thus $\{p_2, p_3\} \subseteq P_1$ must be mapped by $x$ to some part of $Y$ with colour sequence $(c_j, c_j)$. But no part of $Y$ has such a colour sequence, so this possibility cannot arise and we conclude that $H_1 \cap H_Y$ fixes $\{\a,\b\}$ setwise.

To complete the argument in Case 2(b), we may assume $a_1\geqs 4$ is even. Here we define a $2$-partition $Y$ with parts 
\[
\{p_1, p_{2a_1+1}\}, \{p_{a_1 + 1}, p_{2a_1 + 2}\}, \{p_2, p_{a_1 + 2}\}, \ldots, \{p_{a_1}, p_{2a_1}\}, \{p_{2a_1 + 3}, p_{2a_1 + 4}\}, \ldots, \{p_{n - 1}, p_n\}
\]
and our goal is to show that $H_1 \cap H_Y$ fixes $\{\a,\b\}$ setwise. Let $x \in H_1 \cap H_Y$ and recall that $\a=p_1$ and $\b = p_{a_1+1}$. We proceed by considering the images of the parts $\{p_1, p_{2a_1+1}\}$ and $\{p_{a_1 + 1}, p_{2a_1 + 2}\}$ under $x$.

First observe that $\{p_1, p_{2a_1+1}\}$ and $\{p_{a_1 + 1}, p_{2a_1 + 2}\}$ are non-monochrome, so neither of them can be moved by $x$ to any of the monochrome parts $\{p_{2a_1 + 3}, p_{2a_1 + 4}\}, \ldots, \{p_{n - 1}, p_n\}$.

Next assume $\{p_1, p_{2a_1 + 1}\}^x = \{p_i, p_{a_1 + i}\}$ for some $i\in \{2, \ldots, a_1\}$. This implies that the point $p_{2a_1 + 1} \in P_3$ has colour $c_1$ or $c_2$, so $P_3^x \in \{P_1, P_2\}$ and thus $\{p_{2a_1 + 3}, p_{2a_1 + 4}\}^x$ has colour sequence $(c_j, c_j)$ for some $j \in \{1, 2\}$. However, there is no such part in $Y$, so this situation does not arise. This means that $\{p_1, p_{2a_1 + 1}\}^x \ne \{p_i, p_{a_1 + i}\}$ with $i\in \{2, \ldots, a_1\}$, and similarly we deduce that the same conclusion holds for $\{p_{a_1+1}, p_{2a_1 + 2}\}^x$.

We have now shown that $x$ either fixes the parts $\{p_1, p_{2a_1 + 1}\}$ and $\{p_{a_1 + 1}, p_{2a_1 + 2}\}$, or it interchanges them. 

We claim that if $x$ fixes $\{p_1, p_{2a_1 + 1}\}$, then it must fix $p_1 = \a$. Suppose otherwise, so $\alpha^x = p_{2a_1 + 1}$. Since $x\in H_1$, it follows that $P_1^x = P_3$, so the parts $\{p_2, p_{a_1 + 2}\}, \ldots, \{p_{a_1}, p_{2a_1}\}$ must all be mapped by $x$ to some part of $Y$ with colour sequence $(c_3, c_i)$ for some $i\neq 3$. However, there is only one such part (other than $\{p_1, p_{2a_1 + 1}\}$ itself), namely $\{p_{a_1 + 1}, p_{2a_1 + 2}\}$, so we have reached a contradiction since $a_1\geqs 3$. Similarly, since $x$ fixes $\{p_{a_1 + 1}, p_{2a_1 + 2}\}$, we deduce that $\b^x = \b$. 

Finally, let us assume $x$ interchanges the parts $\{p_1, p_{2a_1 + 1}\}$ and $\{p_{a_1 + 1}, p_{2a_1 + 2}\}$. As in the previous paragraph, we know that $P_1^x \neq P_3$, so we must have $\alpha^x = p_1^x = p_{a_1 + 1} = \beta$. This implies that $P_3^x = P_3$ and therefore $\b^x = \a$. 

In conclusion, every element in $H_1 \cap H_Y$ fixes $\{\a,\b\}$ setwise and this completes the proof of the lemma in case (i).

\vs

Finally, let us assume (ii) holds, so $a_1 = n/2$, $a_2 \geqs 3$ and $\alpha$ and $\beta$ lie in different parts of $X_1$. If $a_2 = n/2$ then the desired conclusion follows from Lemma \ref{lem:swap}, so we may assume that $a_2 < n/2$. As before, set $m = n/a_2$ and note that $m \geqs 3$.

Write $P_1 = \{p_1, \ldots, p_{n/2}\}$ and $P_2 = \{p_{n/2 + 1}, \ldots, p_n\}$ for the parts of $X_1$ and assign the colours $c_1$ and $c_2$, respectively. We may assume that $\alpha = p_1$ and $\beta = p_{n/2 + 1}$. We now construct an $a_2$-partition  $Y$ of $[n]$ with parts $Q_1, \ldots, Q_m$ and stabiliser $H_Y$ in $G$, where 
\[
Q_1 = \{p_1, p_{n/2 + 2}, \ldots, p_{n/2 + a_2}\}, \;\; Q_2 = \{p_{n/2 + 1}, p_2, \ldots, p_{a_2}\}
\]
and we define $Q_3, \ldots, Q_m$ by partitioning the remaining points in $[n]$ into subsets of size $a_2$ with respect to the ordering $p_{a_2+1}, \ldots, p_{n/2}, p_{n/2 +a_2+1}, \ldots, p_n$. To complete the proof, it suffices to show that every element in $H_1 \cap H_Y$ fixes $\{\a,\b\}$ setwise.

To see this, we first claim that $Q_2$ is the only part of $Y$ which is colour-equivalent to $Q_1$ (other than $Q_1$ itself). Seeking a contradiction, suppose that $Q_i$ is colour-equivalent to $Q_1$ for some $i\geqs 3$. From the construction of $Y$, this is only possible if $n/2 \equiv \pm 1 \imod{a_2}$, which means that $a_2$ divides $n \mp 2$ and $n$. But this is absurd since $a_2 \geqs 3$. This justifies the claim. Similarly, $Q_1$ is the only part of $Y$, other than $Q_2$, which is colour-equivalent to $Q_2$.

Let $x \in H_1 \cap H_Y$ and suppose that $Q_1^x \ne Q_1$. From the previous claim we see that $Q_1^x = Q_2$ and $Q_2^x = Q_1$, which in turn implies that $\a^x = \b$ and $\b^x = \a$ since $a_2 \geqs 3$. Finally, let us assume $Q_1^x = Q_1$. Then $Q_2^x = Q_2$ and we deduce that $\alpha^x = \alpha$ and $\b^x = \b$. So in all cases, we conclude that $x$ fixes $\{\a,\b\}$ setwise and the proof of the lemma is complete.
\end{proof}

\begin{rem}
Notice that in Lemma \ref{lem:2sets}(ii) we assume the points $\a$ and $\b$ are contained in different parts of $X_1$. Indeed, the conclusion is sometimes false if we drop this condition (for instance, it is straightforward to check that it fails to hold when $a_1 = a_2 = n/2$).
\end{rem}

We need one more ingredient before we are in position to prove our main result on imprimitive tuples. This is a special case of a more general result on base sizes for the action of $S_n$ and $A_n$ on partitions (see \cite[Proposition 2.4 and Remark 2.8]{BGL2}). 

\begin{prop}\label{p:msbase}
Let $G = S_n$ or $A_n$, where $n \geqs 8$ is even, and let $H = S_2 \wr S_{n/2}$ be the stabiliser in $S_n$ of a $2$-partition of $[n]$. Then there exist $x,y \in G$ such that $H \cap H^x \cap H^y = 1$.
\end{prop}

We are now ready to prove our main result on imprimitive tuples.

\begin{prop}\label{prop:rimprim}
Let $n \geqs 6$ be a composite integer. Then 
\[
R_{\rm imprim}(S_n) \leqs n-1,\;\; R_{\rm imprim}(A_n) \leqs n-2.
\]
\end{prop}

\begin{proof}
Let $G = S_n$ or $A_n$ and let $\tau = (J_1, \ldots, J_k)$ be a $k$-tuple of imprimitive subgroups of $G$ with $k = n-|S_n:G|$. We need to show that $\tau$ is regular.

To do this, we may first embed each $J_i$ in a maximal imprimitive subgroup $L_i = S_{a_i} \wr S_{b_i}$ of $S_n$, where $n = a_ib_i$ and $a_i \geqs 2$ for all $i$, and we may assume that 
\[
n/2 \geqs a_1\geqs a_2 \geqs \cdots \geqs a_{n - 2} \geqs 2.
\]
It suffices to show that there exist elements $g_i \in G$ such that 
\begin{equation}\label{e:li}
\left(\bigcap_{i = 1}^{k}L_i^{g_i}\right) \cap G  = 1.
\end{equation}
The result for $n \leqs 7$ can be checked with the aid of {\sc Magma}, so we will assume $n \geqs 8$. Our goal is to prove the following claim:

\vs

\noindent \textbf{Claim $(\star)$.} \emph{There exist elements $g_i \in G$ such that $\bigcap_{i = 1}^{n - 2}L_i^{g_i} \leqs \<(n-1, n)\>$}.

\vs

Indeed, suppose the claim is true. For $G = A_n$ we have $k = n-2$, whence \eqref{e:li} holds and the result follows. Now assume $G = S_n$ and let $X_{n-1}$ be any $a_{n-1}$-partition of $[n]$ with the property that $n-1$ and $n$ are contained in different parts. Then the stabiliser of $X_{n-1}$ in $G$ is of the form $L_{n-1}^{g_{n-1}}$ for some $g_{n-1} \in G$ and we deduce that \eqref{e:li} holds, which completes the proof of the proposition.

So it just remains to establish Claim $(\star)$. In view of the transitivity of $G$ on the set of $\ell$-partitions of $[n]$ for any $\ell$, it suffices to construct $a_i$-partitions $X_{i}$ of $[n]$ for $1 \leqs i \leqs n-2$ such that 
\[
H = \bigcap_{i = 1}^{n - 2}H_i \leqs \<(n-1, n)\>,
\]
where $H_i$ is the stabiliser of $X_i$ in $S_n$. (In the course of the proof, we will define various $a_i$-partitions $X_i$ and we will always write $H_i$ for the stabiliser of $X_i$ in $S_n$.)

First observe that if $a_{n-4} = 2$, then each $L_i$ with $i \in \{n-4,n-3,n-2\}$ is of the form $S_2 \wr S_{n/2}$, in which case Proposition \ref{p:msbase} implies that 
\[
L_{n-4} \cap L_{n-3}^x \cap L_{n-2}^y = 1
\]
for some $x,y \in G$. So for the remainder of the proof, we may assume that $a_{n - 4} \geqs 3$. We now divide the argument into four cases.

\vs 

\noindent \emph{Case 1.} $a_1 < n/2$

\vs

The argument in this case follows very easily from Lemma \ref{lem:2sets}. Consider the $a_1$-partition
\[
X_1 = \{\{1, 2, \ldots, a_1\}, \{a_1 + 1, \ldots, 2a_1\}, \ldots, \{n - a_1 + 1, \ldots, n\}\}.
\]

Suppose $a_{n - 3} \geqs 3$. By applying Lemma \ref{lem:2sets}(i), we can define an $a_2$-partition $X_2$ such that $H_1 \cap H_2$ fixes $\{1,2\}$. The same lemma also implies that there is an $a_3$-partition $X_3$ so that $H_2 \cap H_3$ fixes $\{2,3\}$. And by iteratively applying the same lemma, we can define an $a_i$-partition $X_i$ such that $H_{i - 1}\cap H_i$ fixes $\{i - 1, i\}$ for all $4 \leqs i \leqs n-2$. It follows that $H = \bigcap_{i = 1}^{n - 2} H_i$ fixes $\{1, \ldots, n-2\}$ pointwise and thus $H \leqs \<(n-1, n)\>$ as required. 

The case $a_{n - 3}  = 2$ is very similar. By arguing as above, using $a_{n-4} \geqs 3$, we can define an $a_i$-partition $X_i$ for $2 \leqs i \leqs n-3$ such that $\bigcap_{i=1}^{n-3}H_i$ acts trivially on $\{1, \ldots, n-3\}$. We now define the $2$-partition 
\[
X_{n - 2} = \{\{1, n - 2\}, \{2, n-1\}, \{3, n\}, \{4, 5\}, \ldots, \{n - 4, n - 3\}\}.
\]
Since $\bigcap_{i = 1}^{n - 3} H_i$ fixes $1$, $2$ and $3$, it follows that $\bigcap_{i = 1}^{n - 2}H_i$ must also fix $n-2$, $n-1$ and $n$, which forces $\bigcap_{i = 1}^{n - 2}H_i = 1$.

\vs

\noindent \emph{Case 2.} \emph{$a_1 = n/2$ and $a_2 < n/2$}

\vs

Set 
\[
X_1 = \{\{1, 3, \ldots, n/2+1\}, \{2, n/2+2, \ldots, n\}\}
\]
and assume $a_{n - 3} \geqs 3$ for now. Since the points $1$ and $2$ are contained in different parts of $X_1$, Lemma \ref{lem:2sets}(ii) implies that we can find an $a_2$-partition $X_2$ such that $H_1\cap H_2$ fixes $\{1, 2\}$. We can then appeal to Lemma \ref{lem:2sets}(i), as in Case 1, to iteratively construct an  $a_i$-partition $X_i$ for $3 \leqs i \leqs n-2$ such that 
$\bigcap_{i = 2}^{n - 2}H_i$ fixes each of the subsets 
\[
\{2, 3\}, \{3,4\}, \ldots, \{n - 3, n - 2\}.
\]
Consequently, $\bigcap_{i = 1}^{n - 2}H_i$ fixes $\{1, \ldots, n-2\}$ pointwise and the desired result follows.

The case $a_{n - 3} = 2$ is very similar. Here we define $X_2, \ldots, X_{n - 3}$ as above, so that $\bigcap_{i = 1}^{n - 3} H_{i}$ fixes $\{1, \ldots, n - 3\}$ pointwise. By setting   
\[
X_{n - 2} = \{\{1, n - 2\}, \{2, n-1\}, \{3, n\}, \{4, 5\}, \ldots, \{n - 4, n-3\}\}
\]
as in Case 1, we deduce that $\bigcap_{i = 1}^{n - 2}H_i = 1$.

\vs

\noindent \emph{Case 3.} \emph{$a_1 = a_2 = n/2$ and $a_3 < n/2$}

\vs

First note that $a_3 \geqs 3$ since $n \geqs 8$ and $a_{n-4} \geqs 3$. We define two $(n/2)$-partitions 
\begin{align*}
X _1 &= \{\{1, 4, \ldots, n/2 + 2\}, \{2, 3, n/2 + 3, \ldots, n\}\}\\
X_2 &= \{\{2, 4, \ldots, n/2 + 2\}, \{1, 3, n/2 + 3, \ldots, n\}\}
\end{align*}
and we note that $H_1\cap H_2$ fixes $\{1, 2\}$ by Lemma \ref{lem:swap}. In addition, since $2$ and $3$ are contained in different parts of $X_2$, Lemma \ref{lem:2sets}(ii) allows us to construct an $a_3$-partition $X_3$ such that $H_2\cap H_3$ fixes $\{2, 3\}$. 

Suppose $a_{n - 3} \geqs 3$. By applying Lemma \ref{lem:2sets}(i), we can iteratively construct an $a_i$-partition $X_i$ for $4 \leqs i \leqs n - 2$ such that $H_{i - 1}\cap H_{i}$ fixes $\{i - 1, i\}$ for all such $i$. It follows that $\bigcap_{i = 1}^{n - 2}H_i$ fixes the $2$-sets
\[
\{1, 2\}, \{2, 3\}, \ldots, \{n - 3, n - 2\}
\]
and thus $\bigcap_{i = 1}^{n - 2}H_i \leqs \<(n-1, n)\>$.

Now assume $a_{n - 3} = 2$. As above, we can appeal to Lemma \ref{lem:2sets}(i) in order to construct partitions $X_4, \ldots, X_{n - 3}$ such that $H_{i - 1}\cap H_i$ fixes $\{i - 1, i\}$ for all $i\in \{4, \ldots, n-3\}$. Then $\bigcap_{i = 1}^{n - 3} H_i$ fixes $\{ 1, \ldots, n - 3\}$ pointwise and we complete the argument by defining the $2$-partition
\[
X_{n - 2} = \{\{1, n -2\}, \{2, n-1\}, \{3, n\}, \{4, 5\}, \ldots, \{n - 4, n - 3\}\},
\]
which forces $\bigcap_{i = 1}^{n - 2} H_i = 1$.

\vs

\noindent \emph{Case 4.} \emph{$a_i = n/2$ for some $i\geqs 3$}

\vs

Let $k \in \{1, \ldots, n-2\}$ be maximal such that $a_k = n/2$. By appealing to 
Lemma \ref{lem:pointwise}, there exist $(n/2)$-partitions $X_1, \ldots, X_k$ such that $\bigcap_{i = 1}^k H_i$ fixes $\{1, \ldots, k\}$ pointwise. If $k = n - 2$ we are done, so we may assume $k \leqs n - 3$. 

Suppose $a_{k + 1} = 2$. Here $k + 1 \in \{n-3,n-2\}$ and we define 
\[
X_{k + 1} = \{\{1, n -3\}, \{2, n-2\}, \{3, n-1\}, \{4, n\}, \{5,6\}, \ldots, \{n - 5, n - 4\}\}.
\]
Since $n\geqs 8$, we note that $\bigcap_{i = 1}^{k}H_i$ fixes $1$, $2$, $3$ and $4$, whence $\bigcap_{i = 1}^{k + 1}H_i = 1$. 

Now assume $a_{k + 1} \geqs 3$. Since $k$ is fixed by $\bigcap_{i = 1}^k H_i$, there must exist some $j\in \{1, \ldots, k\}$ such that $k$ and $k + 1$ lie in different parts of $X_j$ (if not, then $\bigcap_{i = 1}^k H_i$ would contain the transposition $(k,k+1)$). Then Lemma \ref{lem:2sets}(ii) implies that we can construct an $a_{k+1}$-partition $X_{k + 1}$ such that $H_j \cap H_{k + 1}$ fixes $\{k, k + 1\}$. If $k = n-3$, then the argument is complete, so we may assume that $k \leqs n-4$ for the remainder of the proof.

First assume $a_{n - 3} \geqs 3$. Here we use Lemma \ref{lem:2sets}(i) to  construct appropriate partitions $X_{k + 2}, \ldots, X_{n - 2}$ so that $H_{i - 1} \cap H_{i}$ fixes $\{i - 1, i\}$ for all $i\in \{k + 2, \ldots, n - 2\}$. It follows that $\bigcap_{i = 1}^{n - 2}H_i$ fixes the $2$-sets $\{k, k + 1\}, \ldots, \{n - 3, n - 2\}$ 
and thus $\bigcap_{i = 1}^{n - 2}H_i \leqs \<(n-1, n)\>$, as claimed.

Finally, suppose $a_{n - 3} = 2$. Since $a_{n-4} \geqs 3$, we can construct partitions $X_{k + 2}, \ldots, X_{n - 3}$ via Lemma \ref{lem:2sets}(i) so that $H_{i - 1}\cap H_i$ fixes $\{i - 1, i\}$ for all $i\in \{k + 2, \ldots, n-3\}$. Then $\bigcap_{i = 1}^{n - 3} H_i$ fixes $\{1, \ldots, n - 3\}$ pointwise and so by setting 
\[
X_{n-2} = \{\{1, n - 2\}, \{2, n-1\}, \{3, n\}, \{4, 5\}, \ldots, \{n - 4, n - 3\}\}
\]
we conclude that $\bigcap_{i = 1}^{n-2}H_i = 1$.

This completes the proof of Claim $(\star)$ and the result follows. 
\end{proof}

\subsubsection{Proof of Theorem \ref{t:main1}}\label{ss:final}

We are now ready to complete the proof of Theorem \ref{t:main1} by showing that 
$R(S_n) = n-1$ and $R(A_n) = n-2$ as in part (i). To do this, we will bring together our earlier work in Sections \ref{ss:prim}, \ref{ss:intrans} and \ref{ss:imprim}; the key ingredient is the following theorem.

\begin{thm}\label{thm:transpositions}
Let $G = S_n$ or $A_n$ with $n \geqs 13$, and let $\tau = (H_1, \ldots, H_{n-2})$ be a core-free tuple of subgroups of $G$. Then there exist elements $g_i \in G$ and $\a,\b \in [n]$ such that 
\[
\bigcap_{i = 1}^{n - 2}H_i^{g_i} \leqs \<(\alpha, \beta)\>.
\]
\end{thm}

\begin{proof}
In view of Lemma \ref{lem:maxreduction}, we may assume that $H_i \in \mathcal{M}(G)$ for all $i$.

Suppose $H_1$ and $H_2$ both act primitively on $\{1, \ldots, n\}$. Since $n \geqs 13$, Theorem \ref{t:main1}(ii) implies that $R_{{\rm prim}}(S_n) = 2$ and thus 
$H_1 \cap H_2^{g_2} = 1$ for some $g_2 \in G$. Therefore, for the remainder, we may assume that there is at most one primitive subgroup in $\tau$. By reordering the subgroups, if necessary, we may also assume that if $H_i$ is intransitive, $H_j$ is imprimitive and $H_k$ is primitive, then $i < j$ and $k = n-2$. 

Set $\s = (a,b,c)$, where $a$, $b$ and $c$ denote the number of intransitive, imprimitive and primitive subgroups in $\tau$, respectively. Then the possibilities for $\s$ are as follows, where $k \geqs 1$:
\[
(n-2,0,0), \, (0,n-2,0), \, (n-3,0,1), \, (0,n-3,1), \, (k,n-k-2,0), \, (k,n-k-3,1).
\]
We now consider each possibility in turn.

\vs

\noindent \emph{Case 1. $\s = (n-2,0,0)$ or $(0,n-2,0)$} 

\vs

First assume $\s = (n-2,0,0)$. Here $\tau$ is an intransitive tuple and this case was handled in the proof of Proposition \ref{p:rintrans} (see \eqref{e:intrans}). Similarly, $\tau$ is an imprimitive tuple if $\s = (0,n-2,0)$ and we can appeal to Claim $(\star)$ in the proof of Proposition \ref{prop:rimprim}.

\vs

\noindent \emph{Case 2. $\s = (n-3,0,1)$ or $(0,n-3,1)$} 

\vs

Next assume $\s = (n-3,0,1)$, so $H_{n-2}$ is primitive and we have $H_i = (S_{a_i} \times S_{n-a_i}) \cap G$ for $1 \leqs i \leqs n-3$, where $1 \leqs a_i \leqs n/2$ and we may assume that $a_1\leqs a_2\leqs \cdots \leqs a_{n - 3}$. For $i \in \{1, \ldots, n-3\}$, we define the specific $a_i$-sets $X_i$ in \eqref{e:xi} (see the proof of Lemma \ref{lem:orbits}) and we write $L_i$ for the stabiliser of $X_i$ in $G$. Then $L_i = H_i^{g_i}$ for some $g_i \in G$ and Lemma \ref{lem:orbits} implies that $\bigcap_{i = 1}^{n - 3} L_i$ fixes $\{1, \ldots, n-4\}$ pointwise. If $n \geqs 40$, then the lower bound on $\mu(H_{n-2})$ in Theorem \ref{thm:maroti} implies that ${\rm supp}(x) \geqs 8$ for every nontrivial element $x \in H_{n-2}$, and with the aid of {\sc Magma} it is easy to check that the same conclusion holds for all $n$ in the range $13 \leqs n < 40$. This immediately implies that 
\[
\left(\bigcap_{i = 1}^{n - 3}L_i\right)\cap H_{n - 2} = 1
\]
and the result follows.

A very similar argument applies when $\s = (0,n-3,1)$. Here $H_i = (S_{a_i}\wr S_{b_i}) \cap G$ for $i \in \{1, \ldots, n-3\}$, where $n = a_ib_i$ and $a_i,b_i \geqs 2$. We may assume that $a_1\geqs a_2\geqs \cdots \geqs a_{n - 3}$. By arguing as in the proof of Proposition \ref{prop:rimprim}, we can construct an $a_i$-partition $X_i$ for $i \in \{1, \ldots, n-3\}$ with stabiliser $L_i$ in $G$ such that $\bigcap_{i = 1}^{n - 3} L_i$ fixes $\{1, \ldots, n-3\}$ pointwise. We now complete the argument as above, appealing to the fact that ${\rm supp}(x) \geqs 8$ for all $1 \ne x \in H_{n-2}$.

\vs

\noindent \emph{Case 3. $\s = (k,n-k-2,0)$}

\vs

Here $k \in \{1, \ldots, n-3\}$ and recall that we arrange the component subgroups of $\tau$ so that for $i \leqs k$ we have $H_i = (S_{a_i}\times S_{n - a_i}) \cap G$ with $a_i \leqs n/2$, and $H_i = (S_{b_i}\wr S_{c_i}) \cap G$ for $i>k$, where $n = b_ic_i$ and $b_i \geqs 2$. In addition, we may assume that $a_1\leqs a_2\leqs \cdots \leqs a_k$ and $b_{k + 1}\geqs b_{k + 2}\geqs \cdots \geqs b_{n - 2}$. 

We start by defining the sets $X_1, \ldots, X_k$ as in \eqref{e:xi} and we write $L_i$ for the stabiliser of $X_i$ in $G$. Then for each $i \in \{k+1, \ldots, n-2\}$ we need to construct a $b_i$-partition $Y_i$ with stabiliser $L_i$ such that $\bigcap_{i = 1}^{n - 2}L_i  \leqs \<(\alpha, \beta)\>$ for some $\a,\b \in [n]$. To do this, we will handle the special case $k \in \{n-4,n-3\}$ separately.

\vs 

\noindent \emph{Case 3(a). $k \in \{n - 4,n-3\}$}

\vs

First assume $k = n-3$ and recall that $m = \max\{\a \, :\, \a \in X_k\}$, so $m \in \{n-3,n-2,n-1\}$. Set $L = \bigcap_{i=1}^{n-3}L_i$. Then by applying Lemma \ref{lem:orbits} we deduce that either $L \leqs {\rm Sym}(X)$ with $X = \{n-3,n-2,n-1\}$ or $\{n-2,n-1,n\}$, or $L \leqs {\rm Sym}(X) \times {\rm Sym}(Y)$ with $X = \{n - 3, n-2\}$ and $Y = \{n-1,n\}$.

Suppose $L \leqs {\rm Sym}(X)$ with $X = \{n-3,n-2,n-1\}$ and let $X_{n-2}$ be a $b_{n-2}$-partition of $[n]$ such that the part $P$ containing $n - 1$ does not contain $n-2$ nor $n-3$. Let $L_{n-2}$ be the stabiliser of $X_{n-2}$ in $G$. Then $L$ fixes every point in $P$, with the possible exception of $n-1$, so 
$\bigcap_{i = 1}^{n - 2}L_i$ fixes $\{1, \ldots, n-4,n-1,n\}$ pointwise and so it is contained in $\<(n - 3, n-2)\>$. An entirely similar argument applies if $L \leqs {\rm Sym}(X)$ with $X = \{n-2,n-1,n\}$, so let us assume $L \leqs {\rm Sym}(X) \times {\rm Sym}(Y)$ with $X = \{n - 3, n-2\}$ and $Y = \{n-1,n\}$. Here we choose a $b_{n-2}$-partition $X_{n - 2}$ so that the part containing $n$ does not intersect $\{n-3,n-2,n-1\}$. If $L_{n-2}$ denotes the stabiliser of $X_{n-2}$, then by arguing as above, we deduce that $\bigcap_{i = 1}^{n - 2}L_i$ fixes $n$, so it must also fix $n-1$ and it is therefore contained in $\<(n - 3, n-2)\>$.

A very similar argument applies when $k = n - 4$. Set $L = \bigcap_{i=1}^{n-4}L_i$, so Lemma \ref{lem:orbits} implies that either $L \leqs {\rm Sym}(X)$ with $X = \{n - 4, n-3, n-2, n-1\}$ or $\{n-3, n-2, n-1, n\}$, or $L \leqs {\rm Sym}(X) \times {\rm Sym}(Y)$ with $X = \{n - 4, n-3, n-2\}$ and $Y = \{n-1,n\}$, or $X = \{n - 4, n-3\}$ and $Y = \{n - 2, n - 1, n\}$. 

Suppose $L \leqs {\rm Sym}(X)$ with $X = \{n - 4, n-3, n-2, n-1\}$. Choose a $b_{n-3}$-partition $X_{n-3}$ with stabiliser $L_{n-3}$ such that the part containing $n-1$ does not intersect $\{n-4,n-3,n-2\}$. Then by arguing as above, we deduce that $\bigcap_{i = 1}^{n - 3}L_i$ fixes $n - 1$ and is therefore contained in ${\rm Sym}(Y)$ for $Y =\{n - 4, n -3, n-2\}$. We can now complete the argument as in the case $k=n-3$. Similarly, if $L \leqs {\rm Sym}(X) \times {\rm Sym}(Y)$ with $X = \{n - 4, n-3, n-2\}$ and $Y = \{n-1,n\}$, then we define a $b_{n-3}$-partition $X_{n-3}$ so that the part containing $n$ does not meet $\{n - 4, n-3, n-2, n-1\}$. Then $\bigcap_{i = 1}^{n - 3}L_i$ fixes $n$ and is therefore contained in ${\rm Sym}(Z)$ for $Z = \{n - 4, n-3, n-2\}$. By choosing a $b_{n-2}$-partition $X_{n-2}$ so that the part containing $n-2$ does not contain $n-4$ nor $n-3$, we conclude that $\bigcap_{i = 1}^{n - 2}L_i \leqs \<(n - 4, n-3)\>$, as required.

The other two possibilities (with $k = n-4$) are entirely similar and we omit the details. 

\vs

\noindent \emph{Case 3(b). $k \leqs n - 5$}

\vs

Now assume $k \leqs n - 5$. Set $L = \bigcap_{i = 1}^{k}L_i$ and note that $L$ fixes $\{1, \ldots, k-1\}$ pointwise by Lemma \ref{lem:orbits}. In addition, the same lemma implies that $k$ and $n$ are contained in distinct $L$-orbits. Therefore, it suffices to construct appropriate partitions $X_{k + 1}, \ldots, X_{n - 2}$ of $[n]$ with stabilisers $L_{k+1}, \ldots, L_{n-2}$ such that $\bigcap_{i = k+1}^{n - 2}L_i$ fixes $\{k + 1, \ldots, n - 2\}$ pointwise. Indeed, this will imply that $\bigcap_{i = 1}^{n - 2}L_i \leqs \< g\>$, where $g = (k,n-1)$ or $(n-1,n)$.

Suppose $b_{n-4}=2$, in which case $H_{i} = (S_2 \wr S_{n/2}) \cap G$ for all $i \in \{n-4,n-3,n-2\}$. Here Proposition \ref{p:msbase} implies that there exists $g_i \in G$ such that $\bigcap_{i=n-4}^{n-2}H_i^{g_i} = 1$. Therefore, for the remainder of Case 3(b) we may assume that $b_{n-4} \geqs 3$.

Let us first consider the special case where $b_i \in \{2, n/2\}$ for all $i$, in which case $n \geqs 14$ is even. Here we have $b_{k+1} = b_{n-4} = n/2$ and we let $j \in \{n-4,n-3,n-2\}$ be maximal such that $b_j = n/2$. We now consider several separate cases.

Suppose $j \leqs k + 2$. Since $j \geqs n-4$, this means that $k \in\{ n-6, n-5\}$ and thus $b_{n - 2} = 2$. Since $n/2 \geqs 7$, we can define a $2$-partition $X_{n-2}$ with the property that all of the points $k, \ldots, n$ lie in different parts. For $\ell \in \{k, \ldots, n\}$, write $P_\ell$ for the part of $X_{n-2}$ containing $\ell$. Since $P_\ell$ contains at least one point fixed by $L$, it follows that $L \cap L_{n - 2}$ fixes $P_\ell$, where $L_{n-2}$ is the stabiliser of $X_{n-2}$. In turn, this implies that $L \cap L_{n - 2}$ fixes every $\ell \in \{k, \ldots, n\}$ and thus $L \cap L_{n-2} = 1$. 

Now assume $j \geqs k + 3$. Here Lemma \ref{lem:pointwise} implies that we can construct a $b_i$-partition $X_i$ for each $i \in \{k+1, \ldots, j\}$ such that $\bigcap_{i = k +1}^{j} L_i$ fixes $\{k+1, \ldots, j\}$ pointwise. If $j = n-2$, then we are done, so let us assume $j =n-4$ or $n-3$. Note that $b_{n-2} = 2$, so $n \geqs 14$ is even.

Since $j + 1\geqs n - 3$, we deduce that $\bigcap_{i = 1}^{j}L_i \leqs {\rm Sym}(\{k, n-3, \ldots, n\})$. Since $n/2 \geqs 7$, we can construct a $2$-partition $X_{j + 1}$ such that each point $\ell \in \{k, n- 3, \ldots, n\}$ is contained in a distinct part $P_{\ell}$. If $L_{j+1}$ denotes the stabiliser of $X_{j+1}$, then it is easy to see that $\bigcap_{i=1}^{j+1}L_i = 1$ and the result follows.

To complete the argument in Case 3(b), we may assume $b_r \not\in \{2,n/2\}$ for some $r$. If $b_{n - 2} \geqs 3$, then we can proceed as in the proof of Proposition \ref{prop:rimprim} to produce a $b_i$-partition $X_i$ for each $i \in \{k+1, \ldots, n-2\}$ such that $\bigcap_{i = k + 1}^{n - 2} L_i$ fixes $\{k + 1, \ldots, n - 2\}$ pointwise, where $L_i$ is the stabiliser of $X_i$. Finally, suppose $b_{n-2} = 2$ and let $j \in \{n-3, n-2\}$ be minimal such that $b_j = 2$. As in the proof of Proposition \ref{prop:rimprim}, we can construct a $b_i$-partition $X_i$ for all $i \in \{k+1,\ldots, j\}$ such that $\bigcap_{i = k + 1}^j L_i$ fixes $\{k + 1, \ldots, j\}$ pointwise, whence the desired result follows if $j = n - 2$. On the other hand, if $j = n-3$ then $\bigcap_{i = 1}^{j}L_i \leqs {\rm Sym}(\{k, n-2, n - 1, n\})$ and we obtain $\bigcap_{i = 1}^{n-2}L_i = 1$ by taking $L_{n-2}$ to be the stabiliser of a $2$-partition of $[n]$ such that the points $k$, $n-2$, $n-1$ and $n$ all lie in distinct parts.

\vs

\noindent \emph{Case 4. $\s = (k,n-k-3,1)$}

\vs

In order to complete the proof of the theorem, we may assume $\s = (k,n-k-3,1)$ for some $k \in \{1, \ldots, n-4\}$. As before, we order the subgroups in $\tau$ so that 
$H_i = (S_{a_i}\times S_{n - a_i}) \cap G$ for $i \leqs k$ and $H_i = (S_{b_i}\wr S_{c_i}) \cap G$ for $i \in \{k+1, \ldots, n-3\}$, where $a_i \leqs n/2$, $n = b_ic_i$ and $b_i \geqs 2$. In addition, $H_{n-2}$ is primitive and we assume that $a_1\leqs a_2\leqs \cdots \leqs a_k$ and $b_{k + 1}\geqs b_{k + 2}\geqs \cdots \geqs b_{n - 3}$. As noted in Case 2, we have ${\rm supp}(x) \geqs 8$ for all $1 \ne x\in H_{n - 2}$.
We define the sets $X_1, \ldots, X_k$ as in the statement of Lemma \ref{lem:orbits}, with stabiliser $L_i$, and we set $L = \bigcap_{i = 1}^{k} L_i$. By the lemma, we note that $L$ fixes $\{1, \ldots, k-1\}$ pointwise.

If $k \geqs n-6$ then Lemma \ref{lem:orbits} implies that every element in $L = \bigcap_{i = 1}^{k} L_i$ moves at most $7$ points in $[n]$, whence $L \cap H_{n-2} = 1$ and the result follows. Now assume $k \leqs n-7$. Here we can proceed as in Case 2 to produce a $b_i$-partition $X_i$ with stabiliser $L_i$ for $i \in \{k+1, \ldots, n-3\}$ such that $\bigcap_{i=k+1}^{n-3}L_i$ fixes $\{k+1, \ldots, n-3\}$ pointwise. Therefore, every element in $\bigcap_{i=1}^{n-3}L_i$ moves at most $4$ points and we conclude that
\[
\left(\bigcap_{i = 1}^{n - 3}L_i\right) \cap H_{n - 2} = 1
\]
as required.
\end{proof}

\begin{rem}
One can check that the conclusion to Theorem \ref{thm:transpositions} also holds for $G = S_n$ or $A_n$ with $5 \leqs n \leqs 12$, with the exception of $G = S_6$. Indeed, if $G = S_6$ and $\tau = (H,H,H,H)$, where $H = S_5$ is primitive, then $\tau$ is non-regular and $H$ does not contain a transposition.
\end{rem}

Finally, we are now in a position to complete the proof of Theorem \ref{t:main1}. 

\begin{proof}[Proof of Theorem \ref{t:main1}]
It just remains to prove that $R(S_n) \leqs n-1$ and $R(A_n) \leqs n-2$. The groups with $n \leqs 12$ can be handled using {\sc Magma} (see \cite{AB_comp} for the details), so we may assume $n \geqs 13$. For $G = A_n$, the result now follows immediately from Theorem \ref{thm:transpositions}, so for the remainder we will assume $G = S_n$. Let $\tau = (H_1, \ldots, H_{n-1})$ be a core-free tuple of subgroups of $G$. By  Theorem \ref{thm:transpositions}, there exists $g_i \in G$ such that $\bigcap_{i = 1}^{n - 2}H_i^{g_i} \leqs \<(\alpha, \beta)\>$ for some $\alpha, \beta \in [n]$, so we just need to show that 
\[
H_{n-1}^{g_{n-1}} \cap \<(\alpha, \beta)\> = 1
\]
for some $g_{n-1} \in G$. This is clear if $H_{n-1}$ is primitive (no core-free primitive subgroup of $G$ contains a transposition). And if $H_{n-1} \leqs S_k \times S_{n-k}$ is intransitive, then the result follows because we can always find a $k$-element subset of $[n]$ that contains $\a$, but not $\b$. Similarly, if $H_{n - 1} \leqs S_a \wr S_b$ is imprimitive with $a,b \geqs 2$, then the result follows because we can find an $a$-partition of $[n]$ with the property that $\a$ and $\b$ are contained in distinct parts. This completes the proof of the theorem.
\end{proof}

With the proof of Theorem \ref{t:main1} in hand, we conclude this section by briefly considering the problem of determining all the non-regular $(n-2)$-tuples for $S_n$, and the analogous problem for $A_n$ with respect to $(n-3)$-tuples. Clearly, if $G = S_n$ and $H = S_{n-1}$ is the stabiliser of a point in $\{1, \ldots, n\}$, then the maximal intransitive $(n-2)$-tuple $(H, \ldots, H)$ is non-regular. In the following result, we determine all the maximal intransitive $(n-2)$-tuples for $S_n$.

\begin{prop}\label{p:snmax}
Let $G = S_n$ with $n \geqs 5$ and let $\tau$ be a maximal intransitive $(n-2)$-tuple. Then up to conjugacy and reordering, $\tau$ is non-regular if and only if 
\[
\tau = (H, \ldots, H, K),
\]
where $H = S_{n-1}$ and $K = S_k \times S_{n-k}$ for some $1 \leqs k \leqs n/2$.
\end{prop}

\begin{proof}
First assume $\tau = (H_1, \ldots, H_{n-2})$ has the given form, so $H_{n-2} = S_k \times S_{n-k}$ and $H_i = S_{n-1}$ for all $1 \leqs i \leqs n-3$. Then for any $g_i \in G$ we observe that 
\[
\bigcap_{i=1}^{n-3} H_i^{g_i} = {\rm Sym}(X),
\]
where $X$ is a subset of $[n]=\{1,\ldots, n\}$ with $|X| \geqs 3$. And so if $Y$ is any subset of $[n]$ of size $k$, then either $|X \cap Y| \geqs 2$ or $|X \cap ([n] \setminus Y)| \geqs 2$ and thus 
$\bigcap_{i=1}^{n-2} H_i^{g_i}$
always contains a transposition. In particular, $\tau$ is non-regular.

To complete the proof, suppose $\s = (L_1, \ldots, L_{n-2})$ is a maximal intransitive $(n-2)$-tuple, where $L_i = S_{a_i} \times S_{n-a_i}$ and 
\[
1 \leqs a_1 \leqs a_2 \leqs \cdots \leqs a_{n-2} \leqs n/2
\]
with $a_{n-3} \geqs 2$. We need to show that $\sigma$ is regular. To do this, define the $a_i$-sets $X_1, \ldots, X_{n-4}$ as in the statement of Lemma \ref{lem:orbits}, with respective stabilisers $K_i$. Set $H = \bigcap_{i=1}^{n-4}K_i$. Then Lemma \ref{lem:orbits} implies that either 

\vspace{1mm}

\begin{itemize}\addtolength{\itemsep}{0.3\baselineskip}
\item [{\rm (a)}] $H \leqs {\rm Sym}(X)$ with $X=\{n-4, n-3, n-2, n-1\}$ or $\{n-3, n-2, n-1, n\}$;
\item[{\rm (b)}] $H \leqs {\rm Sym}(X) \times {\rm Sym}(Y)$ with $X=\{n-4, n-3, n-2\}$ and $Y=\{n-1, n\}$, or $X = \{n-4, n-3\}$ and $Y=\{n-2, n-1, n\}$. 
\end{itemize}
We now define the sets
\[
X_{n-3} = \{n-3, n-2, 1, 2, \ldots, a_{n-3}-2 \},\;\; X_{n-2} = \{n-2, n-1, 1, 2, \ldots, a_{n-2}-2\}
\]
with stabilisers $K_{n-3}$ and $K_{n-2}$ respectively (so if $a_{n-3} = 2$, then $X_{n-3} = \{n-3,n-2\}$). Then $\bigcap_{i=1}^{n-2} K_i = 1$ and we are done.
\end{proof}

\begin{rem}
For $G = A_n$, an entirely similar argument shows that the maximal intransitive $(n-3)$-tuples for $G$ are precisely those of the form $(H, \ldots, H,K)$, up to conjugacy and reordering, where $H = A_{n-1}$ and $K = (S_k \times S_{n-k}) \cap G$ for some $1 \leqs k \leqs n/2$. 
\end{rem}

\begin{rem}\label{r:snmax}
It would be interesting to extend Proposition \ref{p:snmax} by giving a complete  classification of all the non-regular $(n-2)$-tuples for $G = S_n$. Notice that if $n=2m$ is even, $H = S_{n-1}$ and $K = S_m \wr S_2$, then it is easy to see that $(H, \ldots, H,K)$ is a non-regular $(n-2)$-tuple with a transitive component subgroup. And by replacing $K$ by the index-two subgroup $L = S_m \times S_m$, we obtain a non-regular $(n-2)$-tuple containing a non-maximal subgroup. For $n \geqs 5$ odd, we are not aware of any non-regular $(n-2)$-tuples for $G$ that are not of the form described in Proposition \ref{p:snmax}.
\end{rem}

\section{Sporadic groups}\label{s:spor}

For the remainder of the paper, we assume $G$ is an almost simple sporadic group with socle $T$ and our aim is to prove Theorem \ref{t:main2}. We begin in Section \ref{ss:spor_reg} by computing the exact regularity number of $G$, which establishes part (i) of Theorem \ref{t:main2}. We then prove part (ii) in Section \ref{ss:spor_sol}, where we show that every soluble triple of subgroups of $G$ is regular. We apply computational methods throughout this section, working extensively with {\sc Magma} \cite{magma} and \textsf{GAP} \cite{GAP}. We refer the reader to Section \ref{ss:comp} for a brief overview of some of the main methods, which complements the more detailed discussion in the supplementary file \cite{AB_comp}.

As explained in Wilson's survey article \cite{Wilson_survey}, the maximal subgroups of $G$ are known up to conjugacy, apart from the problem of determining the status of a handful of candidate almost simple maximal subgroups of the Monster $\mathbb{M}$ (and the latter ambiguity appears to have been recently resolved by Dietrich et al. in \cite{DLP}). In addition, the character table of $G$ is available in the \textsf{GAP} Character Table Library \cite{GAPCTL}, together with the character tables of all the maximal subgroups of $G$ and the associated fusion maps on conjugacy classes (excluding $G = \mathbb{M}$). A great deal of useful information about the sporadic groups is available in the Web Atlas \cite{WebAt}. 

\subsection{The regularity number}\label{ss:spor_reg}

Let $G$ be an almost simple sporadic group with socle $T$, so $G = T$ or $T.2$. Recall that we define $\mathcal{M}(G)$ to be the set of core-free subgroups of $G$ that are maximal in either $G$ or $T$ (so if $G = T$, then $\mathcal{M}(G)$ is just the set of maximal subgroups of $G$). As noted in Remark \ref{r:large}, if $G = T.2$ and $H<T$ is maximal with $N_G(H) \ne H$, then $H$ is contained in a core-free maximal subgroup of $G$. Therefore, if we define $\mathcal{M}'(G) = \mathcal{A} \cup \mathcal{B}$, where
\begin{align*}
\mathcal{A} & = \{\mbox{core-free maximal subgroups of $G$}\} \\
\mathcal{B} & = \{\mbox{maximal subgroups $H$ of $T$ with $N_G(H) = H$}\}
\end{align*}
then 
\[
B(G) = \max\{b(G,H) \,:\, H \in \mathcal{M}'(G)\}
\]
and $R(G)$ is the smallest positive integer $k$ such that every $k$-tuple of subgroups in $\mathcal{M}'(G)$ is regular. For the remainder of this section, it will be convenient to work with the following definition (note that if $G = T$, then every \emph{large} tuple is \emph{maximal}, and vice versa).

\begin{defn}\label{d:large}
A $k$-tuple of subgroups $(H_1, \ldots, H_k)$ of $G$ is \emph{large} if each component subgroup $H_i$ is contained in the set $\mathcal{M}'(G)$ defined above.
\end{defn}

Recall that $B(G) \geqs 3$ by Proposition \ref{p:low}. We begin by recording the exact base number, which is easily obtained from \cite{BOW}.

\begin{prop}\label{p:spor_bG}
Let $G$ be an almost simple sporadic group with socle $T$. Then the base number $B(G)$ is recorded in Table \ref{tab:RG_sporadic}.
\end{prop}

\begin{proof}
If $G = T$, then we can immediately read off $B(G)$ from the main theorem of \cite{BOW}, so let us assume $G = T.2$ and $H \in \mathcal{M}'(G)$. If $H$ is maximal in $G$, then $b(G,H)$ is computed in \cite{BOW}, so we may assume that $H<T$ is maximal and $N_G(H) = H$. Here $b(T,H)$ is recorded in \cite{BOW}, which is clearly an upper bound for $b(G,H)$. 

By inspecting \cite{BOW}, we see that $G = {\rm M}_{12}.2$ is the only group that arises where $B(T)$ is greater than the maximal value of $b(G,H)$ over all core-free maximal subgroups $H$ of $G$. Here it is easy to check that if $H = {\rm M}_{11} < T$, then $b(G,H) = 5$ and we deduce that $B(G) = 5$.
\end{proof}

Now let us turn to the regularity number of $G$. Our main result is the following, which establishes part (i) of Theorem \ref{t:main2}. See Remark \ref{r:spor_tab} below for comments on the information recorded in the final column of Table 
\ref{tab:RG_sporadic}. In (ii), we write $\mathbb{B}$ and $\mathbb{M}$ for the Baby Monster and Monster groups, respectively.

\begin{prop}\label{p:spor_RG}
Let $G$ be an almost simple sporadic group with socle $T$. 

\vspace{1mm}

\begin{itemize}\addtolength{\itemsep}{0.3\baselineskip}
\item [{\rm (i)}] The regularity number $R(G) = k$ is recorded in Table \ref{tab:RG_sporadic}.
\item [{\rm (ii)}] For $G \ne \mathbb{B},\mathbb{M}$, a complete list of the large non-regular $(k-1)$-tuples, up to ordering and conjugacy, is presented in the final column of Table \ref{tab:RG_sporadic}.
\end{itemize}
\end{prop}

As an immediate corollary, we obtain a strong form of Conjecture \ref{con:main}(ii) for almost simple sporadic groups.

\begin{cor}\label{c:spor_RG}
We have $R(G) \leqs 7$, with equality if and only if $G ={\rm M}_{24}$. Moreover, if $G = {\rm M}_{24}$ then a $6$-tuple is non-regular if and only if every component subgroup is conjugate to ${\rm M}_{23}$. 
\end{cor}

\begin{rem}\label{r:spor_tab}
Suppose $G \ne \mathbb{B},\mathbb{M}$ and $R(G) = k$. In the final column of Table \ref{tab:RG_sporadic} we list all the large non-regular $(k-1)$-tuples for $G$, up to conjugacy and ordering. To do this, we use the notation $(a_1, \ldots, a_{k-1}) \in \mathbb{N}^{k-1}$ to encode each tuple $\tau = (H_{1}, \ldots, H_{k-1})$, noting that the corresponding subgroups are recorded in Tables \ref{tab:sporadic_key} and \ref{tab:sporadic_key2}. So for example, if $G = {\rm J}_3.2$, then $(H_1,H_2) = ( {\rm L}_2(17) \times 2, (3 \times {\rm M}_{10}){:}2 )$ is the non-regular pair labelled $(3,4)$ in Table \ref{tab:RG_sporadic}. We refer the reader to Remark \ref{r:RG_spor} for comments on the  groups $\mathbb{B}$ and $\mathbb{M}$ we are excluding here.
\end{rem}

{\scriptsize
\begin{table}
\[
\begin{array}{lccl} \hline
G & B(G) & R(G) & \\ \hline
{\rm M}_{11} & 4 & 5 & (1,1,2,2), (1,2,2,2) \\
{\rm M}_{12} & 5 & 6 & (1,1,1,2,2), (1,1,2,2,2) \\
{\rm M}_{12}.2 & 5 & 5 & (1,1,1,1), (1,1,1,2) \\
{\rm M}_{22} & 5 & 5 & (1,1,1,1), (1,1,1,2) \\
{\rm M}_{22}.2 & 5 & 5 & ( 1, 1, 1, i ), ( 1, 1, 2, 2 ), ( 1, 1, 2, 4 ), \, 1 \leqs i \leqs 5 \\
{\rm M}_{23} & 6 & 6 & (1,1,1,1,1) \\
{\rm M}_{24} & 7 & 7 & (1,1,1,1,1,1) \\
{\rm J}_1 & 3 & 3 & (1,1), (1,2) \\
{\rm J}_2 & 4 & 4 & (1,1,1) \\
{\rm J}_2.2 & 4 & 4 & (1,1,i), \, 1 \leqs i \leqs 4 \\
{\rm J}_3 & 3 & 3 & ( 1, i ), \, 1 \leqs i \leqs 8 \\
{\rm J}_3.2 & 3 & 3 & ( 1, i ), ( 3, 4 ), ( 3, 6 ), ( 4, 4 ), \, 1 \leqs i \leqs 7 \\
{\rm HS} & 4 & 4 & ( 1, 1, 1, i ), ( 1, 1, 2, 2 ), ( 1, 1, 3, 3 ), \, 1 \leqs i \leqs 3 \\
{\rm HS}.2 & 5 & 5 & ( 1, 1, 1, i ), \, 1 \leqs i \leqs 4 \\
{\rm McL} & 5 & 5 & (1,1,1,1) \\
{\rm McL}.2 & 5 & 5 & (1,1,1,1) \\
{\rm He} & 4 & 4 & (1,1,1), (1,1,2) \\
{\rm He}.2 & 4 & 4 & (1,1,1), (1,1,2) \\
{\rm Suz} & 4 & 4 & ( 1, 1, i ), ( 1, 2, 2 ), ( 1, 2, 3 ), ( 1, 3, 3 ), \, 1 \leqs i \leqs 7 \\ 
{\rm Suz}.2 & 4 & 4 & ( 1, 1, i ),( 1, 2, 2 ),( 1, 2, 3 ), ( 1, 3, 3 ),\, 1 \leqs i \leqs 9 \\
{\rm O'N} & 3 & 3 & (1,i_1), (2,i_2), \, j \leqs i_j \leqs 4 \\
{\rm O'N}.2 & 3 & 3 & (1,i), (2,3), (3,3),\, 1 \leqs i \leqs 3 \\
{\rm HN} & 3 & 3 & (1,i_1), (2,i_2), (3,3),\,  j \leqs i_j \leqs 12  \\
{\rm HN}.2 & 3 & 3 & (1,i), (2,j), (2,12), (3,3), (3,4), \, 1 \leqs i \leqs 12, \, 2 \leqs j \leqs 10 \\
{\rm Co}_3 & 6 & 6 & (1,1,1,1,1) \\
{\rm Co}_2 & 6 & 6 & (1,1,1,1,1) \\
{\rm Co}_1 & 5 & 5 & (1,1,1,1) \\
{\rm Ru} & 4 & 4 & (1,1,1) \\
{\rm Fi}_{22} & 5 & 5 & (1,1,1,1), (1,1,1,2) \\
{\rm Fi}_{22}.2 & 6 & 6 & (1,1,1,1,1) \\
{\rm Fi}_{23} & 5 & 5 & ( 1, 1, 1, 1 ),( 1, 1, 1, 2 ),( 1, 1, 2, 2 ),( 1, 2, 2, 2 ) \\
{\rm Fi}_{24}' & 5 & 5 & (1,1,1,1), (1,1,1,2) \\
{\rm Fi}_{24} & 5 & 5 & (1,1,1,i), 1 \leqs i \leqs 3 \\
{\rm J}_4 & 3 & 3 & (1,i), (2,j), (3,3), \, 1\leqs i \leqs 5, \, 2 \leqs j \leqs 4  \\
{\rm Ly} & 3 & 3 &  (1,i_1), (2,i_2), \, j \leqs i_j \leqs 6 \\
{\rm Th} & 3 & 3 & (1,i), (2,2), \, 1 \leqs i \leqs 4  \\
\mathbb{B} & 4 & 4 & \mbox{See Remark \ref{r:RG_spor}}\\ 
\mathbb{M} & 3 & 3 & \mbox{See Remark \ref{r:RG_spor}}\\ \hline
\end{array}
\]
\caption{The base and regularity numbers for sporadic groups}
\label{tab:RG_sporadic}
\end{table}}

{\scriptsize
\begin{table}
\[
\begin{array}{l|llllll}
 & 1 & 2 & 3 & 4 & 5 & 6 \\ \hline
{\rm M}_{11} & {\rm M}_{10} & {\rm L}_{2}(11) & & & &  \\
{\rm M}_{12} & {\rm M}_{11} & {\rm M}_{11} & & & & \\
{\rm M}_{12}.2 & {\rm M}_{11} & A_6.2^2 && & & \\
{\rm M}_{22} & {\rm L}_3(4) & 2^4{:}A_6 & & & & \\
{\rm M}_{22}.2 & {\rm L}_3(4).2 & 2^4{:}S_6 & 2^5{:}S_5 & 2^3{:}{\rm L}_3(2) \times 2 & A_6.2^2 &  \\
{\rm M}_{23} & {\rm M}_{22} & & & & &  \\
{\rm M}_{24} & {\rm M}_{23} & & & & &  \\
{\rm J}_1 & {\rm L}_2(11) & 19{:}6 & & & & \\
{\rm J}_2 & {\rm U}_3(3) & & & & & \\
{\rm J}_2.2 & {\rm U}_3(3){:}2 & 3.A_6.2^2 & 2^{1+4}.S_5 & 2^{2+4}{:}(3 \times S_3).2 & & \\
{\rm J}_3 & {\rm L}_2(16){:}2 & {\rm L}_2(19) & {\rm L}_2(19) & 2^4{:}(3 \times A_5) & {\rm L}_2(17) & (3 \times A_6).2 \\
{\rm J}_3.2 & {\rm L}_2(16){:}4 & 2^4{:}(3 \times A_5).2 & {\rm L}_2(17) \times 2 & (3 \times {\rm M}_{10}){:}2 & 3^{2+1+2}{:}8.2 & 2^{1+4}{:}S_5 \\
{\rm HS} & {\rm M}_{22} & {\rm U}_3(5){:}2 & {\rm U}_3(5){:}2 & & &  \\
{\rm HS}.2 & {\rm M}_{22}.2 & {\rm L}_3(4){:}2^2 & S_8 \times 2 & {\rm U}_3(5){:}2 & & \\
{\rm McL} & {\rm U}_4(3) & & & & & \\
{\rm McL}.2 & {\rm U}_4(3){:}2 & & & & & \\ 
{\rm He} & {\rm Sp}_4(4){:}2 & 2^2.{\rm L}_3(4).S_3 & & & & \\
{\rm He}.2 & {\rm Sp}_4(4){:}4 & 2^2.{\rm L}_3(4).{\rm D}_{12} & & & & \\
{\rm Suz} & G_2(4) & 3.{\rm U}_4(3).2 & {\rm U}_5(2) & 2^{1+6}.{\rm U}_4(2) & 3^5{:}{\rm M}_{11} & {\rm J}_2{:}2 \\
{\rm Suz}.2 & G_2(4){:}2 & 3.{\rm U}_4(3).2^2 & {\rm U}_5(2){:}2 & 2^{1+6}.{\rm U}_4(2).2 & 3^5{:}({\rm M}_{11} \times 2) & {\rm J}_2{:}2 \times 2 \\
{\rm O'N} & {\rm L}_3(7){:}2 & {\rm L}_3(7){:}2 & {\rm J}_1 & 4.{\rm L}_3(4){:}2 & & \\
{\rm O'N}.2 & {\rm L}_3(7){:}2 & 7^{1+2}{:}(3 \times {\rm D}_{16}) & {\rm J}_1 \times 2 & & &\\
{\rm HN} & A_{12} & 2.{\rm HS}.2 & {\rm U}_3(8){:}3 & 2^{1+8}.(A_5 \times A_5).2 & ({\rm D}_{10} \times {\rm U}_3(5)).2 & 5^{1+4}.2^{1+4}.5.4 \\
{\rm HN}.2 & S_{12} & 4.{\rm HS}.2 & {\rm U}_3(8){:}6 & 2^{1+8}.(A_5 \times A_5).2^2 & 5{:}4 \times {\rm U}_3(5) & 5^{1+4}.2^{1+4}.5.4.2 \\
{\rm Co}_3 & {\rm McL}.2 & & & & & \\
{\rm Co}_2 & {\rm U}_6(2){:}2 & & & & & \\
{\rm Co}_1 & {\rm Co}_2 & & & & & \\
{\rm Ru} & {}^2F_4(2) & & & & & \\
{\rm Fi}_{22} & 2.{\rm U}_6(2) & 2^{10}.{\rm M}_{22} & & & & \\
{\rm Fi}_{22}.2 & 2.{\rm U}_6(2).2 & & & & & \\
{\rm Fi}_{23} & 2.{\rm Fi}_{22} & {\rm P\O}_{8}^{+}(3){:}S_3 & & & & \\
{\rm J}_4 & 2^{11}{:}{\rm M}_{24} & 2^{1+12}.3.{\rm M}_{22}{:}2 & 2^{10}{:}{\rm L}_5(2) &  2^{3+12}.(S_5 \times {\rm L}_3(2)) & {\rm U}_3(11){:}2 & \\
{\rm Ly} & G_2(5) & 3.{\rm McL}{:}2 & 5^3.{\rm L}_3(5) & 2.A_{11} & 5^{1+4}.4S_6 & 3^5{:}(2 \times {\rm M}_{11}) \\
{\rm Th} & {}^3D_4(2){:}3 & 2^5.{\rm L}_5(2) & 2^{1+8}.A_9 & {\rm U}_3(8){:}6 & & \\
{\rm Fi}_{24}' & {\rm Fi}_{23} & 2.{\rm Fi}_{22}.2 & & & \\
{\rm Fi}_{24} & {\rm Fi}_{23} \times 2 & (2 \times 2.{\rm Fi}_{22}){:}2 & S_3 \times {\rm P\O}_{8}^{+}(3){:}S_3 & &\\
\end{array}
\]
\caption{The subgroups recorded in Table \ref{tab:RG_sporadic}, Part I}
\label{tab:sporadic_key}
\end{table}}

{\scriptsize
\begin{table}
\[
\begin{array}{l|llllll}
 & 7 & 8 & 9 & 10 & 11 & 12 \\ \hline
{\rm M}_{12}.2 & S_4 \times S_3 & S_5 & & & & \\
{\rm J}_3 & 3^{2+1+2}{:}8 & 2^{1+4}{:}A_5 & & & & \\
{\rm J}_3.2 & {\rm L}_{2}(19) & & & & & \\ 
{\rm Suz} & 2^{2+8}{:}(A_5 \times S_3) & &  & & & \\
{\rm Suz}.2 & 2^{4+6}{:}3.S_6 & (A_4 \times {\rm L}_3(4){:}2){:}2 & 2^{2+8}{:}(S_5 \times S_3) & & & \\
{\rm HN} & 2^6.{\rm U}_4(2)  & (A_6 \times A_6).{\rm D}_8 & 2^{3+2+6}.(3 \times {\rm L}_3(2)) & 5^{2+1+2}.4.A_5 & {\rm M}_{12}.2 & {\rm M}_{12}.2 \\
{\rm HN}.2 & 2^6.{\rm U}_4(2).2  & (S_6 \times S_6).2^2 & 2^{3+2+6}.(3 \times {\rm L}_3(2)).2 & 5^{2+1+2}.4.A_5.2 & 3^4{:}2.(S_4 \times S_4).2 & {\rm M}_{12}.2 \\
\end{array}
\]
\caption{The subgroups recorded in Table \ref{tab:RG_sporadic}, Part II}
\label{tab:sporadic_key2}
\end{table}}

Let $G$ be an almost simple sporadic group with socle $T$. Our proof of Proposition \ref{p:spor_RG} is entirely computational, working with {\sc Magma} (version 2.28-4) \cite{magma} and \textsf{GAP} (version 4.11.1) \cite{GAP}. In particular, the information available in the \textsf{GAP} Character Table Library \cite{GAPCTL} is essential for handling some of the larger groups. We begin by dividing the $26$ possibilities for $T$ into the following four collections, and we will handle each one in turn:
\begin{align*}
\mathcal{A}_1 & = \{ {\rm M}_{11}, {\rm M}_{12}, {\rm M}_{22}, {\rm M}_{23}, {\rm M}_{24}, {\rm J}_{1}, {\rm J}_{2}, {\rm J}_{3}, {\rm HS}, {\rm McL}, {\rm He}, {\rm Suz}, {\rm O'N}, {\rm Ru}, {\rm Co}_{3}, {\rm Co}_{2},  {\rm Fi}_{22}, {\rm Fi}_{23} \} \\
\mathcal{A}_2 & = \{ {\rm Co}_1, {\rm HN}, {\rm Fi}_{24}' \} \\
\mathcal{A}_3 & = \{ {\rm J}_4, {\rm Ly}, {\rm Th}\} \\
\mathcal{A}_4 &  = \{\mathbb{B}, \mathbb{M}\}
\end{align*}

\begin{lem}\label{l:rG1}
The conclusion to Proposition \ref{p:spor_RG} holds for the groups with $T \in \mathcal{A}_1$.
\end{lem}

\begin{proof}
We begin by reading off $B(G) = b$ from Proposition \ref{p:spor_bG} and we fix a large $(b-1)$-tuple $\s = (H_1, \ldots, H_{b-1})$. As noted in Lemma \ref{l:triv}, if
\[
\prod_{i=1}^{b-1}|H_i| > |G|^{b-2},
\]
then $G$ does not have a regular orbit on $G/H_1 \times \cdots \times G/H_{b-1}$ and thus $\s$ is non-regular, so let us assume otherwise. 

We use the {\sc Magma} functions \texttt{AutomorphismGroupSimpleGroup} and \texttt{MaximalSubgroups} to construct $G$ and each $H_i$ as subgroups of $S_n$, where $n = m(G)$ is the minimal degree of a faithful primitive permutation representation of $G$. We can then randomly search for elements $g_i \in G$ such that $\bigcap_i H_i^{g_i} = 1$. If our search is successful after a specified number of attempts, then $\sigma$ is regular and we repeat for the next $(b-1)$-tuple in our list. In every case, this process quickly produces a distinguished collection of large $(b-1)$-tuples that require closer inspection in order to determine their regularity status (of course, there has to be at least one such tuple since $B(G) = b$).

Let $\s = (H_1, \ldots, H_{b-1})$ be a large tuple for which random search is inconclusive. Now $\s$ is regular if and only if $H_1$ has a regular orbit on $Y = G/H_2 \times \cdots \times G/H_{b-1}$, and we can use {\sc Magma} to determine whether or not such an orbit exists. To do this, we first reorder the subgroups in $\s$ so that $|H_1|$ is minimal. We then use the function \texttt{CosetAction} to construct $G$ as a permutation group on each set $\Gamma_i = G/H_i$ with $2 \leqs i \leqs b-1$ and we compute a set of orbit representatives $\{\a_1, \ldots, \a_m\}$ for the action of $H_1$ on $\Gamma_2$, with corresponding stabilisers $L_i = (H_1)_{\a_i}$. If $b=3$ then $\s$ is regular if and only if at least one of these stabilisers is trivial. And if $b =4$, then we can repeat the process; for each $i$, we find a set of orbit representatives for the action of $L_i$ on $\Gamma_3$ and we determine if any of the corresponding stabilisers in $L_i$ are trivial. And similarly for $b \geqs 5$. 
In this way, we obtain the complete list of large non-regular $(b-1)$-tuples, up to ordering and conjugacy. 

Next observe that a $b$-tuple $(H_1, \ldots, H_{b-1},H_b)$ is non-regular only if $(H_1, \ldots, H_{b-1})$ is non-regular, so we can work through our list of non-regular $(b-1)$-tuples, proceeding as above, to determine whether or not there are any non-regular $b$-tuples. 

For $G \ne {\rm M}_{11}$, ${\rm M}_{12}$, we find that there are no such tuples and we conclude that $R(G) = B(G) = b$. In addition, the list of large non-regular $(b-1)$-tuples is presented in the final column of Table \ref{tab:RG_sporadic} (see Remark \ref{r:spor_tab}). However, large non-regular $b$-tuples do exist when $G = {\rm M}_{11}$ or ${\rm M}_{12}$; it is easy to determine all such tuples, and it is also easy to check that every $(b+1)$-tuple is regular. So in these two cases, we deduce that $R(G) = B(G)+1$.
\end{proof}

\begin{lem}\label{l:rG2}
The conclusion to Proposition \ref{p:spor_RG} holds for the groups with $T \in \mathcal{A}_2$.
\end{lem}

\begin{proof}
First assume $G = {\rm Co}_1$, so $B(G) = 5$. Here it is possible to proceed as in Case 1, but there is a more efficient way to handle the computation, which we can also implement in some of the other remaining cases.

Let $\s = (H_1, \ldots, H_{4})$ be a maximal $4$-tuple and define $\what{Q}(G,\s)$ as in 
Lemma \ref{l:fpr}. Recall that $\s$ is regular if $\what{Q}(G,\s) < 1$. The character tables of $G$ and each maximal subgroup $H_i$ are available in the \textsf{GAP} Character Table Library \cite{GAPCTL}, together with the fusion map from $H_i$-classes to $G$-classes. This allows us to compute 
\[
{\rm fpr}(x,G/H_i) = \frac{|x^G \cap H_i|}{|x^G|}
\]
for every element $x \in G$, which in turn allows us to calculate $\what{Q}(G,\s)$. In this way, we can determine all the maximal $4$-tuples $\s$ with $\what{Q}(G,\s) \geqs 1$, up to conjugacy and ordering. The complete list is as follows:
\[
(1,1,1,1), \; (1,1,1,2), \; (1,1,1,3), \; (1,1,1,4), \; (1,1,1,5), \; (1,1,1,6),
\]
corresponding to the labelling
\[
1{:}\; {\rm Co}_2, \;\; 2{:}\; 3.{\rm Suz}{:}2, \;\; 3{:}\; 2^{11}{:}{\rm M}_{24},\;\; 4{:}\; {\rm Co}_3, \;\; 5{:}\; 2^{1+8}.\O_8^{+}(2), \;\; 6{:}\; {\rm U}_6(2){:}S_3
\]

Suppose $\s = (H_1, \ldots, H_{4})$ is a maximal $4$-tuple with $\what{Q}(G,\s) \geqs 1$, where we have $H_i = {\rm Co}_2$ for $1 \leqs i \leqs 3$. As before, we can use {\sc Magma} to construct $G$ and each component $H_i$ as subgroups of $S_n$ (with $n = |G:H_1| = 98280$) and we can use random search to show that the tuples $(1,1,1,j)$ with $2 \leqs j \leqs 6$ are regular. By \cite{BOW} we know that $b(G,H_1) = 5$, so $(1,1,1,1)$ is non-regular. Finally, for each representative $H$ of a conjugacy class of maximal subgroups of $G$, we define $\tau = (H_1, \ldots, H_{4},H)$ with $H_i = {\rm Co}_2$ for all $i$ and then using \textsf{GAP} we find that $\what{Q}(G,\tau)<1$. This allows us to conclude that  every maximal $5$-tuple is regular and thus $R(G) = B(G) = 5$. 

\vs

We can handle the groups $G = {\rm HN}$ and ${\rm HN}.2$ in a similar fashion. In both cases we have $B(G) = 3$ and we can use \textsf{GAP} to determine all the large pairs $\s$ with $\what{Q}(G,\s) \geqs 1$. This yields a short list of candidate non-regular pairs and it is straightforward to show that $\what{Q}(G,\tau) < 1$ for every large triple $\tau$, whence $R(G) = B(G) = 3$. 

To complete the proof in these two cases, we inspect each large pair $\s = (H_1,H_2)$ with $\what{Q}(G,\s) \geqs 1$. As before, if $|H_1||H_2| > |G|$ then $\s$ is non-regular. In order to handle the remaining cases, we first use {\sc Magma} to construct $G$, $H_1$ and $H_2$ as subgroups of $S_n$ with $n = 1140000$. In several cases, we can then use random search to find an element $x \in G$ with $H_1 \cap H_2^x = 1$, which allows us to conclude that $\s$ is regular. 

For example, if $G = {\rm HN}.2$ then this approach reduces the problem to determining whether or not the following pairs are regular:
\[
( 1,11 ), \; ( 1,12 ), \; (1,13), \;  ( 2,10 ), \; ( 2,11 ), \; ( 2,12 ), \; ( 3,4 )
\]
with respect to the corresponding subgroups:
\[
\begin{array}{llll}
1{:}\; S_{12}, & 2{:}\; 4.{\rm HS}.2, & 3{:}\;  {\rm U}_3(8){:}6, & 4{:}\; 2^{1+8}.(A_5 \times A_5).2^2 \\
10{:}\; 5^{2+1+2}.4.A_5.2, & 11{:}\; 3^4{:}2.(S_4 \times S_4).2, & 12{:} \; {\rm M}_{12}.2, & 13{:}\; 3^{1+4}{:}4.S_5
\end{array}
\]
Let $K_i$ be the subgroup numbered $i$ in this list. Since the given permutation representation of $G$ (of degree $1140000$) corresponds to the action of $G$ on the cosets of $K_1 = S_{12}$, it is easy to show that $K_{13}$ has a regular orbit on $G/K_1$, but $K_{11}$ and $K_{12}$ do not. So $(1,13)$ is regular and $(1,11), (1,12)$ are non-regular. Similarly, using \texttt{CosetAction} we can show that $K_{11}$ has a regular orbit on $G/K_2$, but $K_{10}$ and $K_{12}$ do not. In the same way, we deduce that $K_{4}$ does not have a regular orbit on $G/K_3$. Therefore, $(2,11)$ is regular, whereas $(2,10)$, $(2,12)$ and $(3,4)$ are non-regular. The case $G = {\rm HN}$ is very similar.

\vs

Next assume $G = {\rm Fi}_{24}'$, so $B(G) = 5$ and we claim that $R(G) = 5$. To see this, we first use the \textsf{GAP} Character Table Library to identify all the maximal $4$-tuples $\s$ with $\what{Q}(G,\sigma) \geqs 1$. The tuples that arise are of the form $(1,1,1,i)$ with $1 \leqs i \leqs 6$, where the component subgroups are labelled as follows:
\[
1{:} \; {\rm Fi}_{23}, \;\; 2{:} \; 2.{\rm Fi}_{22}.2, \;\; 3{:} \; (3 \times {\rm P\O}_{8}^{+}(3).3).2, \;\; 4{:} \; \O_{10}^{-}(2), \;\; 5{:} \; 3^7.\O_7(3), \;\; 6{:}\; 3^{1+10}{:}{\rm U}_5(2).2
\]
Therefore, any non-regular maximal $5$-tuple must be of the form 
\[
\tau = (H_1,H_2,H_3,H_4,H_5)
\]
(up to conjugacy and ordering), where $H_1 = H_2 = H_3 = {\rm Fi}_{23}$ and $H_4$ is one of the six subgroups listed above. Using \textsf{GAP} we find that $\what{Q}(G,\tau) \geqs 1$ if and only if every component of $\tau$ is conjugate to ${\rm Fi}_{23}$, which we know is regular since $B(G) = 5$. Therefore $R(G) = B(G) = 5$ and it just remains to determine all the non-regular maximal $4$-tuples.

To do this, we return to the above list of maximal $4$-tuples $\s$ with $\what{Q}(G,\sigma) \geqs 1$, noting that $(1,1,1,1)$ is non-regular by \cite{BOW}. To handle the remaining possibilities, we first use {\sc Magma} to construct $G$ as a permutation group of degree $306936$, which corresponds to the action of $G$ on the cosets of ${\rm Fi}_{23}$. The command \texttt{MaximalSubgroups} is not effective, but we can construct representatives of the relevant conjugacy classes of maximal subgroups using explicit generators in the Web Atlas \cite{WebAt}, which are given as words in the standard generators for $G$ (more precisely, we work in $L = G.2 = {\rm Fi}_{24}$ and we use explicit generators to construct $N_L(H)$, intersecting with $G$ to get $H$). Then random search shows that each $4$-tuple $(1,1,1,i)$ with $3 \leqs i \leqs 6$ is regular. 

Finally, suppose $\s = (H,H,H,K)$, where $H = {\rm Fi}_{23}$ and $K = 2.{\rm Fi}_{22}.2$. Here we can work directly with the permutation representation of $G$ of degree $306936$, which corresponds to the action of $G$ on $\Gamma = G/H$. It is straightforward to compute the order of every $3$-point stabiliser with respect to the action of $K$ on $\Gamma$ and we find that every such stabiliser is nontrivial. This immediately implies that $\s$ is non-regular and this completes the proof for $G = {\rm Fi}_{24}'$. The case $G = {\rm Fi}_{24}$ is entirely similar.
\end{proof}

Next we turn to the groups in $\mathcal{A}_3$. We begin by establishing the non-regularity of a certain collection of maximal pairs.

\begin{lem}\label{l:derek}
All of the maximal pairs $(H,K)$ in Table \ref{tab:derek} are non-regular. 
\end{lem}

{\scriptsize
\begin{table}
\[
\begin{array}{llll} \hline
G & H & K & |G:H| \\ \hline
{\rm Ly} & G_2(5) & 3^5{:}(2 \times {\rm M}_{11}) & 8835156 \\
& 3.{\rm McL}.2 & 5^{1+4}{:}4S_6 & 9606125 \\
& 3.{\rm McL}.2 & 3^5{:}(2 \times {\rm M}_{11}) & 9606125 \\
{\rm Th} & {}^3D_4(2){:}3 & 2^{1+8}.A_9 & 143127000 \\
& {}^3D_4(2){:}3 & {\rm U}_3(8){:}6 & 143127000 \\
{\rm J}_4 & 2^{11}.{\rm M}_{24} & {\rm U}_3(11).2 & 173067389 \\
& 2^{1+12}.3.{\rm M}_{22}.2 & 2^{3+12}.(S_5 \times {\rm L}_3(2)) & 3980549947 \\ \hline
\end{array}
\]
\caption{The non-regular pairs in Lemma \ref{l:derek}}
\label{tab:derek}
\end{table}
}

\begin{proof}
First observe that $|H||K|< |G|$ in every case, so we cannot appeal to Lemma \ref{l:triv}. Using {\sc Magma}, we can construct $G$, $H$ and $K$ as subgroups of the matrix group ${\rm GL}_d(p)$, where $(d,p) = (111,5)$, $(248,2)$ and $(112,2)$ for $G = {\rm Ly}$, ${\rm Th}$ and ${\rm J}_4$, respectively. This allows us to randomly search for an element $g \in G$ with $H \cap K^g = 1$, but this search is inconclusive in every case. In addition, the size of $|G:H|$ is prohibitive and we are not able to use \texttt{CosetAction} to construct $G$ as a permutation group on $G/H$, as we have done in previous cases.

We thank Derek Holt (personal communication) for resolving the first $6$ cases in Table \ref{tab:derek}. To illustrate his approach, let us assume $G = {\rm Ly}$, $H = G_2(5)$ and $K = 3^5{:}(2 \times {\rm M}_{11})$, as in the first row of the table, and consider the irreducible module $V = \mathbb{F}_5^{111}$ for $G$. Holt proceeds by identifying an irreducible $7$-dimensional submodule $W$ for $H$, which he then uses, via the function \texttt{OrbitImage}, to construct $G$ as a permutation group on $G/H$. Then as before, he shows that $(H,K)$ is non-regular by determining all of the $K$-orbits on $G/H$. The other $5$ cases are handled in a similar fashion.

Finally, suppose $G = {\rm J}_4$ and $\tau = (H,K)$, where
\[
H = 2^{1+12}.3.{\rm M}_{22}.2,\;\; K = 2^{3+12}.(S_5 \times {\rm L}_3(2)).
\]
This case is more difficult since the index $|G:H| = 3980549947$ is very large. 
We are grateful to J\"urgen M\"uller (personal communication) for resolving this case (which was later independently verified by Holt). M\"uller did this by using the \textsf{GAP} package \textsc{Orb} \cite{Orb} to establish the existence of $54$ non-regular $K$-orbits on $\Gamma = G/H$ (via random search), labelled $\Gamma_1, \ldots, \Gamma_{54}$, with the property that
\[
|G:H| - \sum_{i=1}^{54}|\Gamma_i|  = 3980549947 - 3980074560 < |K|.
\]
This immediately implies that $K$ has no regular orbit on $G/H$, so $(H,K)$ is non-regular and the proof of the lemma is complete.
\end{proof}

\begin{lem}\label{l:rG3}
The conclusion to Proposition \ref{p:spor_RG} holds for the groups with $T \in \mathcal{A}_3$.
\end{lem}

\begin{proof}
Here $G = {\rm Ly}$, ${\rm Th}$ or ${\rm J}_4$, with $B(G) = 3$ for each group. As in previous cases, we can use \textsf{GAP} to show that $\what{Q}(G,\tau)<1$ for every maximal triple $\tau$, whence $R(G) = B(G) = 3$ and it just remains to determine all the non-regular maximal pairs.

All three cases are very similar, so for brevity we will assume $G = {\rm Th}$. As usual, we begin by using \textsf{GAP} to find all the maximal pairs $\s = (H_1,H_2)$ with $\what{Q}(G,\s) \geqs 1$, which yields the following list
\[
(1,i),\, (2,2), \, (2,3), \, (3,3)
\]
where $1 \leqs i \leqs 5$ and the relevant subgroups are as follows:
\[
1{:} \;\; {}^3D_4(2){:}3, \;\; 2{:}\;\; 2^5.{\rm L}_5(2), \;\; 3{:}\;\; 2^{1+8}.A_9, \;\; 4{:} \;\; {\rm U}_3(8){:}6, \;\; 5{:}\;\; (3 \times G_2(3)){:}2
\]
From \cite{BOW}, we see that $(1,1)$ and $(2,2)$ are non-regular, whereas $(3,3)$ is regular. And by computing orders, we deduce that $(1,2)$ is non-regular. We can then use {\sc Magma} and random search to show that $(1,5)$ and $(2,3)$ are regular, working inside the matrix group ${\rm GL}_{248}(2)$. Finally, we note that the two remaining pairs $(1,3)$ and $(1,4)$ were handled in Lemma \ref{l:derek}.
\end{proof}

In order to complete the proof of Proposition \ref{p:spor_RG}, we may assume $G = \mathbb{B}$ or $\mathbb{M}$.

\begin{lem}\label{l:rG4}
The conclusion to Proposition \ref{p:spor_RG} holds for the groups with $T \in \mathcal{A}_4$.
\end{lem}

\begin{proof}
First assume $G = \mathbb{B}$ is the Baby Monster and note that $B(G) = 4$. As before,  the character table of $G$ and every maximal subgroup of $G$ is available in the \textsf{GAP} Character Table Library \cite{GAPCTL}. Moreover, if $H$ is a maximal subgroup, then the fusion map from $H$-classes to $G$-classes is also available unless  $H = (2^2 \times F_4(2)){:}2$. In order to handle the latter case, we can use the function \texttt{PossibleClassFusions} to determine a list of $64$ candidate fusion maps and we find that each fixed point ratio ${\rm fpr}(x,G/H)$ is independent of the choice of map. Therefore, we can compute $\what{Q}(G,\tau)$ precisely for every maximal $k$-tuple $\tau$ and we can now proceed as we have done in previous cases.

First we use \textsf{GAP} to determine the list of maximal triples $\s$ with $\what{Q}(G,\s) \geqs 1$: 
\[
( 1, 1, i ), \;\; (1,1,6), \;\; (1,1,7), \;\; (1,2,2)
\]
where $1 \leqs i \leqs 4$ and the relevant maximal subgroups are labelled as follows:
\begin{equation}\label{e:bmax}
\begin{array}{llll}
1{:}\;\; 2.{}^2E_6(2){:}2, & 2{:}\;\; 2^{1+22}.{\rm Co}_2, & 3{:}\;\; {\rm Fi}_{23}, & 4{:}\;\; 2^{9+16}.{\rm Sp}_8(2), \\
5{:}\;\; {\rm Th}, & 6{:}\;\; (2^2 \times F_4(2)){:}2, & 7{:}\;\; 2^{2+10+20}.({\rm M}_{22}{:}2 \times S_3) &
\end{array}
\end{equation}
Working with this list of candidate non-regular triples, it is easy to check that $\what{Q}(G,\tau) < 1$ for every maximal $4$-tuple $\tau$ and thus $R(G) = B(G) = 4$. We refer the reader to Remark \ref{r:RG_spor}(a) for comments on the problem of determining all the non-regular maximal triples when $G = \mathbb{B}$.

\vs

Finally, let us assume $G = \mathbb{M}$ is the Monster, so $B(G) = 3$ by \cite{BOW}.
We claim that $\what{Q}(G,\tau) < 1$ for every maximal triple $\tau$, which implies that $R(G) = B(G) = 3$.

First recall that $G$ has $46$ conjugacy classes of maximal subgroups (see the main theorem of \cite{DLP}). For $31$ of these classes, the character table of a representative $H$ and the fusion map from $H$-classes to $G$-classes is available in the \textsf{GAP} Character Table Library \cite{GAPCTL} via the function \texttt{NamesOfFusionSources}. The relevant subgroups are as follows:

{\small
\[
\begin{array}{llll}
2.\mathbb{B} & 2^{1+24}.{\rm Co}_1 & 3.{\rm Fi}_{24} & 2^2.{}^2E_6(2){:}S_3 \\
3^{1+12}.2.{\rm Suz}{:}2 & S_3 \times {\rm Th} & ({\rm D}_{10} \times {\rm HN}).2 &
5^{1+6}{:}2.{\rm J}_2.4 \\
(7{:}3\times {\rm He}){:}2 & (A_5 \times A_{12}){:}2 & 5^{3+3}.(2 \times {\rm L}_3(5)) & (A_6)^3.(2\times S_4) \\
(A_5 \times {\rm U}_3(8){:}3){:}2 & 5^{2+2+4}{:}(S_3 \times {\rm GL}_2(5)) & ({\rm L}_3(2) \times {\rm Sp}_4(4){:}2).2 & 7^{1+4}{:}(3 \times 2S_7) \\
(5^2{:}[2^4] \times {\rm U}_3(5)).S_3 & ({\rm L}_2(11) \times {\rm M}_{12}){:}2 & (A_7 \times (A_5 \times A_5){:}2^2){:}2 & 5^4{:}(3 \times 2.{\rm L}_2(25)).2 \\
7^{2+1+2}{:}{\rm GL}_2(7) & {\rm M}_{11}\times A_6.2^2 & (S_5)^3{:}S_3 & 
13^2{:}2.{\rm L}_2(13).4 \\
(7^2{:}(3 \times 2A_4) \times {\rm L}_2(7)).2 & (13{:}6 \times {\rm L}_3(3)).2 & 13^{1+2}{:}(3 \times 4S_4) & {\rm L}_2(71) \\
{\rm L}_2(59) & {\rm L}_2(41) & 41{:}40 & 
\end{array}
\]}

Let $\mathcal{M}_1$ denote this set of maximal subgroups of $G$ and let $x_1, \ldots, x_t$ be a complete set of representatives of the conjugacy classes in $G$ of elements of prime order. Let $\tau = (H_1, H_2, H_3)$ be a triple of subgroups with each $H_i \in \mathcal{M}_1$ and recall that 
\[
\what{Q}(G,\tau) = \sum_{i=1}^t |x_i^G| \cdot \left(\prod_{j=1}^3 {\rm fpr}(x_i,G/H_j)\right).
\]
We can use the stored fusion maps in \textsf{GAP} to compute ${\rm fpr}(x_i,G/H)$ for all $i$  and all $H \in \mathcal{M}_1$, whence
\begin{equation}\label{e:sporbd}
\what{Q}(G,\tau) \leqs \sum_{i=1}^t a_ib_i^3,
\end{equation}
where 
\[
a_i = |x_i^G|,\;\; b_i = \max\{ {\rm fpr}(x_i,G/H) \,:\, H \in \mathcal{M}_1\}.
\]
For example, if $x_1$ is a \texttt{2A}-involution, then 
\[
a_1 = 97239461142009186000, \;\; b_1 = \frac{56416310497}{467497409336582625}
\]
(with ${\rm fpr}(x_1,G/H) = b_1$ if $H = 2.\mathbb{B}$). 
It is entirely straightforward to check that the upper bound in \eqref{e:sporbd} yields $\what{Q}(G,\tau) < 1$ and thus $\tau$ is regular. 

So to complete the proof, we need to extend our analysis to triples containing one or more subgroups from the remaining $15$ conjugacy classes of maximal subgroups $H$ for which the fusion map from $H$-classes to $G$-classes is not available in \cite{GAPCTL}. 

Let $r$ be a prime divisor of $|G|$ and define
\[
i_r(H) = |\{ x \in H \,:\, |x|=r\}|,\;\; c_r = \min\{a_ib_i \,:\, |x_i|=r\}.
\]
Notice that if 
\begin{equation}\label{e:ir0}
i_r(H) \leqs c_r
\end{equation}
for all $r$, then ${\rm fpr}(x_i,G/H) \leqs b_i$ for all $i$. So by our previous calculation, if $\mathcal{S}$ is a collection of subgroups of $G$ such that \eqref{e:ir0} holds for all $r$ and all $H \in \mathcal{S}$, then every triple of subgroups in $\mathcal{S}$ is regular. For the reader's convenience, the values of $c_r$ are recorded in Table \ref{tab:cr}.

{\scriptsize
\begin{table}
\[
\begin{array}{llllll} \hline
r & c_r & & & & \\ \hline
2 & 11734592583376 & & & 23 & 90321336548400569373425664000000 \\
3 & 500595349782528000 & & & 29 & 86565910978666325606400 \\
5 & 9367743238695946498867200 & & & 31 & 134025209071820199715405824000000 \\
7 & 3954417208796381184000 & & & 41 & 1680 \\
11 & 3147561728201838023619379200000 & & & 47 & 88399605983540982791012352000000 \\
13 & 447153330533129256960000 & & & 59 & 1740 \\
17 & 61099727665094502811435008000000 & & & 71 & 2520 \\
19 & 109336354769116478715199488000000 & & & & \\ \hline
\end{array}
\]
\caption{The values of $c_r$ in the proof of Lemma \ref{l:rG4}}
\label{tab:cr}
\end{table}
}

We claim that the inequality in \eqref{e:ir0} holds (for every prime $r$) when $H$ is any one of the following maximal subgroups:
\[
({\rm L}_2(11) \times {\rm L}_2(11)).4, \, 11^2{:}(5 \times 2A_5), \, {\rm U}_3(4){:}4, \, {\rm L}_2(29){:}2, \, 7^2{:}{\rm SL}_2(7), \, {\rm L}_2(19){:}2, \, {\rm L}_2(13){:}2
\]
\[
3^{3+2+6+6}.({\rm L}_3(3) \times {\rm SD}_{16}),\, 3^{2+5+10}.({\rm M}_{11} \times 2S_4),\, (3^2{:}2 \times {\rm P\Omega}_8^{+}(3)).S_4
\]
\[
2^{3+6+12+18}.({\rm L}_3(2) \times 3S_6), \, 2^{2+11+22}.({\rm M}_{24} \times S_3)
\]
Indeed, in each case we can work with a permutation representation of $H$ in the Web Atlas \cite{WebAt}, which allows us to compute $i_r(H)$ and then verify the desired bound. For example, if $H = 2^{2+11+22}.({\rm M}_{24} \times S_3)$ then the Web Atlas provides a representation of $H$ on $294912$ points and we compute:
\[
i_2(H) = 43521572863,\; i_3(H) = 32114946867200, \; i_5(H) =  68457483730944
\]
\[
i_7(H) = 1564742485278720, \; i_{11}(H) = 23897885229711360,\; i_{23}(H) =  182870773931704320,
\]
noting that $i_r(H) = 0$ for all other primes. It is routine to check that \eqref{e:ir0} holds. Similarly, if $H = 2^{5+10+20}.(S_3 \times {\rm L}_5(2))$ or $2^{10+16}.\Omega_{10}^{+}(2)$, then one checks that the upper bound on $i_r(H)$ given in \cite[Proposition 3.8]{BOW} is sufficient. 

Finally, let us assume $H = 3^8.{\rm P\Omega}_{8}^{-}(3).2$. Here the Web Atlas provides a permutation representation of $H$ on $805896$ points and as before we can use this to compute $i_r(H)$ for every prime divisor $r$ of $|G|$. In this way, we find that the inequality in \eqref{e:ir0} holds unless $r=41$. More precisely, 
$G$ has a unique class of elements of order $41$ and we get
\[
i_{41}(H) = 16245625881139200 > c_{41} = 1680.
\]
However, if we now redefine $b_i$ to be $i_{41}(H)/a_i$ (for the unique $i$ such that $|x_i| = 41$) then it is easy to check that the inequality $\sum_i a_ib_i^3 < 1$ is still satisfied and we conclude that any maximal triple involving $H$ is regular. 

Putting all of this together, it follows that every maximal triple is regular and thus $R(G) = B(G) = 3$.
\end{proof}

This completes the proof of Proposition \ref{p:spor_RG}. In particular, this concludes our proof of Theorem \ref{t:main2}(i).

\begin{rem}\label{r:RG_spor}
Let us briefly comment on the problem of determining all the non-regular maximal $k$-tuples for $G = \mathbb{B}, \mathbb{M}$ with $k=R(G)-1$.

\vspace{1mm}

\begin{itemize}\addtolength{\itemsep}{0.3\baselineskip}
\item[{\rm (a)}] For $G = \mathbb{B}$ we have proved that $R(G) = B(G) = 4$, but we have not been able to classify all the non-regular maximal triples. Let $\s$ be a maximal triple and recall that we can compute $\what{Q}(G,\s)$ with the aid of the \textsf{GAP} Character Table Library \cite{GAPCTL}. In this way, we deduce that $\s$ is non-regular only if it is one of the following:
\[
(1,1,1), \; (1,1,2), \; (1,1,3), \; (1,1,4), \; (1,1,6), \; (1,1,7), \; (1,2,2),
\]
where we are using the numbering in \eqref{e:bmax}. We know that $(1,1,1)$ is non-regular by \cite{BOW}, and it is also easy to see that $(1,1,2)$ is non-regular by comparing orders. It remains an open problem to determine the regularity status of the following triples:
\[
(1,1,3), \; (1,1,4), \; (1,1,6), \; (1,1,7),\; (1,2,2).
\]
     
\item[{\rm (b)}] For the Monster $G = \mathbb{M}$ we have $R(G) = B(G) = 3$ and we have not attempted to determine the non-regular maximal pairs for $G$. From \cite{BOW}, we see that $(2.\mathbb{B},2.\mathbb{B})$ is the only non-regular conjugate maximal pair. In fact, \cite[Theorem 3.1]{B23} implies that this is the only conjugate non-regular pair for $G$ since $b(G,H)=2$ for every nontrivial proper subgroup $H \ne 2.\mathbb{B}$. By comparing orders, we also note that $(2.\mathbb{B}, 2^{1+24}.{\rm Co}_1)$ is non-regular. 
\end{itemize}
\end{rem}

\subsection{Soluble subgroups}\label{ss:spor_sol}
 
In this final section, we complete the proof of Theorem \ref{t:main2} by establishing the bound $R_{{\rm sol}}(G) \leqs 3$ for all almost simple sporadic groups. This generalises the main theorem of \cite{B23}, where the weaker bound $B_{{\rm sol}}(G) \leqs 3$ is established. It also establishes a strong form of Conjecture \ref{con:main}(iii) in this setting. Our main result is the following.

\begin{prop}\label{p:nonregsol}
We have $R_{{\rm sol}}(G) \leqs 3$ for every almost simple sporadic group $G$. More precisely, 
\[
R_{\rm sol}(G) = \left\{\begin{array}{ll}
2 & \mbox{if $G \in \mathcal{A}_1$} \\
3 & \mbox{if $G \in \mathcal{A}_2$} \\
\mbox{$2$ or $3$} & \mbox{if $G \in \mathcal{A}_3$}
\end{array}\right.
\]
where
\begin{align*}
\mathcal{A}_1 & = \{ {\rm J}_1, {\rm J}_3, {\rm J}_3.2, {\rm HS}, {\rm Suz}, {\rm Suz}.2, {\rm McL}, {\rm McL}.2, {\rm Ru}, {\rm He}, {\rm He}.2, {\rm Co}_3, {\rm O'N}, {\rm O'N}.2 \} \\
\mathcal{A}_2 & = \{ {\rm M}_{11}, {\rm M}_{12}, {\rm M}_{12}.2, {\rm M}_{22}, {\rm M}_{22}.2, {\rm M}_{23}, {\rm M}_{24}, {\rm J}_2, {\rm J}_2.2, {\rm HS}.2, {\rm Co}_2, {\rm Fi}_{22}, {\rm Fi}_{22}.2, {\rm Fi}_{23} \} \\
\mathcal{A}_3 & = \{ {\rm Co}_{1}, {\rm HN}, {\rm HN}.2, {\rm J}_{4}, {\rm Ly}, {\rm Th}, {\rm Fi}_{24}', {\rm Fi}_{24}, \mathbb{B}, \mathbb{M} \}
\end{align*}
\end{prop}

\begin{rem}\label{r:sporsol}
Let us record some comments on the statement of Proposition \ref{p:nonregsol}.

\vspace{1mm}

\begin{itemize}\addtolength{\itemsep}{0.3\baselineskip}
\item[{\rm (a)}] Clearly, we have $R_{{\rm sol}}(G) \leqs R(G)$, so the main bound $R_{\rm sol}(G) \leqs 3$ follows immediately from Proposition \ref{p:spor_RG} when $G$ is one of the following:
\[
{\rm J}_1, \, {\rm J}_3, \, {\rm J}_3.2, \, {\rm O'N}, \, {\rm O'N}.2, \, {\rm HN}, \, {\rm HN}.2, \, {\rm J}_4, \, {\rm Ly}, \, {\rm Th}, \, \mathbb{M}.
\]

\item[{\rm (b)}] We can determine all the non-regular soluble pairs for the following groups
\begin{equation}\label{e:list0}
{\rm M}_{11}, \, {\rm M}_{12}, \, {\rm M}_{12}.2, \, {\rm M}_{22}, \, {\rm M}_{22}.2, \, {\rm M}_{23}, \, {\rm M}_{24}, \, {\rm J}_2, \, {\rm J}_2.2, \, {\rm HS}.2
\end{equation}
In particular, if $G = {\rm M}_{23}$, then $(H,H)$ is the only non-regular soluble pair, up to conjugacy, where 
\[
H = 2^4{:}(3 \times A_4){:}2 < 2^4{:}(3 \times A_5){:}2 < G
\]
is a second maximal subgroup of $G$. Similarly, $G = {\rm HS}.2$ has a unique non-regular soluble pair, up to ordering and conjugacy, namely $(H,K)$, where
\begin{align*}
H & = 2^{1+6}.S_4 < 2^{1+6}.S_5 < G \\
K & = 4^3{:}(2 \times 7{:}3) < 4^3{:}(2 \times {\rm L}_3(2)) < G
\end{align*}
are both second maximal subgroups. In Table \ref{tab:spor_sol_pairs}, for each of the remaining groups in \eqref{e:list0}, we record a non-regular conjugate  pair $(H,H)$ with $|H|$ maximal, and also a non-regular non-conjugate  pair $(H,K)$ with $|H||K|$ maximal (in every case,  $K$ is a second maximal subgroup of $G$, and $H$ is either maximal or second maximal). 
 
\item[{\rm (c)}] For the remaining groups in $\mathcal{A}_2$, namely ${\rm Co}_2$, ${\rm Fi}_{22}$, ${\rm Fi}_{22}.2$ and ${\rm Fi}_{23}$, a non-regular conjugate soluble pair $(H,H)$ is recorded in \cite[Table 1]{B23}. But in these cases we have not classified all the non-regular soluble pairs. By inspecting \cite[Table 2]{B23}, we observe that $(H,K)$ is a non-regular non-conjugate pair for $G = {\rm Fi}_{22}.2$, where the first component $H = 3^{1+6}.2^{3+4}{:}3^2{:}2.2$ is maximal and $K = 3^{1+6}.2^{3+4}{:}3^2{:}2 < H$ is an index-two subgroup. Similarly, if $G = {\rm Fi}_{23}$ then $(H,K)$ is non-regular, where $H = 3^{1+8}.2^{1+6}.3^{1+2}.2S_4$ is maximal and $K$ is the unique index-two subgroup of $H$.

\item[{\rm (d)}] We have not been able to determine $R_{{\rm sol}}(G)$ precisely for the groups in $\mathcal{A}_3$. By \cite[Theorem 2]{B23}, none of these groups, with the possible exception of the Baby Monster $\mathbb{B}$, has a non-regular conjugate soluble pair. Determining the existence (or otherwise) of such a pair for $\mathbb{B}$ remains an open problem.

\item[{\rm (e)}] We refer the reader to the end of Section \ref{ss:spor_sol} for some brief comments regarding the regularity of nilpotent pairs $(H,K)$ in almost simple sporadic groups, as well as the pairs $(H,K)$ with $H$ nilpotent and $K$ soluble.
\end{itemize}
\end{rem}

{\scriptsize
\begin{table}
\[
\begin{array}{llll} \hline
G & H & K & \mbox{Comments} \\ \hline
{\rm M}_{11} & {\rm U}_3(2){:}2 & {\rm U}_3(2) & \\
{\rm M}_{12} & {\rm AGL}_2(3) & {\rm ASL}_2(3) & \\
{\rm M}_{12}.2 & {\rm AGL}_2(3) & 4^2{:}{\rm D}_{12}.2 & H < {\rm M}_{12} < G, \; K < G \\  
{\rm M}_{22} & 2^4{:}(3^2{:}4) & 2^4{:}S_4 & H < 2^4{:}A_6 < G,\; K < 2^4{:}A_6 < G \\
{\rm M}_{22}.2 & 2^4{:}(S_3 \wr S_2) & 2^4{:}(S_4 \times S_2) & H < 2^4{:}S_6 < G, \; K < 2^4{:}S_6 < G \\
{\rm M}_{24} & 2^6{:}3.(S_3 \wr S_2) & 2^6{:}(7{:}3 \times S_3) & H < 2^6{:}3.S_6 < G,\; K < 2^6{:}({\rm L}_3(2) \times S_3) < G \\
{\rm J}_2 & 2^{2+4}{:}(3 \times S_3) & 2^{2+4}{:}3^2 & \\
{\rm J}_2.2 & 2^{2+4}{:}(3 \times S_3).2 & 2^{2+4}{:}(3 \times S_3) & \\ \hline
\end{array}
\]
\caption{The subgroups $H$ and $K$ in Remark \ref{r:sporsol}(b)}
\label{tab:spor_sol_pairs}
\end{table}
}

Let $G$ be an almost simple sporadic group with socle $T$. In order to prove Proposition \ref{p:nonregsol} we will adopt a computational approach, working with {\sc Magma} (version 2.28-4) \cite{magma} and \textsf{GAP} (version 4.11.1). It will be convenient to divide up the possibilities for $T$ as follows:
\begin{align*}
\mathcal{B}_1 & = \{ {\rm M}_{11}, {\rm M}_{12}, {\rm M}_{22}, {\rm M}_{23}, {\rm M}_{24}, {\rm J}_{1}, {\rm J}_{2}, {\rm J}_{3}, {\rm HS}, {\rm McL}, {\rm He}, {\rm Suz}, {\rm Ru}, {\rm Co}_{3}, {\rm O'N} \} \\
\mathcal{B}_2 & = \{ {\rm Co}_2, {\rm Fi}_{22}, {\rm Fi}_{23}, {\rm HN}, {\rm J}_4, {\rm Ly}, {\rm Th}, \mathbb{M} \} \\
\mathcal{B}_3 & = \{ {\rm Co}_1, {\rm Fi}_{24}' \} \\
\mathcal{B}_4 & = \{ \mathbb{B}\}
\end{align*}

Before we begin the proof, it will be helpful to introduce the following notation. For a positive integer $k$, let $\mathcal{C}_k(G)$ be the set of subgroups $H$ of $G$ with the property that there exists a chain of subgroups
\[
H= H_0 < H_1 < H_2 < \cdots < H_{k-1} < H_k = G
\]
where $H_1$ is insoluble and each $H_i$ is a maximal subgroup of $H_{i+1}$. For instance, $\mathcal{C}_1(G)$ is the set of maximal subgroups of $G$. In addition, let $\mathcal{D}_k(G)$ be the set of soluble subgroups in $\mathcal{C}_k(G)$, and let $\mathcal{P}(G)$ be the set of all non-Sylow $p$-subgroups of $G$, ranging over all prime divisors $p$ of $|G|$. 

\begin{defn}\label{d:SH}
Let $H \leqs G$ be a subgroup and let $d$ be a positive integer. If $H$ is soluble, we define 
$\mathcal{S}_d(H) = \{ H \}$. Otherwise, we take $\mathcal{S}_d(H)$ to be a set of representatives of the $G$-classes of subgroups in the set
\[
\left(\mathcal{C}_d(H) \cup \bigcup_{k=1}^{d-1} \mathcal{D}_k(H) \right) \setminus \mathcal{P}(H)
\]
\end{defn}

For example, if $G = H = {\rm M}_{11}$, then $\mathcal{S}_2(H)$ contains the soluble maximal subgroups ${\rm U}_3(2){:}2$ and $2S_4$, together with a representative of each conjugacy class of maximal subgroups of ${\rm M}_{10} = A_6.2$, ${\rm L}_{2}(11)$ and $S_5$.

The key property here is that for any positive integer $d$, every soluble subgroup of $H$ is $G$-conjugate to a subgroup of at least one group in $\mathcal{S}_d(H)$. It follows that $R_{\rm sol}(G) \leqs 3$ if there exist positive integers $d_j$ such that every triple $\tau = (H_1,H_2,H_3)$ is regular, where the components of $\tau$ range over the subgroups in 
\[
\bigcup_{j=1}^t \mathcal{S}_{d_j}(M_j)
\]
with respect to a complete set of representatives $\{M_1, \ldots, M_t\}$ of the conjugacy classes of maximal subgroups of $G$.

\begin{lem}\label{l:solspor1}
The conclusion to Proposition \ref{p:nonregsol} holds for the groups with $T \in \mathcal{B}_1$.
\end{lem}

\begin{proof}
Let $G$ be an almost simple sporadic group with socle $T \in \mathcal{B}_1$. Working with {\sc Magma}, we first use the function \texttt{AutomorphismGroupSimpleGroup} to construct a faithful primitive permutation representation of $G$ of minimal degree. We  then use \texttt{SolubleSubgroups} to obtain a complete set of representatives of the conjugacy classes of soluble subgroups of $G$. Let $\mathcal{S}(G)$ be the set obtained by removing all the non-Sylow $p$-subgroups, ranging over all the prime divisors $p$ of $|G|$. We then use random search to determine a list of candidate non-regular soluble pairs $(H_1,H_2)$, where $|H_1| \geqs |H_2|$ and $H_i \in \mathcal{S}(G)$ for $i=1,2$.

Clearly, if random search reveals that every pair $(H_1,H_2)$ with $H_i \in \mathcal{S}(G)$ is regular, then $R_{\rm sol}(G) = 2$. So let us assume that there is at least one such pair $(H_1,H_2)$ for which our random search is inconclusive. Then as in the proof of Proposition \ref{p:spor_RG}, we can determine the regularity status of $(H_1,H_2)$ by computing the orbits of $H_2$ on $G/H_1$ via the function \texttt{CosetAction}. In this way, we obtain the complete list of non-regular soluble pairs, up to ordering and conjugacy, and we can easily extract the information highlighted in part (b) of Remark \ref{r:sporsol}. Of course, if we find that there are no such pairs, then $R_{\rm sol}(G) = 2$.

To complete the argument, let us assume $G$ has at least one non-regular soluble pair. In order to show that every soluble triple is regular, we consider each non-regular soluble pair $(H_1,H_2)$ in turn, and we construct the triples $\tau = (H_1,H_2,H_3)$, where $H_3$ runs through the set $\mathcal{S}(G)$. Using random search, it is straightforward to verify that $\tau$ is regular and we conclude that $R_{{\rm sol}}(G) = 3$ in each of these cases.
\end{proof}

\begin{lem}\label{l:solspor2}
The conclusion to Proposition \ref{p:nonregsol} holds for the groups with $T \in \mathcal{B}_2$.
\end{lem}

\begin{proof}
If $T \in \{{\rm HN}, {\rm J}_4, {\rm Ly}, {\rm Th}, \mathbb{M}\}$ then $R(G) = 3$ by Proposition \ref{p:spor_RG} and thus $R_{{\rm sol}}(G) \leqs 3$, as required. For the remainder, let us assume $T \in \{ {\rm Co}_2, {\rm Fi}_{22}, {\rm Fi}_{23}\}$. In these cases, the main theorem of \cite{B23} gives $B_{{\rm sol}}(G) = 3$, so we just need to show that every soluble triple is regular in order to conclude that $R_{{\rm sol}}(G) = 3$. To do this, we need to modify the computational approach adopted in the proof of Lemma \ref{l:solspor1} since the {\sc Magma} function \texttt{SolubleSubgroups} is ineffective for the groups we are working with here. We proceed as follows.

For now, let us assume $G \ne {\rm Fi}_{22}.2$, so $G = T$.  As before, we begin by constructing $G$ as a permutation group and we use the {\sc Magma} function \texttt{MaximalSubgroups} to obtain a complete set of representatives of the conjugacy classes of maximal subgroups of $G$. Then for each maximal subgroup $H$, we construct the set of subgroups $\mathcal{S}_3(H)$ as defined in Definition \ref{d:SH}. In this way, we obtain a set of subgroups with the property that every soluble subgroup of $G$ is contained in a conjugate of at least one member of the set. We can then use random search to show that every triple of the form $(J_1,J_2,J_3)$ is regular, where each $J_i$ is contained in $\mathcal{S}_3(H_i)$ for some maximal subgroup $H_i$ of $G$, and this allows us to conclude that $R_{\rm sol}(G) \leqs 3$. 

In practice, to improve the efficiency of the latter computation, we first run through triples of the form $(J_1,H_2,H_3)$, where $J_1 \in \mathcal{S}_3(H_1)$ and the $H_i$ are representatives of the classes of maximal subgroups of $G$. Given $H_2$ and $H_3$, if random search is inconclusive for at least one  triple of the form $(J_1,H_2,H_3)$, then we run through the set of triples $\tau = (J_1,J_2,H_3)$ with $J_i \in \mathcal{S}_3(H_i)$ for $i=1,2$, and in every case we find that $\tau$ is regular. 

Finally, suppose $G = {\rm Fi}_{22}.2$. We know that $R_{{\rm sol}}(T) = 3$ by the above argument, so Lemma \ref{l:easy} implies that every soluble triple of the form $(H,K,L)$ with $H < T$ is regular. So we can proceed as above, using the additional fact that we can discard any triples with a component contained in $T$. 
\end{proof}

\begin{lem}\label{l:solspor3}
The conclusion to Proposition \ref{p:nonregsol} holds for the groups with $T \in \mathcal{B}_3$.
\end{lem}

\begin{proof}
First assume $G = {\rm Fi}_{24}'$. Here we proceed as in the proof of Proposition \ref{p:spor_RG}, using the information in the \textsf{GAP} Character Table Library \cite{GAPCTL} to determine all the maximal triples $\s$ with $\what{Q}(G,\s) \geqs 1$. It turns out that this inequality holds only if at least one of the component subgroups is ${\rm Fi}_{23}$. But we know that $R_{{\rm sol}}({\rm Fi}_{23}) =3$ by Lemma \ref{l:solspor2}, so Lemma \ref{l:easy} implies that every soluble triple is regular and thus $R_{\rm sol}(G) \leqs 3$.

\vs

Next assume $G = {\rm Co}_1$. Here we observe that the largest maximal subgroup of $G$ is ${\rm Co}_2$, and we have $R_{\rm sol}({\rm Co}_2) = 3$ by Lemma \ref{l:solspor2}. Similarly, $R_{\rm sol}({\rm Co}_3) = 2$ by Lemma \ref{l:solspor1}. Therefore, Lemma \ref{l:easy} implies that a soluble triple is regular if it has a component contained in one of the maximal subgroups ${\rm Co}_2$ or ${\rm Co}_3$ of $G$. We can now use \textsf{GAP}, as in the previous case, to determine all the maximal triples $\s = (H_1,H_2,H_3)$, up to ordering and conjugacy, such that $\what{Q}(G,\s) \geqs 1$ and $H_i \ne {\rm Co}_2, {\rm Co}_3$ for all $i$. We find that there are just $12$ such triples, up to conjugacy and reordering, which only involve subgroups from $6$ distinct conjugacy classes of maximal subgroups of $G$. 

At this point, we switch to {\sc Magma} and we construct $G$ and each of the $6$ relevant maximal subgroups inside $S_n$, where $n = |G:{\rm Co}_2| = 98280$. Using random search, we find that $10$ of the $12$ relevant maximal triples are regular and it just remains to consider the maximal triples $(H_1,H_2,H_3)$, where
\[
H_1 = H_2 = 3.{\rm Suz}{:}2, \;\; H_3 = \mbox{$3.{\rm Suz}{:}2$ or $2^{11}{:}{\rm M}_{24}$}.
\]
We now complete the argument by proceeding as in the proof of Lemma \ref{l:solspor2}, working with the subgroup collections $\mathcal{S}_3(H_i)$ and random search to show that each triple of the form $(J_1,H_2,H_3)$ with $J_1 \in \mathcal{S}_3(H_1)$ is regular. In this way, we conclude that every soluble triple of subgroups of $G$ is regular and thus $R_{\rm sol}(G) \leqs 3$.

\vs

Finally, let us assume $G = {\rm Fi}_{24}$. Since we have already established the bound  $R_{{\rm sol}}(T) \leqs 3$, Lemma \ref{l:easy} implies that every soluble triple containing a subgroup of $T$ is regular. 

We begin by using \textsf{GAP} to determine all the core-free maximal triples $\s$ with $\what{Q}(G,\s) \geqs 1$. Switching to {\sc Magma}, we then use the Web Atlas \cite{WebAt} to construct $G$ and each relevant maximal subgroup $H_i$ as subgroups of $S_n$, where $n = 306936$. By applying random search, we can then reduce the candidate non-regular core-free maximal triples $(H_1,H_2,H_3)$ to a specific list of $18$ possibilities (up to conjugacy and reordering), all of which include the component $H_1 = {\rm Fi}_{23} \times 2$. Next we construct the set $\mathcal{S}_3(H_1)$ defined in Definition \ref{d:SH}. By removing any subgroups that are contained in $T$, we obtain a subset $\mathcal{S}_3(H_1)'$ of size $103$. Then by applying random search, we can show that every triple of the form $(J_1,H_2,H_3)$ with $J_1 \in \mathcal{S}_3(H_1)'$ is regular, unless $H_2=H_1$ and $H_3$ is one of the following:
\[
{\rm Fi}_{23} \times 2, \; (2 \times 2.{\rm Fi}_{22}){:}2, \; S_3 \times {\rm P\O}_{8}^{+}(3){:}S_3, \; {\rm O}_{10}^{-}(2).
\]
Finally, we handle these cases by showing that every triple of the form $(J_1,J_2,H_3)$ is regular (via random search), where $J_i \in \mathcal{S}_3(H_1)'$ for $i=1,2$.
\end{proof}

To complete the proof of Proposition \ref{p:nonregsol}, we may assume $G = \mathbb{B}$ is the Baby Monster.

\begin{lem}\label{l:solspor4}
The conclusion to Proposition \ref{p:nonregsol} holds for the groups with $T \in \mathcal{B}_4$.
\end{lem}

\begin{proof}
Here $G = \mathbb{B}$ and we begin by considering a maximal triple $\s$. As noted in the proof of Lemma \ref{l:rG4}, we can use the \textsf{GAP} Character Table Library \cite{GAPCTL} to show that $\what{Q}(G,\s)< 1$ unless $\s$ is of the form $(1,1,i)$ or $(1,2,2)$, where $i \in \{1,2,3,4,6,7\}$ and the relevant subgroups are labelled as in \eqref{e:bmax}. We immediately deduce that if $\tau$ is a soluble triple and none of the components are contained in a maximal subgroup $2.{}^2E_6(2){:}2$, then $\tau$ is regular. 

For the remainder of the proof, set 
\[
L = 2.{}^2E_6(2){:}2, \;\; Z = Z(L) = \la z \ra,\;\; \bar{L} = L/Z = {}^2E_6(2){:}2.
\]
As in \cite[Section 4.1]{B23}, write $\bar{L}' = {}^2E_6(2) = (X_{\s})'$, where $X = E_6$ is a simple algebraic group of adjoint type over an algebraically closed field of even characteristic and $\s$ is a Steinberg endomorphism of $X$ such that 
\[
X_{\s} = \{ x \in X \,:\, x^{\s} =x \} = {\rm Inndiag}({}^2E_6(2)) = {}^2E_6(2){:}3.
\]

Let $\mathcal{M}_0$ be a set of representatives of the conjugacy classes of maximal subgroups of $G$, excluding $L$, and define 
\[
\mathcal{M} = \mathcal{M}_1 \cup \mathcal{M}_2,
\]
where
\begin{align*}
\mathcal{M}_1 & = \{ \mbox{maximal subgroups of $L$, other than $2.{}^2E_6(2)$}   \} \\
\mathcal{M}_2 & = \{ \mbox{maximal subgroups of $2.{}^2E_6(2) < L$}  \}
\end{align*}
Note that every soluble subgroup of $L$ is contained in a subgroup $K \in \mathcal{M}$. Then in view of the above \textsf{GAP} calculation, it suffices to prove the following claim:

\vs

\noindent \textbf{Claim.} \emph{Let $\s = (H_1,H_2,H_3)$ be a triple of subgroups of $G$, where $H_1 \in \mathcal{M}$ and $H_2,H_3 \in \mathcal{M}_0 \cup \mathcal{M}$. Then $\what{Q}(G,\s) < 1$.}

\vs

To prove the claim, we proceed as in the proof of Lemma \ref{l:rG4}, where we showed that  $R(\mathbb{M}) = 3$. First observe that we can use the \textsf{GAP} Character Table Library \cite{GAPCTL} to compute ${\rm fpr}(x,G/H)$ for every maximal subgroup $H$ and every element $x \in G$. Let $x_1, \ldots, x_t$ be a complete set of representatives of the conjugacy classes in $G$ of elements of prime order and set $a_i = |x_i^G|$ and
\begin{equation}\label{e:bii}
b_i = \max\{ {\rm fpr}(x_i,G/H) \,:\, H \in \mathcal{M}_0\}
\end{equation}
for all $1 \leqs i \leqs t$. It follows immediately that the bound in \eqref{e:sporbd} holds for every maximal triple $\tau = (H_1,H_2,H_3)$ with $H_i \in \mathcal{M}_0$ for all $i$. Moreover, we check that this upper bound yields $\what{Q}(G,\tau) < 1$ and thus $\tau$ is regular. So in order to prove the claim, it suffices to show that 
\begin{equation}\label{e:fbi}
{\rm fpr}(x_i,G/H) \leqs b_i
\end{equation}
for all $1 \leqs i \leqs t$ and all $H \in \mathcal{M}$. For the reader's convenience, the values of $b_i$ are recorded in Table \ref{tab:bi}, with respect to the standard labelling of the conjugacy classes in $G$ (see \cite{Atlas}).

{\scriptsize
\begin{table}
\[
\begin{array}{lll} \hline
i & x_i & b_i \\ \hline
1 & \texttt{2A} &  
793/2950425 \\
2 & \texttt{2B} &  
269689951/11707448673375 \\
3 & \texttt{2C} &  
13/33634845 \\
4 & \texttt{2D} &  
303281/780496578225 \\
5 & \texttt{3A} &  
11/105887475 \\
6 & \texttt{3B} &  
2/5353200125 \\
7 & \texttt{5A} &  
11/22299902235 \\
8 & \texttt{5B} & 
2/93659589387 \\
9 & \texttt{7A} &  
1/37166503725 \\
10 & \texttt{11A} &  
1/780496578225 \\
11 & \texttt{13A} &  
1/126996316160000 \\
12 & \texttt{17A} &  
1/253992632320000 \\
13 & \texttt{19A} &  
1/22892381208576000 \\
14 & \texttt{23A} &  
1/11707448673375 \\
15 & \texttt{23B} &  
1/11707448673375 \\
16 & \texttt{31A} &  
1/45784762417152000 \\
17 & \texttt{31B} &  
1/45784762417152000 \\
18 & \texttt{47A} &  
  1/3843461129719173164826624000000 \\
19 & \texttt{47B} &  
  1/3843461129719173164826624000000 \\ \hline
\end{array}
\]
\caption{The $b_i$ values in \eqref{e:bii}}
\label{tab:bi}
\end{table}
}

Fix a subgroup $H \in \mathcal{M}_2$, so $H$ is a maximal subgroup of $J = 2.{}^2E_6(2)$. Then $Z \leqs H$ and thus $\bar{H} = H/Z$ is a maximal subgroup of the simple group $\bar{J} = J/Z = {}^2E_6(2)$. The possibilities for $\bar{H}$ (up to conjugacy) have been determined by Wilson \cite{Wilson}, which confirms that the list of maximal subgroups presented in the {\sc Atlas} \cite{Atlas} is complete (also see \cite{Craven}). We observe that 
$\bar{H} = \bar{K} \cap \bar{J}$ for some core-free maximal subgroup $\bar{K} = K/Z$ of $\bar{L} = L/Z = {}^2E_6(2){:}2$, whence $H \leqs K$ and $K$ is contained in $\mathcal{M}_1$. Therefore, we only need to prove that \eqref{e:fbi} holds for each subgroup $H \in \mathcal{M}_1$.

Suppose $H \in \mathcal{M}_1$, so $H \ne 2.{}^2E_6(2)$ is a maximal subgroup of $L = 2.{}^2E_6(2){:}2$. Then $Z \leqs H$ and $\bar{H} = H/Z$ is a maximal subgroup of the almost simple group $\bar{L}$. As above, the possibilities for $\bar{H}$ have been determined by Wilson \cite{Wilson} and we deduce that $\bar{H}$ is one of the following:
\[
\begin{array}{rl}
\mbox{Parabolic:} & P_{1,6}, \, P_2, \, P_{3,5}, \, P_4 \\
\mbox{Algebraic:} & {\rm O}_{10}^{-}(2), \, S_3 \times {\rm U}_{6}(2){:}2, \, S_3 \times \O_{8}^{+}(2){:}S_3, \, {\rm U}_{3}(8){:}6, \, ({\rm L}_{3}(2) \times {\rm L}_{3}(4){:}2).2,  \\
&  3^{1+6}{:}2^{3+6}{:}3^2{:}2^2,\, {\rm U}_{3}(2){:}2 \times G_2(2), \, F_4(2) \times 2 \\
\mbox{Almost simple:} & {\rm SO}_{7}(3), \, {\rm Fi}_{22}{:}2
\end{array}
\]
Here we adopt the standard notation for the maximal parabolic subgroups of $\bar{L}$, which corresponds to the usual labelling of the nodes of the Dynkin diagram of type $E_6$.  
The algebraic subgroups are of the form $N_{\bar{L}}(Y_{\s})$, where $Y$ is a positive dimensional non-parabolic $\s$-stable closed subgroup of the ambient algebraic group $X$. For example, if $\bar{H} = 3^{1+6}{:}2^{3+6}{:}3^2{:}2^2$ then the connected component of $Y$ is of type $A_2^3$ and we will refer to ${\rm SU}_3(2)^3$ as the \emph{type} of $\bar{H}$. We extend this usage of type to the other algebraic subgroups, which provides an approximate description of the given subgroup's structure. 

Let $\pi(\bar{H})$ be the set of prime divisors of $|\bar{H}|$ and recall that our goal is to verify the bound in \eqref{e:fbi} for all $i$. Let $i_r(H)$ be the number of elements of order $r$ in $H$ and observe that
\[
i_2(H) = 2i_2(\bar{H})+1, \;\; i_r(H) = i_r(\bar{H})
\]
for every odd prime $r$. We now consider each possibility for $\bar{H}$ in turn. 

\vs

\noindent \emph{Case (a).} $\bar{H} = ({\rm L}_{3}(2) \times {\rm L}_{3}(4){:}2).2$

\vs

Let $x \in G$ be an element of prime order $r \in \pi(\bar{H}) = \{2,3,5,7\}$ and note that $|H| = 2|\bar{H}| = 27095040$. It is easy to check that the trivial bound $|x^G \cap H| \leqs |H|$ is sufficient unless $x \in \texttt{2A}$. Since $\bar{H} \leqs A$, where $A = {\rm Aut}({\rm L}_3(2) \times {\rm L}_3(4))$, we deduce that $i_2(\bar{H}) \leqs i_2(A) = 98199$ and thus 
\[
|x^G \cap H| \leqs i_2(H) \leqs 196399
\]
for $x \in \texttt{2A}$. This bound is sufficient.

\vs

\noindent \emph{Case (b).} $\bar{H} = 3^{1+6}{:}2^{3+6}{:}3^2{:}2^2$

\vs

Here $\pi(\bar{H}) = \{2,3\}$ and once again we find that the trivial bound $|x^G \cap H| \leqs |H|$ is good enough unless $x \in \texttt{2A}$. So let us assume $x \in \texttt{2A}$. Then by inspecting the proof of \cite[Lemma 4.12]{EP}, noting that $i_2({\rm SU}_3(2)) = 9$, we deduce that $i_2(\bar{H}) \leqs \a+\b$, where
\begin{align*}
\a & = \binom{3}{1}9 + \binom{3}{1}9^2 + 9^3 + \binom{3}{2}|{\rm SU}_3(2)|\cdot \frac{|{\rm SU}_3(2)|}{|{\rm SL}_2(2)|} \\
\b & = 3(1+9)|{\rm SU}_3(2)| + \left(\frac{|{\rm SU}_3(2)|}{|{\rm SL}_2(2)|}\right)^3 
\end{align*}
This yields $|x^G \cap H| \leqs i_2(H) \leqs 154927$ and the desired bound follows.

\vs

\noindent \emph{Case (c).} $\bar{H} = S_3 \times {\rm U}_{6}(2){:}2$, $S_3 \times \O_{8}^{+}(2){:}S_3$, ${\rm U}_{3}(8){:}6$ or ${\rm U}_{3}(2){:}2 \times G_2(2)$

\vs

In each of these cases, with the aid of {\sc Magma}, it is easy to compute $i_r(H)$ for all $r \in \pi(H)$ and we find that the bound $|x^G \cap H| \leqs i_r(H)$ is sufficient unless $\bar{H} = S_3 \times {\rm U}_{6}(2){:}2$ and $x \in \texttt{2A}$. In the latter case, the proof of \cite[Lemma 4.6]{B23} gives $|x^G \cap H| \leqs 7033$ and this bound is good enough.

\vs

\noindent \emph{Case (d).} $\bar{H} = {\rm O}_{10}^{-}(2)$, ${\rm SO}_7(3)$, ${\rm Fi}_{22}{:}2$ or $F_4(2) \times 2$

\vs

Let $x \in G$ be an element of prime order $r$. If $r=2$ then $|x^G \cap H|$ is recorded in \cite[Table 5]{B23} and the desired bound quickly follows (see Remark \ref{r:corr}). Now assume $r$ is odd. Here we can compute $i_r(H)$ from the character table of $\bar{H}$, which is available in \cite{GAPCTL}, and in every case it is easy to check that the bound $|x^G \cap H| \leqs i_r(H)$ is sufficient.

\vs

\noindent \emph{Case (e).} $\bar{H} = P_{1,6}$, $P_2$, $P_{3,5}$ or $P_4$

\vs

To complete the proof, we may assume $\bar{H}$ is a maximal parabolic subgroup of $\bar{L}$. Note that the precise structure of $\bar{H}$ is recorded in \cite[Table 6]{B23}. In addition, let us observe that $\pi(H) = \{2,3,5,7\} \cup \mathcal{P}$, where $\mathcal{P} = \{17\}$ for $P_1$, $\mathcal{P} = \{11\}$ for $P_{1,6}$, and $\mathcal{P}$ is empty for $P_{3,5}$ and $P_4$. Let $x \in G$ be an element of order $r \in \pi(H)$. 

If $r = 2$ or $3$, then $|x^G \cap H|$ is recorded in \cite[Tables 5 and 6]{B23} and it is easy to check that \eqref{e:fbi} holds for all relevant $i$. Now assume $r \in \{5,7\} \cup \mathcal{P}$. Here one can check that the trivial bound
\[
|x^G \cap H| \leqs |H| \leqs 2|P_2| = 2^{23}|{\rm U}_6(2)|
\]
is sufficient for $r \geqs 7$, so we may assume $r=5$ and thus $x \in \texttt{5A}$ or $\texttt{5B}$. Now $L$ has a unique conjugacy class of elements of order $5$ and the fusion map from $L$-classes to $G$-classes shows that this class is contained in the $G$-class labelled $\texttt{5A}$. So we may assume $x \in \texttt{5A}$ and thus
\[
|x^G \cap H| = i_5(H) = i_5(\bar{H}).
\]
In addition, let us observe that if $\bar{H} = P_{3,5}$ or $P_4$, then $|H| \leqs 2|P_4| = 2^{32}3|{\rm L}_3(4)|$ and the crude bound $|x^G \cap H| \leqs 2|P_4|$ is good enough. Therefore, we are free to assume that $\bar{H} = P_{1,6}$ or $P_2$. 

Let $x \in H$ be an element of order $5$ and let $\bar{x}$ be the image of $x$ in $\bar{H} = H/Z$. Then
\[
{\rm fpr}(\bar{x},\bar{L}/\bar{H}) = \frac{i_5(\bar{H})}{|\bar{x}^{\bar{L}}|}
\]
and by applying \cite[Theorem 2]{LLS}, noting that $\bar{x} \in \bar{L}$ is semisimple, we deduce that 
\[
|x^G \cap H| = i_5(\bar{H}) \leqs |\bar{x}^{\bar{L}}| \cdot \frac{1}{3\cdot 2^6}.
\]
Now
\[
|\bar{x}^{\bar{L}}| = 759250790513639424,\;\; |x^G| = 9367743238695946498867200
\]
and we conclude that
\[
{\rm fpr}(x,G/H) \leqs \frac{3954431200591872}{9367743238695946498867200}  < 
\frac{11}{22299902235} = b_7,
\]
as required.
\end{proof}

This completes the proof of Theorem \ref{t:main2}.

\vs

\begin{rem}\label{r:corr}
In the proof of Lemma \ref{l:solspor4} we appealed to the information in \cite[Table 5]{B23}. Here we take the opportunity to correct three inaccuracies in this table:

\vspace{1mm}

\begin{itemize}\addtolength{\itemsep}{0.3\baselineskip}
\item[{\rm (a)}] For $\bar{H} = {\rm O}_{10}^{-}(2)$ we have $|\texttt{2D} \cap H| = 79942500$, rather than $79943000$.
\item[{\rm (b)}] And for $\bar{H} = F_4(2) \times 2$, the stated sizes of $\texttt{2A} \cap H$ and $\texttt{2C} \cap H$ are out by $1$; the correct values are $139232$ and $355284576$, respectively. 
\end{itemize}
\end{rem}

We finish by presenting Proposition \ref{p:spor_maxsol} below, which gives the exact value of $R_{\rm sol \, max}(G)$ for every almost simple sporadic group $G$ with a soluble maximal subgroup. By inspecting \cite{Wilson_survey}, for example, it is easy to see that $G$ has such a subgroup unless $G = {\rm M}_{22}$, ${\rm M}_{22}.2$, ${\rm M}_{24}$ or ${\rm HS}$.

\begin{prop}\label{p:spor_maxsol}
Let $G$ be an almost simple sporadic group with socle $T$ and assume $G$ has a soluble maximal subgroup. Then
\[
R_{\rm sol\, max}(G) = \left\{ \begin{array}{ll}
3 & \mbox{if $T = {\rm M}_{11}$, ${\rm M}_{12}$, ${\rm J}_2$, ${\rm Fi}_{22}$ or ${\rm Fi}_{23}$} \\
2 & \mbox{otherwise.}
\end{array}\right.
\]
Moreover, if $R_{\rm sol\, max}(G) = 3$ then every soluble maximal non-regular pair $(H,K)$ is recorded in Table \ref{table-sporadics}, up to conjugacy and ordering.
\end{prop}

\begin{proof}
As noted above, the assumption that $G$ has a soluble maximal subgroup means that $G \not \in \{{\rm M}_{22}, {\rm M}_{22}.2, {\rm M}_{24}, {\rm HS} \}$. In addition, if  
\[
G \in \{ {\rm M}_{23}, {\rm HS}.2, {\rm McL}, {\rm Suz}, {\rm Suz}.2, {\rm Ru}, {\rm Co}_3, {\rm Co}_2, {\rm O'N}, {\rm Fi}_{22}, {\rm Fi}_{22}.2, {\rm Fi}_{23}, {\rm Fi}_{24}'\},
\]
then $G$ has a unique class of soluble maximal subgroups, so every soluble maximal pair is conjugate and the result follows from \cite{BOW}. 

Next assume $G \in \{{\rm M}_{11}, {\rm M}_{12}, {\rm M}_{12}.2, {\rm J}_2, {\rm J}_2.2\}$ is one of the remaining groups appearing in Table \ref{table-sporadics}. In each of these cases, an entirely straightforward {\sc Magma} computation shows that a soluble maximal pair $(H,K)$ is non-regular if and only if it is one of the cases recorded in the table (up to conjugacy and ordering).

In each of the remaining cases, we can use the \textsf{GAP} Character Table Library \cite{GAPCTL} to compute $\what{Q}(G,\tau)$ precisely for every soluble maximal pair $\tau = (H,K)$. In this way, we deduce that $\what{Q}(G,\tau) < 1$, which implies that $\tau$ is regular, unless $G ={\rm He}.2$ and $H = K = 2^{4+4}.(S_3 \times S_3).2$, or $G = {\rm J}_3.2$ and $H = K = 2^{2+4}{:}(S_3 \times S_3)$. In both cases, the main theorem of \cite{BOW} implies that $\tau$ is regular.
\end{proof}

{\scriptsize
\begin{table}
\[
\begin{array}{lll} \hline
G & H & K \\ \hline
{\rm M}_{11} & {\rm U}_{3}(2){:}2 & {\rm U}_{3}(2){:}2, \, 2S_4 \\ 
{\rm M}_{12} & 3^2{:}2S_4 & 3^2{:}2S_4, \, 2^{1+ 4}{:}S_3, \, 4^2{:}{\rm D}_{12} \\ 
& 2^{1 + 4}{:}S_3 & 2^{1 + 4}{:}S_3,\, 4^2{:}{\rm D}_{12} \\  
& 4^2{:}{\rm D}_{12} & 4^2{:}{\rm D}_{12} \\
{\rm M}_{12}.2 & 2^{1 + 4}{:}S_3.2 & 2^{1 + 4}{:}S_3.2, \, 4^2{:}{\rm D}_{12}.2, \, 3^{1 + 2}{:}{\rm D}_8 \\
& 4^2{:}{\rm D}_{12}.2 & 4^2{:}{\rm D}_{12}.2 \\
 & 3^{1 + 2}{:}{\rm D}_8 & 3^{1 + 2}{:}{\rm D}_8 \\
{\rm J}_2& 2^{2 + 4}{:}(3 \times S_3) &  2^{2 + 4}{:}(3 \times S_3),\, 5^2{:}{\rm D}_{12} \\
{\rm J}_2.2 & 2^{2 + 4}{:}(3 \times S_3).2  & 2^{2 + 4}{:}(3 \times S_3).2,\, 5^2{:}(4 \times S_3)  \\
{\rm Fi}_{22} &3^{1 + 6}{:}2^{3 + 4}{:}3^2{:}2& 3^{1 + 6}{:}2^{3 + 4}{:}3^2{:}2 \\
{\rm Fi}_{22}.2 & 3^{1 + 6}{:}2^{3 + 4}{:}3^2.2.2& 3^{1 + 6}{:}2^{3 + 4}{:}3^2.2.2 \\
{\rm Fi}_{23}& 3^{1 + 8}.2^{1 + 6}.3^{1 + 2}.2S_4& 3^{1 + 8}.2^{1 + 6}.3^{1 + 2}.2S_4 \\ \hline
\end{array}
\]
\caption{The pairs $(H,K)$ in Proposition \ref{p:spor_maxsol}}
\label{table-sporadics}
\end{table}
}

\begin{rem}
The group ${\rm M}_{12}$ has two distinct classes of soluble maximal subgroups isomorphic to $3^2{:}2S_4$; in Table \ref{table-sporadics}, we write $3^2{:}2S_4$ for a  representative of either one of these classes.
\end{rem}

\begin{rem}\label{r:sol_nip}
Zenkov's main theorem in \cite{Zen20} states that every nilpotent pair of subgroups in an almost simple sporadic group is regular. Here we briefly comment on the existence (or otherwise) of non-regular pairs $(H,K)$, where $H$ is nilpotent and $K$ is soluble.

\vspace{1mm}

\begin{itemize}\addtolength{\itemsep}{0.3\baselineskip}
\item[{\rm (a)}] Clearly, if $R_{\rm sol}(G) = 2$, then every nilpotent-soluble pair is regular, so we may assume $G \in \mathcal{A}_2 \cup \mathcal{A}_3$ as defined in the statement of Proposition \ref{p:nonregsol}.

\item[{\rm (b)}] If $G = {\rm M}_{22}.2$ then there are precisely $7$ non-regular pairs $(H,K)$, up to ordering and conjugation, where $H$ is nilpotent and $K$ is soluble. For example, we can take $H$ to be a Sylow $2$-subgroup and $K$ to be the following second maximal subgroup
\[
K = 2^4{:}(S_3 \wr S_2) < 2^4{:}S_6 < G.
\]
For each such pair $(H,K)$ we find that $K$ has Fitting length $3$, which is the maximal Fitting length of all soluble subgroups of $G$.

\item[{\rm (c)}] Similarly, $G = {\rm J}_2.2$ has a unique non-regular pair $(H,K)$ of this form. Here $H$ is a Sylow $2$-subgroup and $K = 2^{2+4}{:}(3 \times S_3).2$ is a maximal subgroup of $G$. Once again, the Fitting length of $K$ is $3$ and one checks this is the maximal Fitting length of all soluble subgroups of $G$.
\end{itemize}
\end{rem}

We have checked computationally that the examples in (b) and (c) are the only non-regular nilpotent-soluble pairs associated with any of the groups in $\mathcal{A}_2$. We are not aware of any additional examples for the groups in $\mathcal{A}_3$.

\end{document}